\documentclass[11pt,reqno]{amsart}
\usepackage[colorlinks=true,allcolors=blue,backref=page]{hyperref}
\usepackage{color}
\usepackage{amsmath, amssymb, amsthm}

\usepackage{mathrsfs}
\usepackage{mathtools}
\usepackage[noabbrev,capitalize,nameinlink]{cleveref}
\crefname{equation}{}{}
\usepackage{fullpage}
\usepackage[noadjust]{cite}
\usepackage{graphics}
\usepackage{pifont}
\usepackage{tikz}
\usepackage{bbm}
\usepackage[T1]{fontenc}

\usetikzlibrary{arrows.meta}

\usepackage{environ}
\usepackage{framed}
\usepackage{thm-restate}
\usepackage{url}
\usepackage[linesnumbered,ruled,vlined]{algorithm2e}
\usepackage[noend]{algpseudocode}
\usepackage[labelfont=bf]{caption}
\usepackage{cite}
\usepackage{framed}
\usepackage[framemethod=tikz]{mdframed}
\usepackage{appendix}
\usepackage{graphicx}
\usepackage[textsize=tiny]{todonotes}
\usepackage{tcolorbox}
\usepackage{enumerate}
\usepackage[shortlabels]{enumitem}
\allowdisplaybreaks[1]

\apptocmd{\sloppy}{\hbadness 10000\relax}{}{} 

\crefname{algocf}{Algorithm}{Algorithms}

\crefname{equation}{}{} 
\crefname{conjecture}{Conjecture}{Conjectures} 
\AtBeginEnvironment{appendices}{\crefalias{section}{appendix}} 

\crefformat{enumi}{#2#1#3}
\crefrangeformat{enumi}{#3#1#4 to~#5#2#6}
\crefmultiformat{enumi}{#2#1#3}%
{ and~#2#1#3}{, #2#1#3}{ and~#2#1#3}

\usepackage[color,final]{showkeys} 

\colorlet{refkey}{orange!20}
\colorlet{labelkey}{blue!30}

\crefname{algocf}{Algorithm}{Algorithms}

\numberwithin{equation}{section}
\newtheorem{theorem}{Theorem}[section]

\newtheorem{lemma}[theorem]{Lemma}
\newtheorem{claim}[theorem]{Claim}
\crefname{claim}{Claim}{Claims}

\newtheorem{corollary}[theorem]{Corollary}

\newtheorem*{question*}{Question}
\newtheorem{fact}[theorem]{Fact}

\theoremstyle{definition}
\newtheorem{definition}[theorem]{Definition}

\newtheorem*{definition*}{Definition}

\theoremstyle{remark}
\newtheorem*{remark}{Remark}

\newcommand{\norm}[1]{\left\lVert#1\right\rVert}
\newcommand{\snorm}[1]{\lVert#1\rVert}

\newcommand{\ceil}[1]{\left\lceil #1 \right\rceil}

\newcommand{\E}{\mathop{\mathbb{E}}}

\newcommand{\mb}{\mathbb}
\newcommand{\mbf}{\mathbf}
\newcommand{\mbm}{\mathbbm}
\newcommand{\mc}{\mathcal}

\newcommand{\on}{\operatorname}
\newcommand{\wh}{\widehat}
\newcommand{\wt}{\widetilde}

\newcommand{\eps}{\varepsilon}
\newcommand{\imod}[1]{~\mathrm{mod}~#1}

\renewcommand{\bar}{\overline}

\newcommand{\rank}{\mathsf{rank}}
\newcommand{\supp}{\mathsf{supp}}

\newcommand{\wte}{\widetilde{\eps}}

\newif\ifshowcomments
\showcommentstrue

\allowdisplaybreaks

\title{Quasipolynomial bounds for the corners theorem}

\author[A1]{Michael Jaber}
\address{Department of Computer Science, University of Texas, Austin, TX 78712}
\email{mjjaber@cs.utexas.edu}

\author[A2]{Yang P. Liu}
\address{Computer Science Department, Carnegie Mellon University, Pittsburgh, PA 15213}
\email{yangl7@andrew.cmu.edu}

\author[A3]{Shachar Lovett}
\address{Department of Computer Science and Engineering, University of California, San Diego, CA 92093}
\email{slovett@ucsd.edu}

\author[A4]{Anthony Ostuni}
\address{Department of Computer Science and Engineering, University of California, San Diego, CA 92093}
\email{aostuni@ucsd.edu}

\author[A5]{Mehtaab Sawhney}
\address{Department of Mathematics, Columbia University, New York, NY 10027}
\email{m.sawhney@columbia.edu}
\begin{document}

\begin{abstract}
Let $G$ be a finite abelian group and $A$ be a subset of $G \times G$ which is corner--free, meaning that there are no $x, y \in G$ and $d \in G \setminus \{0\}$ such that $(x, y)$, $(x+d, y)$, $(x, y+d) \in A$. We prove that
\[|A| \le |G|^2 \cdot \exp(-(\log |G|)^{\Omega(1)}).\]
As a consequence, we obtain polynomial (in the input length) lower bounds on the nondeterministic communication complexity of \textsf{Exactly-N} in the 3-player Number-on-Forehead model. We also obtain the first ``reasonable'' lower bounds on the coloring version of the $3$-dimensional corners problem, as well as on the nondeterministic communication complexity of \textsf{Exactly-N} in the 4-player Number-on-Forehead model.
\end{abstract}

\maketitle
\setcounter{tocdepth}{1}
\tableofcontents

\newpage 

\section{Introduction}

The cornerstone of additive combinatorics is Szemer\'{e}di's theorem \cite{Sze70,Sze75} which states that any dense subset of the integers contains arbitrarily long arithmetic progressions. More precisely, given any positive integer $k$ and $\delta>0$, there exists $N = N(k,\delta)$ such that if $A\subseteq [N]$ and $|A|\ge \delta N$ then $A$ contains a nontrivial $k$-term arithmetic progression. Equivalently, the size of the largest subset $A$ with no $k$-term arithmetic progression is at most $o_k(N)$.

Since Szemer\'{e}di's original work, there have been several different proofs of Szemer\'{e}di's result and various generalizations. One such generalization, the multidimensional Szemer\'{e}di's theorem of Furstenberg and Katznelson \cite{FK78}, extending the ergodic proof of Szemer\'{e}di's theorem of Furstenberg \cite{Fur77}, allows one to find arbitrary patterns in dense subsets of multidimensional grids. More precisely, given a finite pattern $\mc{P}$ in $d$ dimensions and $A\subseteq [N]^d$ which contains no homothetic copy of $\mc{P}$, then $|A| = o_{\mc{P}}(N^d)$. While the result of Furstenberg and Katznelson is a tremendous breakthrough, the use of ergodic theory prevents one from obtaining an explicit decay rate in the $o(N^d)$ bound. A quantitative decay rate was provided independently by works of Nagle, R\"odl, Schacht, and Skokan \cite{NRS06,RS04} and Gowers \cite{Gow07} due to versions of the hypergraph regularity lemma (with a subsequent alternate proof given by Tao \cite{Tao06}). However the quantitative decay rates provided by these works are exceedingly weak, and save at best powers of $\log^{\ast}(N)$ over the trivial bound. Here the iterated logarithm $\log^{\ast}(N)$ is the number of times one takes logarithms of $N$ to reach a number less than $e$.

The project of providing ``reasonable'' bounds (i.e., a density saving of at least a finite number of iterated logarithms) for the multidimensional Szemer\'{e}di's theorem has been reiterated several times by Gowers \cite{Gow98,Gow01}. This appears especially natural in light of seminal work of Gowers \cite{Gow98a,Gow01a} which proved the first reasonable decay rate for sets avoiding $k$-term arithmetic progressions. Precisely, Gowers proved that the densest subset of $[N]$ avoiding $k$-term arithmetic progressions has size $\ll N(\log\log N)^{-\Omega_k(1)}$. Despite a host of subsequent work in higher order Fourier analysis, progress towards obtaining effective bounds for the multidimensional Szemer\'{e}di's theorem has been limited.

In \cite{Gow01a}, Gowers asked whether one could obtain reasonable bounds for the \emph{corners problem}, which is the simplest case of the multidimensional Szemer\'{e}di's theorem: how large can $A \subseteq [N] \times [N]$ be while avoiding the pattern $(x, y)$, $(x+d, y)$, $(x, y+d)$ for $x, y, d \in \mb{Z}$ and $d \neq 0$? This was answered by Shkredov \cite{Shk05,Shk06}, who proved that if $A\subseteq [N] \times [N]$ avoids a corner then $|A| \le N^2(\log\log N)^{-\Omega(1)}$. Since then, the only additional pattern for which reasonable bounds are known is the $L$-shape $(x,y),~(x+d,y),~(x+2d,y),~(x, y+d)$ (over $\mb{F}_p^n$) due to (difficult) work of Peluse \cite{Pel24}.

We note that even considering just corner--free sets has a rich history. Before the work of Furstenberg and Katznelson, Ajtai and Szemer{\'e}di \cite{ASz74} were able to prove that corner--free subsets of $[N]\times [N]$ have size $o(N^2)$. This proof for instance inspired a portion of the combinatorial proof of the Density Hales--Jewett theorem of Polymath \cite{Pol12}; however due to needing to find growing length progressions via Szemer{\'e}di's theorem the proof gives an exceedingly weak bound. The well--known alternate proof of the corners theorem, due to Solymosi \cite{Sol03} via an application of the regularity/triangle removal lemma, also currently comes with weak bounds. 

The corners theorem is also known, via a standard projection argument (see \cite[Section~2.4]{Zha23}), to imply the case of Szemer{\'e}di's theorem for $3$-term arithmetic progressions, more commonly known as Roth's theorem \cite{Roth53}. Introducing the density-increment method, Roth proved that if $A\subseteq [N]$ does not contain a $3$-term arithmetic progression (henceforth $3$-AP) then $|A|\ll N(\log\log N)^{-1}$. Since then, there has been a long line of work devoted to improving this bound, including results by Heath--Brown \cite{HB87}, Szemer\'{e}di \cite{SZ90}, Bourgain \cite{Bou99,Bou08}, Sanders \cite{San11,San12}, Bloom \cite{Bloom16} and Bloom and Sisask \cite{BS21}; however these works at best gave a density savings of slightly better than a single logarithm. The best known lower bounds for $3$-AP free sets and corners are all derived from a construction of Behrend \cite{Beh46} which found sets of density $\ge e^{-O((\log N)^{1/2})}$ avoiding these patterns (see \cite{LS21, Gre21,Hun22} for improvements in the case of corners of the constant in the exponent and \cite{EHPS24} for improvements in the $3$-AP case). 

In a remarkable breakthrough work, Kelley and Meka \cite{KM23} recently nearly matched this bound, and prove that the densest $3$-AP avoiding set $A$ has size bounded by $Ne^{-\Omega((\log N)^{1/12})}$. The constant $1/12$ has been subsequently refined in work of Bloom and Sisask \cite{BS23b} to $1/9$. 

Given the work of Kelley and Meka \cite{KM23} and the relation between corners and $3$-AP free sets, many researchers have speculated whether the recent breakthrough of Kelley and Meka \cite{KM23} could be used to improve bounds for corner--free sets (see e.g.~\cite{mekatalk} and \cite{kelleytalk}). Peluse \cite[Problem~1.18]{Pel24b} even asked whether methods underlying Kelley--Meka could allow one to achieve a savings of one logarithm over Shkredov's bound. Preliminary work in this direction considered the easier problem of obtaining quasipolynomial bounds for so-called ``skew corners'' \cite{Mil24,JL024}; however the bounds of Shkredov remained unimproved. In this work, we provide quasipolynomial bounds for the corners theorem, achieving a doubly-exponential improvement over the work of Shkredov \cite{Shk06} and nearly matching Behrend's lower bound.

\begin{theorem}\label{thm:main}
There exists a constant $c>0$ such that the following holds. Let $(G, +)$ be a finite abelian group. Let $A \subseteq G \times G$ with no $x,y,d \in G$ with $d \neq 0$ such that $(x,y)$, $(x+d,y)$, $(x,y+d) \in A$. Then \[|A| \le |G|^{2} \cdot \exp(-c(\log |G|)^{1/600}).\]
\end{theorem}
By embedding $[N]$ into the cyclic group $\mb{Z}/(4N\mb{Z})$ one may obtain a similar conclusion for corner--free subsets of $[N]\times [N]$. Furthermore a trick of Green (see \cite[Lemma~6.2]{Shk06}) allows one to prove essentially identical bounds only avoiding corners with $d>0$. Our techniques appear unable to yield results approaching the $1/2$ exponent in Behrend's lower bound, so we do not attempt to optimize the constant $1/600$.

We remark that the corners problem has a direct connection to theoretical computer science. In particular, the seminal work of Chandra, Furst, and Lipton \cite{CFL83} which introduced the Number-on-Forehead (NOF) communication complexity model noted an equivalence between the complexity of computing the \textsf{Exactly-N} function for $3$ parties and bounds for the the corners problem. We elaborate on this connection and state the associated corollaries in \cref{sec:nof}.

Moreover, our results can be combined with a proof of Graham and Solymosi \cite{GS06} to provide improved bounds for the coloring variant of finding $3$-dimensional corners. The following is a consequence of \cref{thm:main}.
\begin{corollary}\label{cor:3d}
For a sufficiently small constant $c$, and any abelian group $G$ and coloring of $G \times G \times G$ with $c \log \log \log |G|$ colors, there are $x, y, z, d \in G$ with $d \neq 0$ such that $(x,y,z)$, $(x+d,y,z)$, $(x,y+d,z)$, and $(x,y,z+d)$ are all of the same color.
\end{corollary}
We remark that in order to obtain ``reasonable'' bounds in \cref{cor:3d}, it is crucial that we obtain a quasipolynomial bound in \cref{thm:main}; for example, an inverse logarithmic-type bound would give a tower type dependence in \cref{cor:3d}.

Following the influential survey of Green \cite{Gre05}, it has become commonplace to consider problems in additive combinatorics in the model setting of finite field vector spaces before considering general abelian groups. Our main result over $\mb{F}_2^n$ is the following bound for corners. 
\begin{theorem}\label{thm:mainff}
There exists a constant $c>0$ such that the following holds. Let $A \subseteq \mb{F}_2^n \times \mb{F}_2^n$ with no $x,y,d \in \mb{F}_2^n$ with $d \neq 0$ such that $(x,y)$, $(x+d,y)$, $(x,y+d) \in A$. Then for $N = 2^n$, it holds that \[|A| \le N^2 \cdot \exp(-c(\log N)^{1/178}).\]
\end{theorem}
We remark that the previous state of the art in the finite field setting was essentially of the same shape as work of Shkredov (due to Lacey and McClain \cite{LM07}). This is unlike the case of three--term arithmetic progressions and the resolution of the capset problem due to Ellenberg and Gijswijt \cite{EG17} based on the polynomial method of Croot, Lev and Pach \cite{CLP17}. Furthermore it is now understood that na\"{i}ve generalizations of the polynomial method technique are unlikely to work for corners \cite{CFTZ22}.

\subsection{Corners and progressions in nonabelian groups}
\label{subsec:nonabelian}

In this section we record an argument due to Fox (personal communication) which shows that \cref{thm:main} implies quasipolynomial bounds for BMZ \cite{BMZ97} and na\"{i}ve corners in all groups, not necessarily abelian. This also implies similar bounds for Roth's theorem in general groups.
\begin{corollary}
\label{cor:bmz}
There exists a constant $c > 0$ such that the following holds.
Let $G$ be a finite group and $A \subseteq G \times G$ with no $x,y,g \in G$ with $g \neq \mbf{1}_G$ such that $(x,y), (xg,y), (x,gy) \in A$ (i.e., a BMZ corner). Then
\[ |A| \le |G|^2 \cdot \exp(-c(\log |G|)^{1/1200}). \]
\end{corollary}
\begin{proof}
Let $H \subseteq G$ be an abelian subgroup of $G$ of size $|H| \ge 2^{\Omega(\sqrt{\log |G|})}$ -- such a subgroup exists by a result of Pyber \cite{Pyber97}. For $x, y \in G$ define the set
\[ A_{x,y} \coloneqq \{ (h_1, h_2) \in H \times H : (xh_1, h_2y) \in A\}. \]
We claim that $A_{x,y}$ has no corners. If it does, then there are $h_1, h_2, h_3$ such that
\[ (xh_1, h_2y), (xh_1h_3, h_2y), (xh_1, h_3h_2y) \in A, \] which implies that $A$ contains a BMZ corner.
By \cref{thm:main} we conclude that
\[ |A_{x,y}| \le |H|^2 \cdot \exp(-c(\log |H|)^{1/600}) \le |H|^2 \cdot \exp(-c(\log |G|)^{1/1200}). \]
By averaging, we know that
\[ |A| = \frac{|G|^2}{|H|^2} \cdot \E_{x,y \in G}\Big[|A_{x,y}| \Big] \le |G|^2 \cdot \exp(-c(\log |G|)^{1/1200}). \qedhere\]
\end{proof}

\begin{corollary}
\label{cor:naive}
There exists a constant $c > 0$ such that the following holds.
Let $G$ be a finite group and $A \subseteq G \times G$ with no $x,y,g \in G$ with $g \neq \mbf{1}_G$ such that $(x,y), (xg,y), (x,yg) \in A$ (i.e., a na\"{i}ve corner). Then
\[ |A| \le |G|^2 \cdot \exp(-c(\log |G|)^{1/1200}). \]
\end{corollary}
\begin{proof}
The proof is nearly identical to that of \cref{cor:bmz}, except one instead defines
\[ A_{x,y} \coloneqq \{ (h_1, h_2) \in H \times H : (xh_1, yh_2) \in A\}. \qedhere \]
\end{proof}

This implies quasipolynomial bounds for subsets $A \subseteq G$ avoiding nontrivial solutions to $xy = z^2$, which is an analog of the 3-AP question in nonabelian groups.

\begin{corollary}
\label{cor:roth}
There exists a constant $c > 0$ such that the following holds.
Let $G$ be a finite group and $A \subseteq G$ with no $x, y, z$ not all equal satisfying $xy = z^2$. Then
\[ |A| \le |G| \cdot \exp(-c(\log |G|)^{1/1200}). \]
\end{corollary}
\begin{proof}
Define $S = \{ (x, y) \in G \times G : x^{-1}y \in A \}.$ We claim that $S$ does not contain a corner of the form in \cref{cor:naive}, i.e., $(x, y)$, $(xg, y)$, $(x, yg) \in S$ for some $x, y, g \in G$ with $g \neq \mbf{1}_G$. Indeed, otherwise $x^{-1}y, (xg)^{-1}y, x^{-1}yg \in A$, and $x^{-1}yg (xg)^{-1}y = (x^{-1}y)^2$. Thus, \cref{cor:naive} implies that
\[ |A||G| = |S| \le |G|^2 \cdot \exp(-c(\log |G|)^{1/1200}) \]
as desired.
\end{proof}
The previous best bound for general groups was $|G|^2 \cdot (\log \log |G|)^{-1}$ due to Sanders \cite{Sanders17} (which is recovered by the argument of Fox plus Shkredov's corners bound). For the symmetric group $S_n$, the previous best bound was $N \cdot e^{-\Omega((\log \log N)^2)}$ for $N = n!$, and is due to Keevash and Lifshitz \cite{KL23}.

\subsection{Applications to Communication Complexity}
\label{sec:nof}

The work of Chandra, Furst, and Lipton \cite{CFL83} in 1983 introduced the Number-on-Forehead model of communication complexity to model interaction between parties with shared information. The $k$-NOF model is defined by $k$ players communicating over a shared channel in order to compute a function $f : (\{0,1\}^n)^k \to \{0,1\}$. Each player can see the $k-1$ inputs of every other player, but they cannot see their own. This model has a number of striking connections in theoretical computer science and combinatorics. For instance lower bounds for $k = \omega(\log n)$ players would imply breakthrough circuit lower bounds \cite{BNS89, NW93, Raz00, BH09}. The work of Chandra, Furst, and Lipton \cite{CFL83} was primarily concerned with the \textsf{Exactly-N} problem, which is now one of the most studied problems in the NOF model. In \textsf{Exactly-N} each player receives a number in $[N] \coloneqq \{1,2,\dots,N\}$ (given as $\lceil \log_2 N \rceil$ bits) and the players aim to check if these numbers sum up to $N$. The authors of \cite{CFL83} in fact observe an equivalence between \textsf{Exactly-N} (for three players) and the size of sets $S \subseteq [N]^2$ without corners. Thus, \cref{thm:main} implies the following corollary.
\begin{corollary}\label{corollary:3nof}
Any nondeterministic $3$-NOF protocol computing \textsf{Exactly-N} requires $\Omega((\log N)^{\Omega(1)})$ bits of communication.
\end{corollary}

Previously, despite the fact that optimal separations between randomized and deterministic NOF communication were known nonexplicitly \cite{BDPW10} for more than a decade, explicit constructions exhibiting strong separations for $3$-NOF have only been developed recently. In particular, Kelley, Lovett, and Meka \cite{KLM24}, building off Kelley-Meka's work on $3$-APs \cite{KM23}, exhibited an explicit 3-player $n$-bit boolean function which has a constant cost randomized protocol, but requires $\Omega(n^{1/3})$ bits of communication to compute nondeterministically. 
Even more recently, Kelley and Lyu refined the analysis to improve the lower bound to $\Omega(n^{1/2})$ \cite{KL25}.
However, these works did not apply to the \textsf{Exactly-N} problem.

The communication complexity translation also carries over to colorings and higher-dimensional corners (see e.g.~\cite[Appendix A]{HPSSS24}) to give an $\Omega(\log\log\log\log N)$ lower bound against deterministic $4$-NOF protocols for \textsf{Exactly-N}. 
Moreover, Beame, David, Pitassi, and Woelfel observed \cite[Theorem 3.4]{BDPW10} that any \emph{nondeterministic} $4$-NOF protocol for \textsf{Exactly-N}\footnote{Technically, their result is for functions of the form $f:X_1\times \cdots X_k \to \{0,1\}$ such that for all $(x_2, \dots, x_k) \in X_2\times \cdots X_k$, there exists exactly one $x_1 \in X_1$ such that $f(x_1, x_2, \dots, x_k) = 1$.
This is not true of \textsf{Exactly-N}, since if the latter $k-1$ players' inputs sum to at least $N$, there is no possible choice of the first player's input to make the function evaluate to 1. However, they may perform an initial $k$-bit ``input validation'' round to check for this case, and afterwards perform the protocol in \cite{BDPW10}.} which exchanges $b$ bits of communication can be converted into a \emph{deterministic} one which exchanges $b + O(1)$ bits.
Thus, the $\Omega(\log\log\log\log N)$ lower bound also applies to nondeterministic protocols.
\begin{corollary}
\label{corollary:4nof}
Any nondeterministic $4$-NOF protocol for \textsf{Exactly-N} uses $\Omega(\log\log\log\log N)$ bits of communication.
\end{corollary}

\subsection{Organization of the remainder of the paper}

The remainder of the paper is organized as follows. In \cref{sec:overview} we give a high-level overview of the main ideas in the proof. In \cref{sec:relative-sifting} we define what it means for a set to be combinatorially spread against rectangles, and prove a \emph{sifting} statement, i.e., that functions with a large grid norm have large density on a large rectangle. Critically, our sifting statement is relative, meaning that it works even when the sets we consider are subsets of a sparse pseudorandom object. In \cref{sec:add-tool} we introduce additional additive combinatorial tools needed for our analysis, such as spectral positivity. In \cref{sec:finite-field} we prove \cref{thm:mainff}, our corners bound for finite fields, which captures a number of the main conceptual ideas of the general abelian group case (pseudorandomization, density increment), but is technically simpler.

The next two sections of the paper are devoted to proving \cref{thm:main}, the corners bound over general abelian groups. We introduce Bohr sets in \cref{sec:bohr-prelim}, and then define and apply several pseudorandomness properties of subsets of Bohr sets. In addition, we describe how to perform the pseudorandomization procedure in the Bohr set setting. Finally, \cref{sec:dsbohr} uses a density increment argument to prove \cref{thm:main}. \cref{cor:3d}, our coloring lower bound for $3$-dimensional corners, is proven in \cref{sec:3dcorners}. 
The appendix contains a section which states and applies almost periodicity (\cref{sec:almost-period}).

\subsection{Acknowledgements}
MJ, SL, and AO thank Russell Impagliazzo and David Zuckerman for helpful conversations and Ilya Shkredov for answering a question about the current state-of-the-art. MJ would like to thank Freddie Manners for useful discussions, Sarah Peluse for her encouragement and collaboration on this question, as well as Amey Bhangale and Surya Teja Gavva for their collaboration on this question at the Simons Institute for the Theory of Computing. MS thanks Tim Gowers, Huy Pham, Ashwin Sah for useful discussions and Sarah Peluse for discussions regarding \cref{sec:3dcorners}. A portion of this work was conducted when MS visited ``New Frontiers in Extremal and Probabilistic Combinatorics'' at the SwissMAP Research Station.

MJ is supported by NSF Grant CCF-2312573 and a Simons Investigator Award (\#409864, David Zuckerman). Part of this research was conducted when YL was a Postdoctoral Member at the IAS, and is based upon work supported by the National Science Foundation under Grant No.~DMS-1926686. SL and AO are supported by the Simons Investigator Award \#929894 (Shachar Lovett) and NSF award CCF-2425349. This research was conducted during the period MS served as a Clay Research Fellow.

A preliminary version of this manuscript appeared in the 66th IEEE Symposium on Foundations of Computer Science (FOCS) 2025.

\section{Overview of the Proof}
\label{sec:overview}

We now provide an overview of the proof strategy for \cref{thm:mainff} and then discuss the modifications required for \cref{thm:main}.
Our result, like several results in additive combinatorics, is established via a density increment argument -- either the set $A$ contains the expected number of corners, or has higher density onto some structured sub-instance. We start by reviewing Shkredov's bound \cite{Shk06} for the corners problem before discussing the novel aspects of this work.

\subsection{Shkredov's corners bound} Consider a subset $A \subseteq \mb{F}_2^n \times \mb{F}_2^n$ with density $\alpha$, so $|A| = \alpha \cdot 4^n$. In this case, the expected density of corners is about $\alpha^3$ for a random set $A$.
The first step in Shkredov's corners bound is to prove that if $A$ has much fewer than this expected density of corners, then $A$ has a density increment onto a subrectangle. That is, there are subsets $X, Y \subseteq \mb{F}_2^n$ such that:
\begin{enumerate}
    \item $|X|, |Y| \ge \alpha^{O(1)} \cdot 2^n$, and
    \item $|A \cap (X \times Y)| \ge (\alpha + \alpha^{O(1)})|X||Y|$.
\end{enumerate}
This is proven via a Cauchy-Schwarz argument: if $A$ has few corners, then its balanced indicator function, i.e.~$\mbm{1}_A - \alpha$, has a large \emph{box norm}, which is the $(2,2)$-grid norm (defined later in \cref{def:grid}).
This suggests a natural density increment approach: maintain $X \times Y \subseteq \mb{F}_2^n \times \mb{F}_2^n$ such that the density of $A$ within $X \times Y$ goes up over time. To carry this out, one must establish a statement of the following form: if $A \subseteq X \times Y$ has few corners, then $A$ admits an $\alpha^{O(1)}$ density increment onto some $X' \times Y' \subseteq X \times Y$.

It turns out that proving such a statement requires an additional pseudorandomness assumption on $X$ and $Y$. In Shkredov's approach, the pseudorandomness assumption was that the indicator functions of $X$ and $Y$ have very small nontrivial Fourier coefficients -- smaller than the densities of $X$ and $Y$ themselves. Because this pseudorandomness guarantee may not hold for the $X$ and $Y$ that are incremented onto, Shkredov uses a pseudorandomization procedure to make $X$ and $Y$ satisfy this property. At a high level, the pseudorandomization procedure repeatedly passes to subspaces of $\mb{F}_2^n$ on which the Fourier coefficients of $X$ and $Y$ are large.

We now explain why this approach gets a bound of $N^2/(\log \log N)^c$ for corners, where $N = 2^n$, which is two logarithmic factors off the Behrend-type bounds of $N^2/2^{O((\log N)^c)}$. The first point to understand is what the densities of $X$ and $Y$ are at the end of the density increment procedure. The number of density increment steps is $\alpha^{-O(1)}$, each of which decreases the densities of $X$ and $Y$ by $\alpha^{O(1)}$, so the final densities are $\delta \coloneqq \exp(-1/\alpha^{O(1)})$. In the pseudorandomization procedure, guaranteeing that all Fourier coefficients are less than $\delta$ requires passing to a codimension $\delta^{-O(1)}$ subspace. Thus, we need that $\delta^{-O(1)} \le n$ so $\alpha = (\log n)^{-c} = (\log \log N)^{-c}$.

This discussion clarifies why Shkredov's bound is doubly logarithmic: one logarithmic factor stems from the density increment, and the other from pseudorandomization. The remainder of this overview is devoted to explaining our approach for avoiding \emph{both} these losses, and establishing a Behrend-type bound for the corners theorem.

\subsection{H\"{o}ldering and $XYD$ containers}\label{subsec:XYD-expo}
In this section we discuss how to avoid the first logarithmic loss in Shkredov's bound, by improving the density increment obtained from $\alpha^{O(1)}$ to a multiplicative $(1+\eps)$ for some absolute constant $\eps > 0$. This reduces the number of density increment steps down to $O(\log(1/\alpha))$. The starting point for understanding how to do this is by replacing the initial use of Cauchy--Schwarz in Shkredov \cite{Shk06} with the following H\"{o}lder manipulation inspired by \cite{KM23} (see \cite[1.~H\"{o}lder--lifting]{BS23}).

Let $A$ denote a corner--free set of density $\alpha$ (so $|A| = \alpha \cdot 4^n$) and $f_A = \mbm{1}_A - \alpha$. For an even positive integer $k$ we have that (after substituting $d$ to be $x+y+d$) 
\begin{align*}
\Big|\E_{x,y,d\in \mb{F}_2^{n}}&\mbm{1}_A(x,y)\mbm{1}_A(x,y+d)f_A(x+d,y)\Big|^{k}=\Big|\E_{x,y,d\in \mb{F}_2^{n}}\mbm{1}_A(x,y)\mbm{1}_A(x,x+d)f_A(y+d,y)\Big|^{k}\\
&\le (\E_{x,y\in \mb{F}_2^{n}}\mbm{1}_A(x,y))^{k-1} \cdot (\E_{x,y\in \mb{F}_2^{n}}\mbm{1}_A(x,y)\cdot (\E_{d\in \mb{F}_2^n}\mbm{1}_A(x,x+d)f_A(y+d,y))^{k})\\
&\le (\E_{x,y\in \mb{F}_2^{n}}\mbm{1}_A(x,y))^{k-1} \cdot (\E_{x,y\in \mb{F}_2^{n}}(\E_{d\in \mb{F}_2^n}\mbm{1}_A(x,x+d)f_A(y+d,y))^{k}).
\end{align*}
Via switching the order of summation and a further application of Cauchy--Schwarz, we have that 
\begin{align*}
\phantom{\le} &(\E_{x,y\in \mb{F}_2^{n}}(\E_{d\in \mb{F}_2^n}\mbm{1}_A(x,x+d)f_A(y+d,y))^{k})^2\\
\le &(\E_{x_1,x_2\in \mb{F}_2^n}(\E_{d\in \mb{F}_2^n}\mbm{1}_A(x_1,x_1+d)\mbm{1}_A(x_2,x_2+d))^k) \cdot (\E_{y_1,y_2\in \mb{F}_2^n}(\E_{d\in \mb{F}_2^n}f_A(y_1+d,y_1)f_A(y_2+d,y_2))^k).
\end{align*}
Taking $k \approx \log(1/\alpha)$, we find that either
\[\E_{x_1,x_2\in \mb{F}_2^n}(\E_{d\in \mb{F}_2^n}\mbm{1}_A(x_1,x_1+d)\mbm{1}_A(x_2,x_2+d))^k\ge (3\alpha)^{2k}\]
or 
\[\E_{y_1,y_2\in \mb{F}_2^n}(\E_{d\in \mb{F}_2^n}f_A(y_1+d, y_1)f_A(y_2+d, y_2))^k\ge (\alpha/9)^{2k}.\]
We want to express the second case as some property of $\mbm{1}_A$ more directly. This is done by using a graph theoretic analog of spectral positivity (see \cref{lem:spectral-pos}), as developed by Kelley, Lovett, and Meka \cite{KLM24}, which upgrades the dichotomy to be:
\begin{equation} \E_{x_1,x_2\in \mb{F}_2^n}(\E_{d\in \mb{F}_2^n}\mbm{1}_A(x_1,x_1+d)\mbm{1}_A(x_2,x_2+d))^k\ge (3\alpha)^{2k} \label{eq:xdgrid} \end{equation}
or 
\begin{equation} \E_{y_1,y_2\in \mb{F}_2^n}(\E_{d\in \mb{F}_2^n}\mbm{1}_A(y_1+d,y_1)\mbm{1}_A(y_2+d,y_2))^{k'}\ge ((1+1/1000)\alpha)^{2k'} \label{eq:ydgrid} \end{equation}
where $k' \ge k$ but still satisfies  $k' \approx \log(1/\alpha)$. Observe that in our analysis to this point we have had $A$ which lives in $\mb{F}_2^{n}\times \mb{F}_2^{n}$.

Equations \eqref{eq:xdgrid} and \eqref{eq:ydgrid} are best interpreted as conditions on certain \emph{grid norms}, as introduced in work of Kelley, Lovett, and Meka \cite{KLM24}.
\begin{definition}[Grid norms]\label{def:grid}
Fix positive integers $k,\ell\ge 1$ and finite sets $\Omega_1, \Omega_2$. Let $A:\Omega_1\times \Omega_2\to \mb{R}$. We define the \emph{$(k,\ell)$-grid norm} of $A$ as  
\[\snorm{A}_{G(k,\ell)} = \Big|\E_{\substack{x_1,\ldots,x_{k}\in \Omega_1\\y_1,\ldots,y_{\ell}\in \Omega_2}}\prod_{\substack{1\le i\le k\\1\le j\le \ell}}A(x_i,y_j)\Big|^{1/(k\ell)}.\]
\end{definition}
\begin{remark}
It is known that the $(k,\ell)$-grid norm is a seminorm if $k,\ell$ are both even positive integers \cite{Hat10}.
\end{remark}
Intuitively, the grid norm counts the number of copies of the complete bipartite graph $K_{k,\ell}$ in a ``graph'' given by a function $f: \Omega_1 \times \Omega_2 \to \mb{R}$. To interpret \eqref{eq:xdgrid} as a grid norm, one can define the function $F: \mb{F}_2^n \times \mb{F}_2^n \to [0, 1]$ as $F(x, d) \coloneqq \mbm{1}_A(x, x+d)$ and note that \eqref{eq:xdgrid} is exactly equivalent to $\|F\|_{G(2,k)} \ge 3\alpha$. Similarly, defining $G: \mb{F}_2^n \times \mb{F}_2^n \to [0, 1]$ as $G(y, d) \coloneqq \mbm{1}_A(y+d, y)$, \eqref{eq:ydgrid} is equivalent to $\|G\|_{G(2,k')} \ge (1+1/1000)\alpha$. A key result of \cite{KLM24} is a structural theorem regarding functions with large grid norm, which they refer to as \emph{sifting}. Informally, it says that if $\|F\|_{G(k,\ell)} \ge \tau$ for a nonnegative function $F$, then there is a subrectangle $\Omega_1' \times \Omega_2' \subseteq \Omega_1 \times \Omega_2$ on which $F$ has average value at least $(1-\eps)\tau$, and $|\Omega_1'| \ge (\eps\tau)^{O(k+\ell)}|\Omega_1|$ and $|\Omega_2'| \ge (\eps\tau)^{O(k+\ell)}|\Omega_2|$, i.e., the subrectangle is not too small. For completeness, we provide a proof of such a sifting statement in \cref{thm:sift}. Later we will see that this sifting theorem does \emph{not} suffice for our corners bound for two reasons, and instead we require an \emph{asymmetric} and \emph{relative} sifting theorem which is much more challenging to establish (see \cref{thm:quasisifting2}).

Applying \cref{thm:sift}, if \eqref{eq:xdgrid} holds, one may find sets $X \subseteq \mb{F}_2^n$ and $D \subseteq \mb{F}_2^n$ such that
\[\E_{x \in X, d \in D} \mbm{1}_A(x,x+d) \ge (1+\Omega(1))\alpha\]
and 
\[\E_{x,d}\mbm{1}_{x\in X}\mbm{1}_{d\in D} \ge \exp(-O(\log(1/\alpha)^{2})).\]
In other words, the set $A$ has larger density on an $(X, D)$-set, i.e., a set of the form
\[ \{(x,y) \in \mb{F}_2^n \times \mb{F}_2^n : x \in X, x+y \in D\}. \]
Similarly, if \eqref{eq:ydgrid} holds then \cref{thm:sift} gives that $A$ has increased density on a $(Y,D)$-set:
\[ \{(x,y) \in \mb{F}_2^n \times \mb{F}_2^n : y \in Y, x+y \in D\}. \]

This now gives a first crucial difference between our setting and that of Shkredov \cite{Shk06}. While Shkredov's proof always maintained a set $X \times Y$ as the pseudorandom container to density increment against, we cannot afford to do this as we may be forced to use either $(X,D)$ or $(Y,D)$ increments. Instead, our pseudorandom containers take the form (for a subspace $W \subseteq \mb{F}_2^n$)
\[ S(X,Y,D) \coloneqq \{(x,y) \in W \times W : x \in X, y \in Y, x+y \in D\}. \]
We refer to these sets as $(X, Y, D)$-sets.
Thus, in our density increment we maintain a subset $A \subseteq S(X, Y, D)$ whose density increases over the course of the procedure.
While this initially seems worrying because there is no obvious guarantee on the size of the set $S(X,Y,D)$, it is not hard to prove that if any of $X, Y, D$ are Fourier-pseudorandom (like in Shkredov's proof), then the size of such a set is as expected: approximately $|X||Y||D|/|W|$.

It remains to discuss how having a more complicated pseudorandom container affects the remainder of the argument. The main difficulty to establishing bounds of the form $4^{n} \cdot \exp(-(\log n)^{\Omega(1)})$ (i.e., a single logarithmic improvement over Shkredov's bound) is in establishing a version of the ``sifting'' lemma in this setting. The challenge is that in \cref{thm:sift} the resulting output rectangle depends on the density/grid norm of the function. Because we are working with $(X,Y,D)$-sets, when viewed as a subset of $X \times Y$ e.g., the set itself is also extremely sparse. Thus, na\"{i}vely applying the sifting lemma of \cite{KLM24} loses factors corresponding to relative density of $S(X,Y,D)$ within $X \times Y$. A key innovation of this paper is therefore proving a version of sifting which does not lose such factors. 

\subsection{Quasirandom sifting}\label{subsec:quasi-expo}
We now discuss quasirandom sifting which is one of the primary technical innovations of the paper. The necessary quasirandom sifting statement concerns when a function has large $(2,k)$-grid norm. In this section, we take a graph-theoretic view -- if a function has large $(2,k)$-grid norm then the corresponding graph has ``excess'' copies of $K_{2,k}$. We prove that this implies that then there is a density increment onto a rectangle. 

We first specialize to the case of $K_{2,2}$ which is already nontrivial. In this case, we seek to analyze $f(x,y):\Omega_1\times \Omega_2\to [0,1]$ which is bounded by a pseudorandom majorant $0 \le T(x,y) \le 1$ (meaning that $f(x,y) \le T(x,y)$) such that $\E[f] = \alpha \tau$ and $\E[T] = \tau$ and with
\[ \E_{\substack{x, x' \in \Omega_1 \\ y, y' \in \Omega_2}}\left[f(x,y)f(x',y)f(x,y')f(x',y') \right] \ge (1+\eps)^4\alpha^4\tau^4.\]
We must specify what it means for $T$ to be pseudorandom. Our notion of pseudorandomness will be called \emph{combinatorial spreadness}: the density of $T$ within any somewhat large rectangle is bounded by $(1+\eps)\tau$ (\cref{def:combpseudo}).

We wish to find a large rectangle on which $f$ has density at least $(1+\Omega(\eps))\alpha\tau$. Crucially, we want the size of the rectangle to not depend on $\tau$, and only on $\alpha$. The proof of this must necessarily rely on the pseudorandomness of $T$. In the setting without the pseudorandom majorant $T$, the proof of sifting in \cite{KLM24} picks a random $(x',y')$ where $f(x',y') = 1$, fixes it, and then notes that $f(x, y')f(x', y)$ gives a rectangle which correlates to $f$. 

However, the density of the rectangle itself is $O(\tau^2)$, because for a fixed $(x',y')$ the set of $x$ such that $f(x,y') = 1$ is $O(\tau)$ due to the presence of the majorant $T$, and similarly for $y$.

Motivated by ``densification'' in works of Conlon, Fox, and Zhao \cite{CFZ14}, we try to replace the $K_{2,2}$ counts with counts of a larger graph such that there are edges not involved in cycles of length $4$. We will see later why this is useful. In order to not lose factors of $\alpha$, our first step is to apply H\"{o}lder's inequality to obtain that 
\[ \E_{\substack{x \in \Omega_1 \\ y\in \Omega_2}}\left[f(x,y) \right]^{k-1} \cdot \E_{\substack{x \in \Omega_1 \\ y\in \Omega_2}}\left[f(x,y)\Big(\E_{\substack{x' \in \Omega_1 \\ y'\in \Omega_2}} f(x',y)f(x,y')f(x',y')\Big)^k\right]\ge (1+\eps)^{4k}\alpha^{4k}\tau^{4k}. \]
Taking $k\approx \log(1/\alpha)/\eps$ and rearranging, we find that 
\[\E_{\substack{x \in \Omega_1 \\ y\in \Omega_2}}\left[f(x,y)\Big(\E_{\substack{x' \in \Omega_1 \\ y'\in \Omega_2}} f(x',y)f(x,y')f(x',y')\Big)^k\right]\ge (1+\eps)^{4k}\alpha^{4k}\tau^{4k}/(\alpha \cdot \tau)^{k-1}\ge (1+\eps)^{4k} \cdot \alpha^{3k+1}\tau^{3k+1}. \]
We next apply the fact that $f\le T$ to obtain 
\[\E_{\substack{x \in \Omega_1 \\ y\in \Omega_2}}\left[T(x,y)\Big(\E_{\substack{x' \in \Omega_1 \\ y'\in \Omega_2}} f(x',y)f(x,y')f(x',y')\Big)^k\right]\ge (1+\eps)^{4k} \cdot \alpha^{3k+1}\tau^{3k+1}.\]
As $T$ is combinatorially spread, we find that 
\[\E_{\substack{x \in \Omega_1 \\ y\in \Omega_2}}\left[\Big(\E_{\substack{x' \in \Omega_1 \\ y'\in \Omega_2}} f(x',y)f(x,y')f(x',y')\Big)^k\right]\ge (1+\eps)^{4k} \cdot \alpha^{3k+1}\tau^{3k}\ge (1+\eps)^{3k} \cdot \alpha^{3k}\tau^{3k}.\]

Observe that through these manipulations we have transformed a $C_4$ into a graph $G'$ which is a series of $k$ paths of length $3$ which are joined at a pair of vertices. The key difference is that this graph has girth $6$. By telescoping, there exists an edge $e$ and a subgraph $H\subseteq G'$ such that the count of $H$ is at least $(1+\eps)\alpha\tau$ times the count of $H \setminus \{e\}$; for the sake of simplicity we assume here that $H = G'$. This give that
\begin{align*}
\E_{\substack{x,x' \in \Omega_1 \\ y,y'\in \Omega_2}}&\left[f(x',y)f(x,y')f(x',y')\Big(\E_{\substack{x'' \in \Omega_1 \\ y''\in \Omega_2}} f(x'',y)f(x,y'')f(x'',y'')\Big)^{k-1}\right]\\
&\ge (1+\eps)\alpha\tau\cdot \E_{\substack{x,x' \in \Omega_1 \\ y,y'\in \Omega_2}}\left[f(x,y')f(x',y')\Big(\E_{\substack{x'' \in \Omega_1 \\ y''\in \Omega_2}} f(x'',y)f(x,y'')f(x'',y'')\Big)^{k-1}\right].
\end{align*}
Let $R(x,y) = \E_{\substack{x'' \in \Omega_1 \\ y''\in \Omega_2}} f(x'',y)f(x,y'')f(x'',y'')$ and observe that $R(x,y)$ has mean at least $\alpha^{3}\tau^{3}$ while essentially always being bounded by $\tau^{3}$. In particular, observe that $\tau^{-3}R(x,y)$ is a ``dense'' function. Furthermore let $R_2(x',x) = \E_{y'\in \Omega_2}f(x,y')f(x',y')$; we similarly have that $\tau^{-2}R(x',x)$ is a dense function.
This gives that 
\begin{align*}
\E_{x\in \Omega_1} \E_{\substack{x' \in \Omega_1 \\ y\in \Omega_2}}f(x',y)R_2(x',x)R(x,y)^{k-1}
&\ge (1+\eps)\alpha\tau\cdot \E_{x\in \Omega_1} \E_{\substack{x' \in \Omega_1 \\ y\in \Omega_2}}R_2(x',x)R(x,y)^{k-1}.
\end{align*}
Choosing an $x$ for which $R_2(x',x)$ and $R(x,y)$ are appropriately dense then immediately gives the desired conclusion. 
The extension of the above argument to $K_{2,k}$ involves an iterative combination of H\"{o}lder's inequality and induction on smaller graphs.

This technical manipulation of replacing $f$ by ``denser codegree'' functions has been described as ``densification'' in works of Conlon, Fox, and Zhao \cite{CFZ14} which proved various graph counting lemmas in the presence of a pseudorandom majorant. An extension of these results to hypergraph was given in work of the same authors \cite{CFZ15} to prove a relative Szemer\'{e}di theorem; this in turn simplified the proof of a key ingredient to the celebrated Green--Tao theorem \cite{GT08}. However these results, if applied directly, would lose various factors of $\alpha$ in the density increment which is unacceptable. Intuitively, one way we avoid losing these factors of $\alpha$ is by applying H\"{o}lder's inequality as above instead of Cauchy-Schwarz when doing the densification procedure.

There are several remarks in order for the general case. First observe that the results of Conlon, Fox, and Zhao \cite{CFZ14} and more generally counting lemmas with respect to sparse graphs have been phrased with forcing $T$ to be pseudorandom with respect to the cut norm. While this would be acceptable for $4^{n} \cdot \exp(-(\log n)^{\Omega(1)})$-bounds, it turns out to be insufficient for quasipolynomial bounds. An important observation is that for our proof all that is required is that $T$ is bounded ``on large rectangles''; e.g.~$\E[A(x)T(x,y)B(y)]\le (1+\eps)\tau \cdot \E[A(x)B(y)] + \gamma$ for all $1$-bounded functions $A$ and $B$. This pseudorandomness criterion will prove to be quantitatively superior.

One critical feature we have not discussed to this point is that the rectangles we obtain will naturally have asymmetric sizes -- note in the statement of \cref{thm:quasisifting2} that the size of the $D$ side does not have any dependence on $k$, and is instead just $2^{-O(\log(1/\alpha)^2)}$.
To see why this is sensible, let us first consider the simpler case where there is no pseudorandom majorant and $f$ has density $\alpha$.

Suppose that 
\[\E\Big[\prod_{i\in [k]}f(x_1,y_i)f(x_2,y_i)\Big]\ge (1+\eps)^{2k}\alpha^{2k}.\]
We will argue informally that there are $g_1(x)$ and $g_2(y)$ such that 
\begin{align*}
\E[g_1(x)f(x,y)g_2(y)]&\ge (1+\eps/2) \cdot \alpha \cdot \E[g_1] \cdot \E[g_2]\\
\E[g_1] &\ge \Omega(\eps\alpha)^{O(k)}\\
\E[g_2] &\ge \Omega(\eps\alpha)^{O(1)}.
\end{align*}
Think of how $g_1$ and $g_2$ are constructed in the proof of sifting (\cref{thm:sift}), say when the edge $(x_1, y_1)$ is removed. Then $g_1$ is formed as the intersection of neighborhoods of $k-1$ vertices $y_2, \dots, y_k$, and $g_2$ is the neighborhood of the vertex $x_2$. Thus, if $f$ has density $\alpha$, then one expects that $g_1$ has density at least $\alpha^{O(k)}$ and $g_2$ has density at least $\alpha^{O(1)}$.
Observe in particular that $\E[g_2]$ has size independent of $k$. This has previously gone unnoticed and will be crucial for our analysis.

\subsection{Asymmetric increment}\label{sec:overview_asymmetric_inc}
We now return to the initial H\"{o}lder manipulation in the case where $A \subseteq S(X, Y, D)$ (equivalently, $\mbm{1}_A(x,y)\le \mbm{1}_X(x)\mbm{1}_Y(y)\mbm{1}_D(x+y)$) and $\E[f(x,y)] = \alpha \cdot \delta_X\delta_Y\delta_D$, where $\delta_X, \delta_Y, \delta_D$ are the densities of $X,Y,D$ respectively. We have not currently specified the pseudorandomness conditions we will impose on $X$, $Y$, and $D$, but they will guarantee that $\E[\mbm{1}_X(x)\mbm{1}_Y(y)\mbm{1}_D(x+y)] \approx \delta_X \delta_Y \delta_D$.

We begin with the H\"{o}lder manipulation as in \cref{subsec:XYD-expo}; tracking various support functions we find that 
\[\E_{x,y\in \mb{F}_2^{n}}\mbm{1}_A(x,y)(\E_{d\in \mb{F}_2^n}\mbm{1}_A(x,x+d)f_A(y+d,y))^{k}\ge (c\alpha)^{2k+1} \cdot \delta_X^{k+1}\delta_Y^{k+1}\delta_D^{k+1}\]
where $c$ is a sufficiently small absolute constant. As $k$ is at least a large constant times $\log(1/\alpha)$, we may (at the cost of changing $c$) instead consider 
\[\E_{x,y\in \mb{F}_2^{n}}\mbm{1}_D(x+y)(\E_{d\in \mb{F}_2^n}\mbm{1}_A(x,x+d)f_A(y+d,y))^{k}\ge (c\alpha)^{2k+1} \cdot \delta_X^{k+1}\delta_Y^{k+1}\delta_D^{k+1}.\]

We now come to one of the main gambits in this paper. In the work of Shkredov, one maintains that the $X$ and $Y$ are Fourier pseudorandom. One may attempt to maintain $X,Y,D$ which are Fourier pseudorandom and this is possible; however the associated pseudorandomization procedure is inherently far too lossy for quasipolynomial bounds. The trick is to give up $D$ (and a factor of $\delta_D$). This is possible if one takes $k \approx \log(1/(\alpha\delta_D))$.

We now apply the appropriate quasirandom sifting theorem (\cref{thm:quasisifting2}) and we find that there exists $g_i$ such that $g_1(x)\le \mbm{1}_X(x)$, $g_2(d)\le \mbm{1}_D(d)$, and  
\begin{align*}
\E[g_1(x)\mbm{1}_A(x,x+d)g_2(d)]&\ge (1+\Omega(1)) \cdot \alpha \cdot \delta_Y \cdot \E[g_1] \cdot \E[g_2]\\
\E[g_1(x)] &\ge \delta_X \cdot \exp(-\log(1/( \alpha\delta_D))^{O(1)})\\
\E[g_2(d)] &\ge \delta_D \cdot \exp(-\log(1/\alpha)^{O(1)})
\end{align*}
or an analogous statement for $\mbm{1}_A(y+d,y)$. We are crucially using here the asymmetric nature of the density increment. In general, in our argument $D$ will end up being much denser than $X,Y$.

Heuristically, if pseudorandomization is not too costly, we would have that 
\begin{align*}
\delta_D' &\leftarrow \delta_D \cdot \exp(-\log(1/\alpha)^{O(1)}), \\
\delta_X'&\leftarrow \delta_X \cdot \exp(-\log(1/(\alpha\delta_D))^{O(1)}),\text{ and } \\
\delta_Y'&\leftarrow \delta_Y \cdot \exp(-\log(1/(\alpha\delta_D))^{O(1)}).
\end{align*}
As we will have only $\log(1/\alpha)$ iterations, the densities obtained will always be quasipolynomial. Furthermore, observe that this ``dropping'' of the $D$ indicator will mean that we will require no pseudorandomness conditions on $D$. 
(In our later discussion of the general abelian case, we will need to reimpose such a condition on $D$.)
This asymmetry and the realization to ``drop'' the indicator of $D$ (e.g.~not attempt to account for its density throughout) are crucial ingredients in this work.

\subsection{Pseudorandomization into spread components}

The one remaining consideration is how to decompose the ``dense'' set which is output by the previous step and pass to a rectangle where the sides are suitably ``spread'' (the pseudorandomness property we maintain on $X$ and $Y$). The work of Shkredov \cite{Shk06} relies on an $L^2$-energy increment argument. This however is rather costly; in order to guarantee that the underlying ``sides'' have Fourier coefficients bounded by $\delta$ one needs to pass to a codimension $\delta^{-O(1)}$ set. Furthermore observe that to this point in the argument we have made no use of the tools of almost periodicity which are crucial in the work of Kelley and Meka \cite{KM23} as well as much of the recent work on Roth's theorem.

The trick therefore is to relax the notion of pseudorandomness to being ``upper bounded'' against rectangles. More precisely, we will ensure that if $X:\mb{F}_2^{n}\to \{0,1\}$ then for all $A,B:\mb{F}_2^{n} \to [0,1]$ that 
\[\E[A(x)X(x+y)B(y)] \le (1+\eps) \cdot \E[X] \cdot \E[A] \cdot \E[B] + \gamma. \]
The main consequence of the methods of Kelley--Meka \cite{KM23} (essentially \cite[Theorem~4.10]{KM23}), is that if $X$ is not upper bounded in such a manner, then $X$ has increased density (by a $1+\Omega(\eps)$ factor) on a subspace of codimension $O_{\eps}(\log(1/\gamma)^{O(1)})$. This motivates our definition of an algebraically spread set $X$ (\cref{def:algpseudo}): a set $X$ is \emph{algebraically spread} if it does not admit a density increment onto an affine subspace of small codimension $r = O_{\eps}(\log(1/\gamma)^{O(1)})$.

This suggests the idea of using a ``spread decomposition'' to do the pseudorandomization procedure to $X$ and $Y$. Given a set $X$, one may either say $X$ is spread and output $X$ or repeatedly pass $X$ to a subspace where it has increased density. This procedure allows one to decompose $X$ into spread pieces and is closely related to various spread decompositions which have notably been used in the context of the sunflower conjecture \cite{ALWZ21} and by Kupavskii and Zakharov \cite{KZ24} in the context of forbidden intersection problems. In particular, this latter work crucially relies on a decomposition into spread pieces.  

The major difficulty in this situation, however, is finding a decomposition which makes both $X$ and $Y$ spread at the same time, while also not causing the density of $D$ and the set $A$ to drop significantly. The ultimate proof is essentially a certain recursive ping-pong; one ensures that $X$ is first so spread that various manipulations on $Y$ will not increase the density of $X$. Then one splits $Y$ and argues that almost all pieces that $X$ splits into do not drop in density, and are hence still spread. The result of this is a decomposition of $X \times Y$ into subrectangles, where at least half of them have both sides spread. Thus, recursively iterating on the rectangles that are not spread, one decomposes $X\times Y$ into a set of ``spread rectangles'', and this allows one to conclude the proof. The details are rather algorithmic/recursive in nature; the main issue whenever one is dealing with spread decomposition in our context is that while one has an upper bound on the density of $X$, it may drop dramatically on a small fraction of the set. We also remark that our proof could even be used in the case of Shkredov \cite{Shk06} and guarantee a Fourier pseudorandomness condition (with a similar dimension drop as in \cite{Shk06}) while completely avoiding the $L^2$-energy increment strategy central to this work.  

\subsection{Bohr set adaptations and additional pseudorandomization}
We now briefly discuss the crucial technical changes required for adapting the proof to general groups. The first technical issue is with the use of the parameterization of $(x,y)$, $(-y-z,y)$, $(x,-x-z)$ in the context of Bohr sets.
(Recall such a parameterization arises from the change of variable $z \to -x-y-z$ in the definition of a corner.)
Observe that if one restricts $x\in B_1$ and $y\in B_2$ where $B_1,B_2$ are Bohr sets and $B_2$ is more ``narrow'' than $B_1$, then essentially any restriction on the range of $z$ causes the three coordinates to not ``live'' on the same range. In this context, we instead consider 
\[(x+x',y+y'), (x-y'+z',y+y'), (x+x',y-x'+z')\]
where $x\sim B_1$, $y\sim B_2$, $x'\sim B_3$, $y'\sim B_4$ and $z'\sim B_5$. This essentially corresponds to considering corners in a ``narrow'' box around $(x,y)$, and note that now each coordinate individually ranges over the whole domain $B_1\times B_2$ in an essentially uniform way. The setup of the initial argument is now rather more delicate; one needs to find $(x,y)$ (which we call a \emph{great pair}) such that the three rectangles induced by $(x+x',y+y')$, $(x-y'+z',y+y')$, $(x+x',y-x'+z')$ all have sufficient density and spreadness. This is quite technical, but a certain relatively routine but lengthy argument suffices.

The final technical hurdle is rather more subtle. Observe in the case of Bohr sets that the size of the container is given by $\E_{\substack{x\sim B_1\\y\sim B_2}}\mbm{1}_X(x)\mbm{1}_Y(y)\mbm{1}_D(x+y)$. We want to enforce pseudorandomness notions on $X, Y, D$ that guarantee that the size of the container is close to expected. A natural notion to enforce on $X, Y$ is the generalization of algebraic spreadness to Bohr sets: there is no density increment onto a smaller Bohr set of slightly higher rank and smaller radius. Using that $X$ and $Y$ are spread one can prove that this is quite close to $(\E_{\substack{x\sim B_1\\y\sim B_2}}\mbm{1}_X(x)\mbm{1}_D(x+y)) \cdot (\E_{y\in B_2}\mbm{1}_Y(y))$. 

There is no guarantee, however, that $\E_{\substack{x\sim B_1\\y\sim B_2}}\mbm{1}_X(x)\mbm{1}_D(x+y)$ is close to $\E_{x\sim B_1}\mbm{1}_X(x) \cdot \E_{z\sim B_1}\mbm{1}_D(z)$. 
(This is unlike the finite field setting where the initial term factors.)
The failure case occurs precisely when $D$ places ``most'' of its mass where $X$ dips below its mean. Essentially, the minimal condition one can use to obtain our desired guarantee is that $D$ is typically not below its mean; for this the rather weak bound of the form $\E_{z\sim B_1}|\E_{z'\sim B_3}\mbm{1}_D(z+z') - \delta_D|\le \eps \cdot \delta_D$ suffices. Observe that if our pseudorandomization procedure could also make $D$ spread this would be sufficient. However, this seems technically quite challenging. Instead, we observe that if this weak $\ell_1$-condition is not true, then $\E_{z\sim B_1}|\E_{z'\sim B_3}\mbm{1}_D(z+z') - \delta_D|\le \eps \cdot \delta_D$ fails on a constant fraction of translates of $B_3$. This combined with a certain log-potential analysis (see \cref{lemma:dsplit}) and an additional recursive layer allows one to handle the necessary pseudorandomization procedure. It appears a rather interesting question to understand what the limit of these ``efficient pseudorandomization'' procedures are, as these are rather different than the more standard $L^2$-based strategies used in the literature. 

We remark that much of our analysis in this section is in terms of arithmetic grid norms following the work of Mili{\'c}evi{\'c} \cite{Mil24} on skew corners.

\subsection{Notation}
We use standard asymptotic notation for functions on $\mb{N}$ throughout. Given functions $f=f(x)$ and $g=g(x)$, we write $f=O(g)$, or $g = \Omega(f)$, to mean that there is a constant $C$ such that $|f(x)|\le Cg(x)$ for sufficiently large $x$. Additionally, we write $f = \Theta(g)$ to mean $f = O(g)$ and $g = O(f)$.
Subscripts indicate dependence on parameters. Furthermore given a function $f:\mb{Z}\to \mb{C}$, we say $f$ is $1$-bounded if $\sup_{x\in \mb{Z}}|f(x)|\le 1$. For a finite set $\Omega$ and function $f \colon \Omega \to \mb{C}$, we denote the $k$-norm of $f$ by 
$$
    \|f\|_k = \left(\E_{x \in \Omega} |f(x)|^k \right)^{1/k}.
$$
For a finite abelian group $G$, we denote the convolution of two functions $f,g : G \to \mathbb{R}$ by
$$
    (f \ast g)(x) = \E_{y \in G} \left[ f(y)g(x-y) \right].
$$

\textbf{Parameters:} $\alpha$ will always be related to the density of the set within the container.
Throughout the proof, $\eps$ with subscripts, etc.~will all denote absolute constants independent of the size of the group $G$. $K, k, \ell$ will be used as the size of the grid norms we consider, and will all be $O(\log(1/\alpha)^{O(1)})$. $\gamma$ will denote a pseudorandomness parameter, and will be $\exp(-\log(1/\alpha)^{O(1)})$. $r$ will be related to the dimension up to which our sets are algebraically spread, and will be $O(\log(1/\alpha)^{O(1)})$.

$\delta$ and $\delta_X, \delta_Y, \delta_D$ will denote densities of the sets inducing the container within the ambient subspace. These will all be at least $\exp(-\log(1/\alpha)^{O(1)})$.

In the section on Bohr sets, $\eta$ will be used to denote the ratio between radii of regular Bohr sets with the same set of frequencies. In particular, we will often consider $B_1 \supseteq B_2 \supseteq B_3 \supseteq \dots$ where $B_i$ are all regular and the radii (denoted as $r_i$) satisfy $r_{i+1}/r_i \le \eta$.

\section{Combinatorial Spreadness and Relative Sifting}\label{sec:relative-sifting}
Throughout our analysis in this paper the sets $\Omega_i$ are finite (and therefore have a well-defined uniform measure). We first define the notion of combinatorial spreadness which will be used throughout this section.
\begin{definition}[Combinatorial spreadness]
\label{def:combpseudo}
Let $\Omega_1, \Omega_2$ be sets. A subset $T \subseteq \Omega_1 \times \Omega_2$ is \emph{$(\tau, \gamma)$-combinatorially spread} if for all functions $f: \Omega_1 \to [0, 1]$ and $g: \Omega_2 \to [0, 1]$ it holds that
\[ \E_{x \in \Omega_1, y \in \Omega_2}[f(x)g(y)\mbm{1}_T(x,y)] \le \tau\E_{x \in \Omega_1}[f(x)] \E_{y \in \Omega_2}[g(y)] + \gamma. \]
\end{definition}
Intuitively, one should think that $\E_{x \in \Omega_1, y \in \Omega_2}[\mbm{1}_T(x, y)]$ is approximately $\tau$.
Note we will occasionally refer to a boolean function as $(\tau, \gamma)$-combinatorially spread if it is the indicator function of a $(\tau, \gamma)$-combinatorially spread set. 
If we restrict \Cref{def:combpseudo} to functions $f,g$ which are $\{0,1\}$-valued, it is equivalent to think of combinatorial spreadness as mandating that $T$ has no significant density increment on rectangles of density at least $\Omega(\gamma)$. Later, we will define a version of combinatorial spreadness which is asymmetric; namely, we will have different density requirements on the functions $f,g$, rather than a single density requirement on the rectangle $f(x)g(y)$.
We first note a counting lemma for our notion of combinatorial spreadness.
\begin{lemma}[Counting lemma]
\label{lemma:counting}
Let $\Omega_1, \Omega_2$ be sets, $T:\Omega_1\times\Omega_2\to \{0,1\}$ be $(\tau,\gamma)$-combinatorially spread. Let $G = ([k] \cup [\ell], E)$ be a bipartite graph with edge $(i^*,j^*)\in E$. For $(i,j)\in E\setminus \{(i^*,j^*)\}$ let $f_{ij}: \Omega_1\times\Omega_2\to [0,1]$ be nonnegative $1$-bounded functions. Then
\[ \E_{\substack{x_1,\dots,x_k \in \Omega_1 \\ y_1,\dots,y_\ell \in \Omega_2}} \Big[T(x_{i^*},y_{j^*}) \prod_{(i,j) \in E\setminus \{(i^*,j^*)\}} f_{ij}(x_{i},y_{j}) \Big] \le \tau \E_{\substack{x_1,\dots,x_k \in \Omega_1 \\ y_1,\dots,y_\ell \in \Omega_2}} \Big[\prod_{(i,j) \in E\setminus \{(i^*,j^*)\}} f_{ij}(x_{i},y_{j}) \Big] + \gamma. \]
\end{lemma}

\begin{proof}
We can write
\begin{align*} &\E_{\substack{x_1,\dots,x_k \in \Omega_1 \\ y_1,\dots,y_\ell \in \Omega_2}}  \Big[T(x_{i^*},y_{j^*}) \prod_{(i,j) \in E\setminus \{(i^*,j^*)\}} f_{ij}(x_{i},y_{j}) \Big] \\ =
~& \E_{\substack{x_i : i \in [k]\setminus \{i^*\} \\ y_j : j \in [\ell]\setminus\{j^*\}}} \Bigg[\prod_{\substack{(i,j) \in E\\ i\neq i^*,j\neq j^*}} f_{ij}(x_i,y_j) \E_{x_{i^*},y_{j^*}} \Big[T(x_{i^*},y_{j^*}) \prod_{j:(i^*,j)\in E} f_{i^*j}(x_{i^*},y_j) \prod_{i:(i, j^*)\in E} f_{ij^*}(x_i,y_{j^*}) \Big]\Bigg] \\
\le ~&\E_{\substack{x_i : i \in [k]\setminus \{i^*\} \\ y_j : j \in [\ell]\setminus\{j^*\}}}\Bigg[\prod_{\substack{(i,j) \in E\\ i\neq i^*,j\neq j^*}} f_{ij}(x_i,y_j) \Big(\tau \E_{x_{i^*},y_{j^*}} \Big[ \prod_{j:(i^*,j)\in E} f_{i^*j}(x_{i^*},y_j) \prod_{i:(i, j^*)\in E} f_{ij^*}(x_i,y_{j^*}) \Big] + \gamma\Big) \Bigg] \\
\le ~& \tau \E_{\substack{x_1,\dots,x_k \in \Omega_1 \\ y_1,\dots,y_\ell \in \Omega_2}}   \Big[\prod_{(i,j) \in E\setminus \{(i^*,j^*)\}} f_{ij}(x_{i},y_{j}) \Big] + \gamma,
\end{align*}
where the first inequality uses $(\tau, \gamma)$-combinatorial spreadness of $T$, and the second uses that the function $f_{ij}$ are $1$-bounded.
\end{proof}

We first require the following lemma which states that if one correlates with a bounded product function then one may extract a correlation with a pair of sets. A version of this statement appears as \cite[Claim~4.6]{KLM24}.
\begin{lemma}\label{lem:extract-cor}
Let $\tau>0$, $A:\Omega_1\times \Omega_2\to \mb{R}$, and $f_i:\Omega_i\to [0,1]$ be such that 
\[\E_{x,y}[f_1(x)f_2(y)A(x,y)]\ge \tau \cdot \E_{x,y}[f_1(x)f_2(y)].\]
Then there exist $g_i:\Omega_i\to \{0,1\}$ with 
\[\E_{x,y}[g_1(x)g_2(y)A(x,y)]\ge \tau \cdot \E_{x,y}[g_1(x)g_2(y)]\]
with $\E[g_i(x)]\ge \E[f_i(x)]/2$.
\end{lemma}
\begin{proof}
Consider $\vec{\alpha} = (\alpha_1,\ldots,\alpha_{\ell})\in [0,1]^{\ell}$ with $\sum \alpha_i = k$. We will show that $\vec{\alpha}$ may be written as a convex combination of $\vec{v}\in \{0,1\}^{\ell}$ with $\snorm{\vec{v}}_1\in \{\lfloor k\rfloor,\lceil k\rceil\}$. Applying this claim to $\alpha = (f_i(x): x \in \Omega_i)$, we may decompose 
\begin{equation}\label{eq:f_i_decomp}
    f_i(x) = \sum_{j=1}^{t_i}w_j g_{i,j}(x),
\end{equation}
where $g_{i,j} : \Omega_i \to \{0,1\}$ with
\[ \E[g_{i,j}(x)] \in \Big\{ \frac{1}{|\Omega_i|} \cdot \Big\lfloor |\Omega_i| \cdot \E[f_i(x)] \Big\rfloor, \frac{1}{|\Omega_i|} \Big\lceil |\Omega_i| \cdot \E[f_i(x)] \Big\rceil \Big\} \]
and $w_j>0$.
If $\lfloor \E[f_i(x)] \cdot |\Omega_i|\rfloor = 0$, then we may drop terms with $\E[g_{i,j}(x)] = 0$; note $\E[g_{i,j}(x)] \ge \E[f_i(x)]$ for the remaining terms.
Otherwise $\E[f_i(x)] \ge 1/|\Omega_i|$, and we deduce $\E[g_{i,j}(x_i)]\ge \E[f_i(x)]/2$ for all $j$. Substituting \Cref{eq:f_i_decomp} into our original assumption, we find 
\[\sum_{j=1}^{t_1}\sum_{j'=1}^{t_2}w_jw_{j'}\E[(A(x,y) - \tau) g_{1,j}(x)g_{2,j'}(y)]\ge 0.\]
At least one of the summands is nonnegative and the result follows.

To prove the initial claim, we proceed by induction on the number of coordinates $\ell$ which are strictly between zero and one. The case when $\ell = 1$ is trivial. If there are coordinates $\alpha_i + \alpha_j \le 1$, then by writing 
\[(\alpha_i,\alpha_j) = \frac{\alpha_i}{\alpha_i+ \alpha_j}(\alpha_i+ \alpha_j,0) + \frac{\alpha_j}{\alpha_i+ \alpha_j}(0,\alpha_i+ \alpha_j)\]
we may proceed by induction downward. Else if $1<\alpha_i + \alpha_j\le 2$, then 
\[(\alpha_i,\alpha_j) = \frac{1-\alpha_i}{2-\alpha_i-\alpha_j}(\alpha_i + \alpha_j - 1, 1)  + \frac{1-\alpha_j}{2-\alpha_i-\alpha_j}(1, \alpha_i + \alpha_j - 1),\]
and again we may proceed by induction downward.
\end{proof}

Finally, we require the basic (non-relative) version of sifting. For completeness (and to provide a slightly better bound than is in the literature, e.g., \cite[Lemma 4.7]{KLM24}), we provide a proof.
\begin{theorem}[Sifting]
\label{thm:sift}
Let $f: \Omega_1 \times \Omega_2 \to [0, 1]$ satisfy that $\|f\|_{G(k, \ell)} \ge \alpha$. For any $\eps > 0$, there are functions $g_1: \Omega_1 \to [0, 1]$ and $g_2: \Omega_2 \to [0, 1]$ such that
\[ \E_{x \in \Omega_1, y \in \Omega_2}[f(x,y)g_1(x)g_2(y)] \ge (1-\eps)\alpha\E_{x \in \Omega_1}[g_1(x)]\E_{y \in \Omega_2}[g_2(y)], \]
and
\[ \E_{x \in \Omega_1}[g_1(x)] \E_{y \in \Omega_2}[g_2(y)] \ge \eps\alpha^{O(k+\ell)}. \]
\end{theorem}

\begin{proof}
If $\|f\|_{G(i, j)} \ge \alpha$ for any positive integers $i \le k$ and $j \le \ell$ with $i+j < k+\ell$, then the result follows by induction. Assume otherwise for the remainder of the argument. Let $G^{(0)}, G^{(1)}, \dots, G^{(k+\ell-1)}$ be an increasing sequence of graphs (each adding one edge to the previous) where $G^{(0)}$ is $K_{k-1,\ell-1}$ and $G^{(k+\ell-1)}$ is $K_{k,\ell}$, each with vertex set $V = [k] \cup [\ell]$. Then there is some $0 \le t \le k+\ell-2$ such that if $(i^*, j^*)$ is the edge in $G^{(t+1)} \setminus G^{(t)}$, then
\[ \E_{\substack{x_1,\dots,x_k \in \Omega_1 \\ y_1,\dots, y_\ell \in \Omega_2}}\left[f(x_{i^*}, y_{j^*}) \prod_{(i,j) \in E(G^{(t)})} f(x_i, y_j) \right] \ge \alpha \E_{\substack{x_1,\dots,x_k \in \Omega_1 \\ y_1,\dots, y_\ell \in \Omega_2}}\left[\prod_{(i,j) \in E(G^{(t)})} f(x_i, y_j) \right], \]
which can be rearranged to get
\begin{align*}
\E_{\substack{x_1,\dots,x_k \in \Omega_1 \\ y_1,\dots, y_\ell \in \Omega_2}}\left[(f(x_{i^*}, y_{j^*}) - (1-\eps)\alpha) \prod_{(i,j) \in E(G^{(t)})} f(x_i, y_j) \right] &\ge \eps\alpha \E_{\substack{x_1,\dots,x_k \in \Omega_1 \\ y_1,\dots, y_\ell \in \Omega_2}}\left[\prod_{(i,j) \in E(G^{(t)})} f(x_i, y_j) \right] \\
&\ge \eps\alpha \|f\|_{G(k,\ell)}^{k\ell}.
\end{align*}

Let $\bar{E} \subseteq E(G^{(t)})$ be the edges not involving either of $i^*$ or $j^*$ and observe that $\bar{E}$ contains $K_{k-2,\ell-2}$. Thus we have that $\eps\alpha \|f\|_{G(k,\ell)}^{k\ell}\ge \eps \cdot \alpha^{O(k+\ell)} \cdot \mb{E}[\prod_{(i,j) \in \bar{E}} f(x_i, y_j)]$. Observe that 
\begin{align*}
\eps&\cdot \alpha^{O(k+\ell)} \cdot \mb{E}[\prod_{(i,j) \in \bar{E}} f(x_i, y_j)]\\ &\le \E_{\substack{x_1,\dots,x_k \in \Omega_1 \\ y_1,\dots, y_\ell \in \Omega_2}}\left[(f(x_{i^*}, y_{j^*}) - (1-\eps)\alpha) \prod_{(i,j) \in E(G^{(t)})} f(x_i, y_j) \right] \\
&= \E_{x_{i^*}, y_{j^*}}\Bigg[(f(x_{i^*}, y_{j^*}) - (1-\eps)\alpha) \cdot \\
& \qquad\E_{\substack{x_i : i \in [k]\setminus \{i^*\} \\ y_j : j \in [\ell]\setminus\{j^*\}}}\Big[\prod_{j : (i^*,j) \in E(G^{(t)})} f(x_{i^*}, y_j) \prod_{i : (i, j^*) \in E(G^{(t)})} f(x_i, y_{j^*}) \prod_{(i,j) \in \bar{E}} f(x_i, y_j) \Big] \Bigg].
\end{align*}

Thus we have that where exists a choice of $\{x_i : i \in [k]\setminus \{i^*\}, y_j : j \in [\ell]\setminus \{j^*\}\}$ with 
\begin{align*}
\E_{x_{i^*}, y_{j^*}}\Bigg[(f(x_{i^*}, y_{j^*}) - (1-\eps)\alpha) \prod_{j : (i^*,j) \in E(G^{(t)})} f(x_{i^*}, y_j) \prod_{i : (i, j^*) \in E(G^{(t)})} f(x_i, y_{j^*})\Bigg]
&\ge \eps\alpha^{O(k+\ell)}.
\end{align*}
Now we can define $g_1(x) = \prod_{j : (i^*,j) \in E(G^{(t)})} f(x, y_j)$ and $g_2(y) = \prod_{i : (i, j^*) \in E(G^{(t)})} f(x_i, y)$, so that the above equation gives
\[ \E_{x,y}[f(x,y)g_1(x)g_2(y)] \ge (1-\eps)\alpha \E_x[g_1(x)]\E_y[g_2(y)] + \eps\alpha^{O(k+\ell)}. \]
This implies the first conclusion. The second conclusion follows because
\[ \E_x[g_1(x)]\E_y[g_2(y)] \ge \E_{x,y}[f(x,y)g_1(x)g_2(y)] \ge \eps\alpha^{O(k+\ell)}. \qedhere\]
\end{proof}

Now we are ready to state our main relative sifting statement. There are two elements to note. The first is that the function $f$ is supported on a pseudorandom set $T$. This allows the expectations of the sifted functions to have \emph{no} dependence on the parameter $\tau$, which one should morally view as the density of $T$ in its universe. Such a conclusion would be impossible to obtain via the standard approach to sifting. Secondly, these resulting functions are asymmetric; only one of them has a dependence on a parameter $k$. Somewhat surprisingly, leveraging this asymmetry will eventually be critical in our application.

\begin{theorem}[Relative sifting]
\label{thm:quasisifting2}

Let $\alpha, \eps, \gamma, \tau \in (0, 1)$ be parameters and $k$ a positive integer, satisfying that
\[ \gamma \le (\alpha\tau)^{O(\eps^{-2}k\log(1/\alpha)^2 + \eps^{-1}k\log(1/\tau))}. \] Then the following holds.

Let $T \subseteq \Omega_1 \times \Omega_2$ be $(\tau, \gamma)$-combinatorially spread, and let $f: \Omega_1 \times \Omega_2 \to [0, 1]$ be a function supported on $T$. Suppose that
\[ \|f\|_{G(2,k)}^{2k} = \E_{\substack{x_1,x_2 \in \Omega_1 \\ y_1,\dots,y_k \in \Omega_2}}\left[\prod_{i=1}^k f(x_1,y_i)f(x_2,y_i) \right] \ge \alpha^{2k}\tau^{2k}. \]
Then there are functions $g_1: \Omega_1 \to [0, 1]$ and $g_2: \Omega_2 \to [0, 1]$ such that \[ \E_{x\in \Omega_1, y\in \Omega_2}[f(x,y)g_1(x)g_2(y)] \ge (1-\eps)\alpha\tau \E_{x\in \Omega_1}[g_1(x)] \E_{y\in \Omega_2}[g_2(y)] \] and
\[\E_{x\in \Omega_1}[g_1(x)] \ge (\eps\alpha/2)^{O(\eps^{-1}k^2 \log(1/\alpha))} \enspace \text{ and } \enspace \E_{y \in \Omega_2}[g_2(y)] \ge (\eps\alpha/2)^{O(\eps^{-1}\log(1/\alpha))}. \]
\end{theorem}
\begin{remark}
In this work we only prove relative sifting for $(2,k)$ norms. It may be of interest to prove relative sifting for $(\ell,k)$ norms where $\gamma \le e^{-(\eps^{-1}\log(1/\alpha)k\ell)^{O(1)}}$ is sufficient quasi-randomness and the associated sides have sizes $e^{-(\eps^{-1}\log(1/\alpha)k)^{O(1)}}$, $e^{-(\eps^{-1}\log(1/\alpha)\ell)^{O(1)}}$.
\end{remark}
\begin{proof}
We proceed by induction on $k$. Define the function $F: \Omega_1 \times \Omega_1 \to [0, 1]$ as
\[ F(x_1, x_2) \coloneqq \E_{y_1, \dots, y_{k-1} \in \Omega_2}\left[\prod_{i=1}^{k-1} f(x_1, y_i)f(x_2, y_i)\right]. \]
By this definition, we can express the $(2, k)$-grid norm of $f$ as
\[ \|f\|_{G(2,k)}^{2k} = \E_{\substack{x_1, x_2 \in \Omega_1 \\ y \in \Omega_2}}[F(x_1, x_2)f(x_1, y)f(x_2, y)]. \]
If $\E_{x \in \Omega_1, y \in \Omega_2}f(x, y) \ge (1-\eps)\alpha\tau$ then we are done. Otherwise, let $\ell = 100\lceil \log(1/\alpha)\eps^{-1} \rceil$ and apply H\"{o}lder's inequality to get
\begin{align*}
    \alpha^{2\ell k}\tau^{2\ell k} &\le \left(\E_{\substack{x_1, x_2 \in \Omega_1 \\ y \in \Omega_2}}[F(x_1, x_2)f(x_1, y)f(x_2, y)] \right)^{\ell} \\
    &\le \left(\E_{x_2 \in \Omega_1, y \in \Omega_2} f(x_2, y) \right)^{\ell-1}\left(\E_{x_2 \in \Omega_1, y \in \Omega_2} f(x_2, y) \left(\E_{x_1 \in \Omega_1} F(x_1, x_2)f(x_1, y) \right)^{\ell} \right) \\
    &\le ((1-\eps)\alpha\tau)^{\ell-1} \left(\E_{x_2 \in \Omega_1, y \in \Omega_2} f(x_2, y) \left(\E_{x_1 \in \Omega_1} F(x_1, x_2)f(x_1, y) \right)^{\ell} \right).
\end{align*}

Using $f \le T$ and applying the $(\tau, \gamma)$-combinatorial spreadness of $T$ gives:
\begin{align*}
(1+\eps)^{\ell-1}\alpha^{(2k-1)\ell + 1}\tau^{(2k-1)\ell + 1} &\le \E_{x_2 \in \Omega_1, y \in \Omega_2} f(x_2, y) \left(\E_{x_1 \in \Omega_1} F(x_1, x_2)f(x_1, y) \right)^{\ell} \\
&\le \tau \E_{x_2 \in \Omega_1, y \in \Omega_2} \left(\E_{x_1 \in \Omega_1} F(x_1, x_2)f(x_1, y) \right)^{\ell} + \gamma.
\end{align*}
For our choice of $\gamma$ and $\ell$, we conclude that
\begin{align*}
(1+\eps/2)^{\ell}\alpha^{(2k-1)\ell}\tau^{(2k-1)\ell} &\le \E_{x_2 \in \Omega_1, y \in \Omega_2} \left(\E_{x_1 \in \Omega_1} F(x_1, x_2)f(x_1, y) \right)^{\ell} \\
&= \E_{x_1^{(1)}, \dots, x_1^{(\ell)} \in \Omega_1} \left(\E_{x_2 \in \Omega_1} \prod_{i=1}^{\ell} F(x_1^{(i)}, x_2) \right)\left(\E_{y \in \Omega_2} \prod_{i=1}^{\ell} f(x_1^{(i)}, y)\right) \\
&\le \left(\E_{x_1^{(1)}, \dots, x_1^{(\ell)} \in \Omega_1} \E_{y \in \Omega_2} \prod_{i=1}^{\ell} f(x_1^{(i)}, y) \right)^{\frac{\ell-1}{\ell}} \\
\cdot ~& \left(\E_{x_1^{(1)}, \dots, x_1^{(\ell)} \in \Omega_1} \left(\E_{y \in \Omega_2} \prod_{i=1}^{\ell} f(x_1^{(i)}, y)\right) \left(\E_{x_2 \in \Omega_1} \prod_{i=1}^{\ell} F(x_1^{(i)}, x_2) \right)^{\ell} \right)^{\frac{1}{\ell}} \\
&\le \|f\|_{G(\ell, 1)}^{\ell-1} \cdot \left(\tau^{\ell}\E_{x_1^{(1)}, \dots, x_1^{(\ell)} \in \Omega_1} \left(\E_{x_2 \in \Omega_1} \prod_{i=1}^{\ell} F(x_1^{(i)}, x_2) \right)^{\ell} + O(\gamma) \right)^{\frac{1}{\ell}} \\
&= \|f\|_{G(\ell, 1)}^{\ell-1} \left(\tau^{\ell} \|F\|_{G(\ell,\ell)}^{\ell^2} + O(\gamma) \right)^{\frac{1}{\ell}}.
\end{align*}
Here, the second inequality is H\"{o}lder's inequality, and the third uses that $f \le T$, $T$ is $(\tau, \gamma)$-combinatorially spread, and the counting lemma (\Cref{lemma:counting}).

We first handle the case where $\|f\|_{G(\ell,1)} \ge \alpha\tau$. Define $d(y) = \E_{x \in \Omega_1} f(x, y)$ and note that $\|f\|_{G(\ell,1)}^\ell = \E_{y \in \Omega_2} d(y)^\ell$. Define $S_1, S_2 \subseteq \Omega_2$ as $S_1 = \{y : d(y) \ge (1-\eps)\alpha\tau\}$ and $S_2 = \{y : d(y) \ge 2\tau\}$. We want to prove that $|S_1|/|\Omega_2| \ge (\alpha/2)^{O(\ell)}$, as then
\[ \E_{x \in \Omega_1, y \in \Omega_2}[f(x, y)\mbm{1}_{S_1}(y)] \ge (1-\eps)\alpha\tau \E_{y \in \Omega_2}[\mbm{1}_{S_1}(y)] \] by definition. In particular, the theorem would follow by setting $g_1=1$ and $g_2 = \mbm{1}_{S_1}$. To prove this, observe that the combinatorial spreadness of $T$ gives that $|S_2| \le \frac{\gamma}{\tau}|\Omega_2|$. Thus,
\begin{align*}
    (\alpha\tau)^\ell &\le \|f\|_{G(\ell,1)}^\ell = \E_{y \in \Omega_2} d(y)^\ell \le ((1-\eps)\alpha\tau)^\ell + \E_{y \in \Omega_2}\left[\mbm{1}_{y \in S_1 \setminus S_2} d(y)^\ell \right] + \E_{y \in \Omega_2}\left[\mbm{1}_{y \in S_2} d(y)^\ell \right] \\
    &\le ((1-\eps)\alpha\tau)^\ell + \frac{|S_1|}{|\Omega_2|} (2\tau)^\ell + \frac{\gamma}{\tau}.
\end{align*}
Using the choice of $\gamma$, this rearranges to $\frac{|S_1|}{|\Omega_2|} \ge (\alpha/2)^{O(\ell)}$ as desired.

Otherwise, we assume $\|f\|_{G(\ell,1)} \le \alpha\tau$, and this implies that
\[ \tau^\ell \|F\|_{G(\ell,\ell)}^{\ell^2} + O(\gamma) \ge (1+\eps/2)^{\ell^2} \alpha^{(2k-2)\ell^2 + \ell} \tau^{(2k-2)\ell^2 + \ell}, \] which for the choice of $\gamma$ and $\ell$ rearranges to
\[ \|F\|_{G(\ell,\ell)} \ge (1 + \eps/4)(\alpha\tau)^{2k-2}. \]
The idea is to now use non-relative sifting (\cref{thm:sift}) to get a correlation of $F$ onto a subset $\Omega_1' \subseteq \Omega_1$. However, done na\"{i}vely this will lose factors of $\tau$ in the size of $\Omega_1'$, which is unacceptable. To remedy this, we first argue that $F$ is basically bounded by $M \coloneqq \alpha^{-k}\tau^{2k-2}$. Formally, define the function $\wt{F} : \Omega_1 \times \Omega_1 \to [0, 1]$ as $\wt{F}(x_1, x_2) = M^{-1} \min\{F(x_1, x_2), M\}$. 

We will prove that
\begin{equation}
\Pr_{x_1, x_2 \in \Omega_1}\left[M \wt{F}(x_1, x_2) \neq F(x_1, x_2)\right] \le \gamma^{\Omega\left(\frac{\log(1/\alpha)}{\log(1/\tau)}\right)}. \label{eq:fvf}
\end{equation}
It suffices to bound $\Pr_{x_1, x_2 \in \Omega_1}\left[F(x_1, x_2) > M\right]$. Let $t$ be the maximum positive integer such that $\gamma \le \tau^{(2k-2)t}$. Markov's inequality tells us
\[ \Pr_{x_1, x_2 \in \Omega_1}\left[F(x_1, x_2) > M\right] \le M^{-t} \E_{x_1, x_2}\left[F(x_1, x_2)^t\right]. \]
By the definition of $F$, $f \le T$, and the counting lemma, we know that
\[ \E_{x_1, x_2}\left[F(x_1, x_2)^t\right] = \|f\|_{G(2,(k-1)t)}^{(2k-2)t} \le \tau^{(2k-2)t} + O(\gamma) \le 2\tau^{(2k-2)t}, \]
by applying combinatorial spreadness of $T$ and the choice of $\gamma$. Thus,
\[ \Pr_{x_1, x_2 \in \Omega_1}\left[F(x_1, x_2) > M\right] \le M^{-t} \cdot 2\tau^{(2k-2)t} \le 2\alpha^{kt} \le \gamma^{\Omega\left(\frac{\log(1/\alpha)}{\log(1/\tau)}\right)}. \]

By \eqref{eq:fvf} and the triangle inequality (which we may apply, since our choice of $\ell$ is even, so $\|\cdot\|_{G(\ell, \ell)}$ is a seminorm \cite{Hat10}), we get that
\begin{align*}
M\|\wt{F}\|_{G(\ell,\ell)} &\ge \|F\|_{G(\ell,\ell)} - \|F - M\wt{F}\|_{G(\ell,\ell)} \\
&\ge (1+\eps/4)(\alpha\tau)^{2k-2} - \gamma^{\Omega\left(\frac{\log(1/\alpha)}{\ell\log(1/\tau)}\right)} \ge (1+\eps/8)(\alpha\tau)^{2k-2},
\end{align*} 
for the choice of $\gamma$. Here, we have used the fact that if a function $G: \Omega_1 \times \Omega_1 \to [-1, 1]$ is supported on at most a $\gamma'$ fraction of $\Omega_1 \times \Omega_1$, then $\|G\|_{G(\ell,\ell)} \le (\gamma')^{1/\ell}$, applied for $G = F - M\wt{F}$. Indeed,
\[ \|G\|_{G(\ell,\ell)}^{\ell^2} = \left| \E_{\substack{x_1,\dots,x_{\ell} \in \Omega_1 \\ y_1, \dots, y_{\ell} \in \Omega_1}}\Big[\prod_{i=1}^{\ell} \prod_{j=1}^{\ell} G(x_i,y_j) \Big] \right| \le \E_{\substack{x_1,\dots,x_{\ell} \in \Omega_1 \\ y_1, \dots, y_{\ell} \in \Omega_1}}\Big[\prod_{i=1}^{\ell} |G(x_i,y_i)|\Big] \le (\gamma')^{\ell}. \]
Thus,
\[ \|\wt{F}\|_{G(\ell,\ell)} \ge M^{-1}(1+\eps/8)(\alpha\tau)^{2k-2} = (1+\eps/8)\alpha^{3k-2}. \]
By \cref{thm:sift} and \cref{lem:extract-cor}, we conclude that there are functions $g_1: \Omega_1 \to \{0, 1\}$ and $g_2 : \Omega_1 \to \{0, 1\}$ such that 
\[ \E_{x \in \Omega_1}[g_i(x)] \ge \left(\frac{\eps\alpha^{3k-2}}{2} \right)^{O(\ell)} \enspace \text{ for } \enspace i = 1, 2, \]
and
\[ \E_{x_1, x_2 \in \Omega_1}\left[\wt{F}(x_1, x_2)g_1(x_1)g_2(x_2) \right] \ge \alpha^{3k-2} \E_{x \in \Omega_1}[g_1(x)]\E_{x \in \Omega_1}[g_2(x)]. \]
Thus,
\begin{align*}
\E_{x_1, x_2 \in \Omega_1}\left[F(x_1, x_2)g_1(x_1)g_2(x_2) \right] &\ge M\alpha^{3k-2} \E_{x \in \Omega_1}[g_1(x)]\E_{x \in \Omega_1}[g_2(x)] \\
&= (\alpha\tau)^{2k-2} \E_{x \in \Omega_1}[g_1(x)]\E_{x \in \Omega_1}[g_2(x)].
\end{align*}
Note that $F$, viewed as a matrix on $\Omega_1 \times \Omega_1$, is symmetric and positive semi-definite. Thus if we define $h_i(x)=g_i(x)/\E[g_i]$ then
\begin{align*}
0 &\le \E_{x_1, x_2 \in \Omega_1}[F(x_1,x_2)(h_1(x_1)-h_2(x_1))(h_1(x_2)-h_2(x_2))] \\
&= \E_{x_1, x_2 \in \Omega_1}\left[F(x_1, x_2)h_1(x_1)h_1(x_2) \right] + \E_{x_1, x_2 \in \Omega_1}\left[F(x_1, x_2)h_2(x_1)h_2(x_2) \right] - 2 \E_{x_1, x_2 \in \Omega_1}\left[F(x_1, x_2)h_1(x_1)h_2(x_2) \right]
\end{align*}
gives us that either
\begin{align*} &\E_{x_1, x_2 \in \Omega_1}\left[F(x_1, x_2)g_1(x_1)g_1(x_2) \right] \ge (\alpha\tau)^{2k-2}\E_{x \in \Omega_1}[g_1(x)]^2, \enspace \text{ or } \\  ~&\E_{x_1, x_2 \in \Omega_1}\left[F(x_1, x_2)g_2(x_1)g_2(x_2) \right] \ge (\alpha\tau)^{2k-2}\E_{x \in \Omega_1}[g_2(x)]^2.
\end{align*}
Without loss of generality, assume the former holds. Let \[ \Omega_1' = \{x \in \Omega_1 : g_1(x) = 1\} \] and define $f': \Omega_1' \times \Omega_2 \to [0, 1]$ as the restriction of $f$, i.e., $f'(x, y) = f(x, y)$. Then by definition,
\[ \|f'\|_{G(2,k-1)}^{2k-2} = \E_{x_1 \in \Omega_1', x_2 \in \Omega_1'}[F(x_1,x_2)] \ge (\alpha\tau)^{2k-2}. \] Now, the result follows by induction on $f'$ and the lower bound we have on $|\Omega_1'|/|\Omega_1| = \E_{x \in \Omega_1}[g_1(x)]$.
\end{proof}

\section{Additional Tools}\label{sec:add-tool}

In this section we reproduce several combinatorial tools from the works of Kelley-Meka and Kelley-Lovett-Meka, e.g. \cite[Theorem~2.8]{KLM24}. The main result is a key spectral positivity input that we need. We remark that there is a minor issue in the spectral positivity argument as given in \cite[Theorem~2.8]{KLM24} and hence we provide complete proofs here. These give graph theoretic interpretations of two of the key steps in breakthrough work of Kelley and Meka on $3$-term arithmetic progressions \cite{KM23}.

We now give the key spectral positivity argument. The first step is the following unbalancing inequality (see \cite[Proposition~D.1]{KM23}). We include the proof for completeness.
\begin{lemma}\label{lem:nonnegative-real}
Let $\eps\in (0,1/10)$, $k$ be a positive integer, and $p = 6\lceil k/\eps\rceil$. Let $X$ be a real random variable such that $\E[X^{k}] \ge \eps^{k}$ and $\E[X^{r}]\ge 0$ for all $r\in \mb{Z}^{\ge 0}$.
Then 
\[\E[(X+1)^p]\ge (1+\eps/2)^{p}.\]  
\end{lemma}
\begin{proof}
Using $\E[X^{r}]\ge 0$ for all $r\in \mb{Z}^{\ge 0}$ and H\"{o}lder's inequality, we have
\begin{align*}
\E[(X + 1)^p] &=\sum_{\ell = 0}^{p}\binom{p}{\ell}\E[X^{\ell}]\ge \sum_{\ell\equiv 0\imod 2}\binom{p}{\ell} \cdot \E[X^{\ell}]\\
&\ge \sum_{\substack{\ell\equiv 0\imod 2\\\ell\ge k}}\binom{p}{\ell}\cdot (\E[X^k])^{\ell/k}\ge \sum_{\substack{\ell\equiv 0\imod 2\\\ell\ge k}}\binom{p}{\ell}\cdot \eps^{\ell}\\
&=(1+\eps)^{p} \cdot \sum_{\substack{\ell\equiv 0\imod 2\\\ell\ge k}}\binom{p}{\ell}\cdot \Big(\frac{\eps}{1+\eps}\Big)^{\ell} \cdot \Big(\frac{1}{1+\eps}\Big)^{p-\ell}.
\end{align*}
Let $\eps' = \eps/(1+\eps)$ and $\on{Bin}(n,q)$ denote a binomial random variable with $n$ independent trials where each has success probability $q$. Via chasing definitions, we have that 
\[\sum_{\substack{\ell\equiv 0\imod 2\\\ell\ge k}} \binom{p}{\ell} \cdot \Big(\frac{\eps}{1+\eps}\Big)^{\ell}\Big(\frac{1}{1+\eps}\Big)^{p-\ell} = \Pr_{X\sim \on{Bin}(p,\eps')}[X\ge k\wedge X\equiv 0\imod 2].\]
For the sake of simplicity, we omit that $X\sim  \on{Bin}(p,\eps')$ below. Note that 
\begin{align*}
\Pr[X\equiv 0\imod 2] &= \frac{1}{2} + \frac{\Pr[X\equiv 0\imod 2] - \Pr[X\equiv 1\imod 2]}{2} = \frac{1}{2} + \frac{\sum_{\ell=0}^{p}(-1)^\ell \binom{p}{\ell} (\eps')^\ell(1-\eps')^{p-\ell}}{2}\\
&=\frac{1}{2} +\frac{(1-2\eps')^{p}}{2}\ge \frac{1}{2}.
\end{align*}
Thus 
\[\Pr[X\ge k\wedge X\equiv 0\imod 2] \ge \frac{1}{2} - \Pr[X<k].\]
Observe that $\E[X] = p \cdot \eps'\ge 4k$. Furthermore we have that 
\[\on{Var}[X]= p \cdot \eps' \cdot (1-\eps') \le \frac{7k}{(1+\eps)^2}\le 7k.\]
Thus by Chebyshev's inequality, we have that 
\[\Pr[X<k]\le \frac{\on{Var}[X]}{(\E[X]-k)^2}\le \frac{7k}{(3k)^2}\le \frac{1}{k} \le \frac{1}{6}.\]
Thus we have that
\begin{align*}
\E[(X + 1)^p] &\ge\frac{1}{3} \cdot (1+\eps)^{p}\ge(1+\eps/2)^{p} \cdot \frac{(1+\eps/3)^{6/\eps}}{3}\ge (1+\eps/2)^{p}.\qedhere
\end{align*}
\end{proof}

We now complete the proof of the necessary spectral positivity argument. 
\begin{lemma}\label{lem:spectral-pos}
Let $\eps\in (0,1/10)$, $k$ be an even positive integer, and $p = 36\lceil k/\eps^4 \rceil$. Consider $A:\Omega_1\times \Omega_2\to \mb{R}^{+}$, and define $\alpha \coloneqq \|A\|_{G(1,1)} = \E_{x\in\Omega_1, y \in \Omega_2}[A(x, y)]$. Suppose that
\[\snorm{A-\alpha}_{G(2,k)}\ge \eps \alpha \quad\text{and}\quad\inf_{x}\E_{y}[A(x,y)-\alpha]\ge -\eps^2 \alpha/36.\]
Then 
\[\snorm{A}_{G(2,p)}\ge (1+\eps^2/36) \alpha.\]  
\end{lemma}

\begin{proof}
By normalizing, we may assume that $\alpha = 1$. We prove that if
\[\snorm{A}_{G(2,p)}\le 1+\eps^2\quad\text{and}\quad\inf_{x}\E_{y}[(A(x,y)-1)]\ge -\eps^2\]
then 
\[\snorm{A-1}_{G(2,k)}\le 6\eps.\]

Let $q = 6\lceil k /\eps^2 \rceil$. We first handle the case that
\[\snorm{A-1}_{G(1,q)}\ge 7 \eps^2.\]
In this case, observe that
\[\snorm{A-1}_{G(1,q)}^{q} = \E_{x}(\E_{y}A(x,y)-1)^{q};\]
therefore if we define $F(x) = \E_{y}A(x,y)-1$ then $\snorm{F}_{q}\ge 7 \eps^2$. Let 
\[F(x) = \max(F(x), 0) + \min(F(x), 0) \reflectbox{ $\coloneqq$ } F^+(x) + F^-(x)\]
and observe that $|F^{-}(x)|\le \eps^2$ by our earlier assumption. Hence by the triangle inequality $\snorm{F^{+}}_{q}\ge 6 \eps^2$ and by definition $F^{+}\ge 0$. Thus via \cref{lem:nonnegative-real}, we have that 
\[\snorm{F^{+} + 1}_{p}\ge 1+3\eps^2\]
and therefore by the triangle inequality again, \[\snorm{F + 1}_{p}\ge 1+2\eps^2.\]
However by the monotonicity of grid norms (see, e.g., \cite[Claim 4.2]{KLM24})
\[1+2\eps^2 \le \snorm{F + 1}_{p} = \snorm{A}_{G(1,p)}\le \snorm{A}_{G(2,p)},\]
yielding the desired contradiction.

Thus we may assume for the remainder of the proof that
\[\snorm{A-1}_{G(1,q)}\le 7 \eps^2.\]
Observe that 
\begin{equation}\label{eq:double_A_expanded}
    \E_{y}A(x,y)A(x',y) = \E_{y} \left[ (A(x,y)-1)(A(x',y)-1) + 1 + (A(x,y)-1) + (A(x',y)-1)\right].
\end{equation}
Define
\[B(x,x') \coloneqq \E_{y}\left[(A(x,y)-1)(A(x',y)-1)\right] \]
and observe that for any integer $r \ge 1$ it satisfies $\E[B^r] \ge 0$ since
\[\E_{x,x'}[B(x,x')^{r}] = 
\E_{y_1,\ldots,y_r} \E_{x,x'} \prod_{j=1}^r (A(x,y_j)-1) (A(x',y_j)-1) = 
\E_{y_1,\ldots,y_r}\Big(\E_{x}\prod_{j=1}^{r}(A(x,y_j)-1)\Big)^2\ge 0.\]
Now assume by contradiction that 
\[\snorm{A-1}_{G(2,k)} = \E_{x,x'}[B(x,x')^{k}]^{1/(2k)} > 6\eps.\]
That is, we have $\|B\|_k \ge 36 \eps^2$.
We may apply \cref{lem:nonnegative-real} to obtain that 
\[\|B+1\|_q = \E_{x,x'}[(B(x,x') + 1)^{q}]^{1/q}\ge 1 + 18\eps^2.\]
Via \Cref{eq:double_A_expanded} and the triangle inequality (as $\snorm{A-1}_{G(1,q)}\le 7 \eps^2$), we obtain that
\begin{align*}
    \snorm{A}_{G(2,q)}^2 &= \E_{x,x'}[(\E_{y}A(x,y)A(x',y))^{q}]^{1/q} \\
    &= \E_{x,x'}[(\E_{y}(A(x,y)-1)(A(x',y)-1) + 1 + (A(x,y)-1) + (A(x',y)-1))^{q}]^{1/q} \\
    &\ge \|B+1\|_q - 2 \|A-1\|_{G(1,q)} \ge 1 + 18\eps^2 - 2(7\eps^2) \ge 1 + 4\eps^2.
\end{align*}
By monotonicity of grid norms, we get
\[\snorm{A}_{G(2,p)}\ge \snorm{A}_{G(2,q)} \ge (1+4\eps^2)^{1/2} > 1+\eps^2\]
which gives the desired contradiction. 
\end{proof}

A fact we use several times throughout the work is the ``reverse Markov inequality''.
\begin{fact}\label{fct:rev-mark}
Let $\rho,\gamma\in (0,1)$. Let $V$ be a random variable with $V\le (1+\rho) \E[V]$. Then 
\[\Pr[V\le (1-\gamma)\E[V]]\le \frac{\rho}{\gamma + \rho}.\]
\end{fact}
\begin{proof}
Define $V' = (1+\rho)\E[V] - V$. Observe that $V'\ge 0$ and $\E[V'] = \rho\E[V]$. Thus 
\[\Pr[V\le (1-\gamma)\E[V]] = \Pr[V'\ge (\gamma + \rho)\E[V]] \le \frac{\rho}{\gamma + \rho}.\qedhere\]
\end{proof}

\section{The Finite Field Case}\label{sec:finite-field}

\subsection{General setup}
\label{sec:setup}
Throughout our analysis we will consider a linear subspace $W \subseteq \mb{F}_2^{n}$. (Although in occasional convenient contexts, we will let $W$ be an affine subspace.) We will additionally consider subsets $X, Y, D \subseteq W$, and define our \emph{container set} to be
\[ S(X, Y, D) \coloneqq \{(x,y) \in W \times W : x\in X,~y\in Y,~x+y\in D\}\subseteq W\times W.\]
The sets $X, Y$ will sometimes be pseudorandom in the analysis, according to the definition of \emph{algebraically spread} in \cref{def:algpseudo}. At the start of the density increment phase, both $X$ and $Y$ will be $(r, \eps_s)$-algebraically spread for well-chosen parameters $r$ and $\eps_s$. At the start of the pseudorandomization phase, we will only guarantee that at least one of $X, Y$ is $(r, \eps_s)$-algebraically spread, and after the pseudorandomization, both $X, Y$ will be $(r, \eps_s)$-algebraically spread.

\subsection{Algebraic spreadness}

We start by defining a notion of algebraic pseudorandomness.
\begin{definition}[Algebraic spreadness]
\label{def:algpseudo}
Let $W \subseteq \mb{F}_2^n$ be an affine subspace. We say that a subset $X \subseteq W$ is \emph{$(r, \eps)$-algebraically spread} within $W$ if for all affine subspaces $W' \subseteq W$ satisfying $\dim(W') \ge \dim(W) - r$, it holds that
\[ \frac{|X \cap W'|}{|W'|} \le (1+\eps) \frac{|X|}{|W|}. \]
\end{definition}

The following results of Kelley and Meka \cite{KM23} will prove useful throughout the subsection.
\begin{restatable}[{\cite[Theorem 4.10]{KM23}}]{lemma}{KMFourTen}\label{lem:KM_4.10}
    Let $\eps, \tau, \gamma \in (0,1/2)$, $r \ge \Omega(\eps^{-7}\log(1/\tau)\log(1/\gamma)^7)$ a large enough integer, and $A \subseteq \mb{F}_2^n$ of size $|A| = \tau|\mb{F}_2^n|$. If $A$ is $(r,\eps/8)$-algebraically spread, then for all $B,C \subseteq \mb{F}_2^n$ of size at least $\gamma |\mb{F}_2^n|$, we have
    \[
        \langle \mbm{1}_B \ast \mbm{1}_C, \mbm{1}_A \rangle \le (1+\eps)\E[\mbm{1}_A] \E[\mbm{1}_B]\E[\mbm{1}_C].
    \]
\end{restatable}

\begin{restatable}[{\cite[Proposition 2.16]{KM23}}]{lemma}{KMTwoSixteen}\label{lem:KM_2.16}
    Let $\eps,\tau \in (0,1/2), k \ge 1$ an integer, $r \ge \Omega(\eps^{-7}\log(1/\tau)^4 k^4)$ a sufficiently large integer, and $A,B \subseteq \mb{F}_2^n$ of size at least $\tau|\mb{F}_2^n|$.
    If $A$ and $B$ are both $(r, \eps/8)$-algebraically spread, then 
    \[
        \left\|\frac{\mbm{1}_A}{\E[\mbm{1}_A]} \ast \frac{\mbm{1}_B}{\E[\mbm{1}_B]} - 1\right\|_k \le \eps.
    \]
\end{restatable}

We require a lemma which relates algebraic spreadness to combinatorial spreadness in the case of corners.
\begin{lemma}
\label{lem:relate}
Let $\eps, \tau, \gamma \in (0,1/2)$, and $r \ge \Omega(\eps^{-7}\log(1/\tau)\log(1/\gamma)^7)$ a large enough integer.
Let $W \subseteq \mb{F}_2^n$ be a linear subspace, and let $D \subseteq W$ with $\tau = \frac{|D|}{|W|}$. Define $T \subseteq W \times W$ as $T \coloneqq \{(x, y) : x+y \in D\}$. If $D$ is $(r, \eps/8)$-algebraically spread within $W$ for $r \ge \Omega(\eps^{-7}\log(1/\tau)\log(1/\gamma)^7)$, then T is $((1+\eps)\tau, \gamma)$-combinatorially spread.
\end{lemma}
\begin{proof}
Suppose that $T$ is not $((1+\eps)\tau, \gamma)$-combinatorially spread. Then there exist functions $f,g: W\to [0,1]$ such that 
\[\E_{x,y\in W}[f(x)g(y)\mbm{1}_{T}(x,y)]>(1+\eps)\tau \cdot \E_{x\in W}[f(x)] \cdot \E_{x\in W}[g(x)] + \gamma>(1+\eps)\tau \cdot \E_{x\in W}[f(x)] \cdot \E_{x\in W}[g(x)].\]
Note that $f$ and $g$ are $1$-bounded, and therefore for the first inequality to hold we must have that $\E[f]\ge \gamma$ and $\E[g]\ge \gamma$. 

By applying \cref{lem:extract-cor}, there exist boolean functions $F,G: W\to \{0,1\}$ such that 
\[\E_{x,y\in W}[F(x)G(y)\mbm{1}_{T}(x,y)]>(1+\eps)\tau \cdot \E_{x\in W}[F(x)] \cdot \E_{x\in W}[G(x)]\]
where $\E[F]\ge \gamma/2$ and $\E[G]\ge \gamma/2$. Recalling that $\mbm{1}_{T}(x,y) = \mbm{1}_D(x+y)$, we have that 
\[\E_{x,y\in W}[F(x)G(y)\mbm{1}_{D}(x+y)]>(1+\eps)\tau \cdot \E_{x\in W}[F(x)] \cdot \E_{x\in W}[G(x)].\]
This however contradicts \cref{lem:KM_4.10}.
\end{proof}

An upper bound on the size of $S(X, Y, D)$ follows from one of $X$ or $Y$ being algebraically spread.
\begin{lemma}
\label{lem:relatesxyd}
    Let $X, Y, D \subseteq W$ for a linear subspace $W \subseteq \mb{F}_2^n$, and let $\delta = |W|^{-3}|X||Y||D|$. If $X$ or $Y$ is $(r, \eps/8)$-algebraically spread for $r \ge \Omega(\eps^{-7}\log(1/\delta)^8)$, then $|S(X, Y, D)| \le (1+\eps)\delta |W|^2$.
\end{lemma}
\begin{proof}
    We focus on the case where $X$ is spread; the $Y$ case is identical. Observe that
    \begin{align*}
        |S(X,Y,D)| = \sum_{x,y \in W} \mbm{1}_X(x) \mbm{1}_Y(y) \mbm{1}_D(x+y) &= |W|^2\mathop{\E}_{x\in W}\left[\mbm{1}_X(x)\mathop{\E}_{y\in W} \left[\mbm{1}_Y(y) \mbm{1}_D(x+y)\right]\right] \\
        &= |W|^2 \langle \mbm{1}_X, \mbm{1}_Y\ast \mbm{1}_D \rangle.
    \end{align*}
    We conclude by observing that \Cref{lem:KM_4.10} yields $\langle \mbm{1}_X, \mbm{1}_Y\ast\mbm{1}_ D \rangle \le (1+\eps) \E[\mbm{1}_X] \E[\mbm{1}_Y] \E[\mbm{1}_D] = (1+\eps)\delta$.
\end{proof}

If additionally both $X$ and $Y$ are algebraically spread, then the size of $S(X, Y, D)$ is also lower--bounded.

\begin{lemma}
\label{lem:2side}
Let $X, Y, D \subseteq W$ for a linear subspace $W \subseteq \mb{F}_2^n$, and let $\delta = |W|^{-3}|X||Y||D|$. If both $X$ and $Y$ are $(r, \eps/16)$-algebraically spread for $r \ge \Omega(\eps^{-7}\log(1/\delta)^8)$,
then it holds that
\[ (1-\eps)\delta|W|^2 \le |S(X, Y, D)| \le (1+\eps)\delta|W|^2. \]
\end{lemma}
\begin{proof}
    As above, we have $|S(X,Y,D)| = |W|^2 \langle \mbm{1}_X \ast \mbm{1}_Y, \mbm{1}_D \rangle.$ Then H\"{o}lder's inequality with $p = \lceil \log(1/\delta) \rceil$ and \Cref{lem:KM_2.16} give 
    \begin{align*}
        \left||S(X,Y,D)| - \delta |W|^2 \right| &= |W|^2 \left|\left\langle \mbm{1}_X \ast \mbm{1}_Y, \mbm{1}_D \right \rangle - \delta \right|\\
        &\le |W|^2 \left\| \mbm{1}_X \ast \mbm{1}_Y - \E[\mbm{1}_X]\E[\mbm{1}_Y] \right\|_p \cdot (\E[\mbm{1}_D])^{1-1/p} \le \eps \delta |W|^2. \qedhere
\end{align*}
\end{proof}

\subsection{Algebraic pseudorandomization}

In the subsequent subsections, we will show that under certain conditions, which include the algebraic spreadness of $X$ and $Y$, $A$ admits a density increment onto either some $S(X', Y, D')$ or $S(X, Y', D')$. However, the density increment may spoil the algebraic spreadness of (say) $X'$.
Thus, we require a procedure to pseudorandomize $X'$. This is precisely the goal of this section. Now we state the algebraic pseudorandomization theorem.

\begin{theorem}
\label{thm:algopseudo}
Let $X, Y, D \subseteq W$ for a linear subspace $W \subseteq \mb{F}_2^n$. Define $\delta_X = |X|/|W|$, $\delta_Y = |Y|/|W|$, $\delta_D = |D|/|W|$, and $\delta = \delta_X\delta_Y\delta_D$. Let $\eps>0$ and $r \ge \Omega(\eps^{-8}\log(1/(\alpha\delta))^8 + \eps^{-8}\log(1/\delta_D)^{16})$ be an integer, such that at least one of $X$ or $Y$ is $(r, \eps/8)$-algebraically spread. Let $A \subseteq S(X, Y, D)$ with $|A| \ge \alpha\delta|W|^2$. Then there is a linear subspace $W' \subseteq W$ with shifts $x,y \in W$ and sets $X'\subseteq W' + x, Y' \subseteq W' + y, D' \subseteq W' + x + y$, and $A' \coloneqq A \cap S(X', Y', D')$ satisfying:
\begin{enumerate}
\item $\dim(W') \ge \dim(W) - O\left(r\eps^{-2}\log(1/(\eps \alpha\delta))^2\log(1/(\eps\alpha\delta_D)) + r\eps^{-2}\log(1/(\eps\alpha\delta_D))^5\right)$, 
\item $\frac{|D'|}{|W'|} \ge (\eps\alpha/2)\delta_D$,
\item $\frac{|X'||Y'|}{|W'|^2} \geq  2^{-O(\log(1/(\eps\alpha\delta_D))^2)}\delta_X\delta_Y$,
\item $X'$ and $Y'$ are $(r, \eps)$-algebraically spread in $W' + x$ and $W' + y$, respectively, and
\item $|A'| \ge (1-5\eps)\alpha\delta_{X'}\delta_{Y'}\delta_{D'}|W'|^2$ where $\delta_{X'} = \frac{|X'|}{|W'|}$, $\delta_{Y'} = \frac{|Y'|}{|W'|}$, and $\delta_{D'} = \frac{|D'|}{|W'|}$.
\end{enumerate}
\end{theorem}
The first step towards proving \cref{thm:algopseudo} is to note that we can almost completely partition $X \times Y$ into subrectangles $X' \times Y'$ which are all algebraically spread (inside their respective subspaces) and cover all but an $\eta$-fraction of $X \times Y$. For this, it is useful to observe that for any set $X \subseteq W$, we can find a relatively large subset within an affine subspace that is algebraically spread.

\begin{claim}
\label{claim:spread}
Let $r$ be an integer, $W \subseteq \mb{F}_2^n$ be a linear subspace, and $\eps > 0$. For $X \subseteq W$ with $\delta_X \coloneqq |X|/|W|$, there is an affine subspace $W' \subseteq W$ such that $X' \coloneqq X \cap W'$ is $(r, \eps)$-algebraically spread within $W'$, $\dim(W') \ge \dim(W) - O(r\eps^{-1}\log(1/\delta_X))$, and $\frac{|X'|}{|W'|} \ge \delta_X$. 
\end{claim}
\begin{proof}
We proceed iteratively. Initialize $W^{(0)} = W$. If $X \cap W^{(0)}$ is $(r,\eps)$-algebraically spread within $W^{(0)}$, we are done. Otherwise, there must exist an affine subspace $W^{(1)} \subseteq W^{(0)}$ with $\dim(W^{(1)}) \ge \dim(W^{(0)}) - r$ satisfying
\[
    \frac{|X \cap W^{(1)}|}{|W^{(1)}|} > (1+\eps)\frac{|X \cap W^{(0)}|}{|W^{(0)}|}.
\]
We now repeat this process with $W^{(1)}$. After $i$ iterations, the density of $X \cap W^{(i)}$ within $W^{(i)}$ is at least $(1+\eps)^i\delta_X \ge \delta_X$ and $\dim(W^{(i)}) \ge \dim(W^{(0)}) - ri$. As the density may not exceed 1, we are guaranteed to obtain our desired affine subspace after $O(\eps^{-1}\log(1/\delta_X))$ iterations.
\end{proof}

Now we state and prove the key partitioning lemma over finite fields. 
\begin{lemma}
\label{lemma:2dim}
Let $r$ be an integer, $W \subseteq \mb{F}_2^n$ be a linear subspace, and $\eps, \eta \in (0,1/10)$. Additionally, let $X, Y \subseteq W$, and define $\delta_X = |X|/|W|$ and $\delta_Y = |Y|/|W|$. Then there is a positive integer $T$ and for all $i = 1, \dots, T$, a subspace $V_i \subseteq W$, points $x_i, y_i \in W / V_i$, and subsets $X_i \subseteq V_i + x_i, Y_i \subseteq V_i+y_i$ satisfying:
\begin{enumerate}
    \item $\dim(V_i) \ge \dim(W) - O(r \eps^{-2} \log(1/(\delta_X\delta_Y))^2\log(1/\eta) + r\eps^{-2}\log(1/\eta)^5)$ for $i = 1, \dots, T$.
    \item $X_i \times Y_i$ and $X_j \times Y_j$ are disjoint for $1 \le i \neq j \le T$.
    \item $|X \times Y \setminus \bigcup_{i=1}^T X_i \times Y_i| \le \eta|X||Y|$.
    \item $\frac{|X_i||Y_i|}{|V_i|^2} \ge 2^{-O(\log(1/\eta)^2)}\delta_X\delta_Y$ for all $i = 1, 2, \dots, T$. 
    \item $X_i$ and $Y_i$ are $(r, \eps)$-algebraically spread within $V_i+x_i$ and  $V_i+y_i$, respectively, for $i = 1, \dots, T$.
\end{enumerate}
\end{lemma}
Towards proving \cref{lemma:2dim}, we first provide a simpler one round partitioning result.
\begin{lemma}
\label{lemma:oneround2dim}
Let $r$ be an integer, $W \subseteq \mb{F}_2^n$ be a linear subspace, and $\eps, \eta \in (0,1/10)$. Additionally, let $X, Y \subseteq W$, and define $\delta_X = |X|/|W|$ and $\delta_Y = |Y|/|W|$. Then there is a positive integer $T$, subset $\mc{G} \subseteq [T]$, and for all $i = 1, \dots, T$, a subspace $V_i \subseteq W$, points $x_i, y_i \in W / V_i$, and subsets $X_i \subseteq V_i + x_i, Y_i \subseteq V_i+y_i$ satisfying:
\begin{enumerate}
    \item $\dim(V_i) \ge \dim(W) - O(r \eps^{-2} \log(1/(\delta_X\delta_Y\eta))^2)$ for $i = 1, \dots, T$. 
    \item $X_i \times Y_i$ and $X_j \times Y_j$ are disjoint for $1 \le i \neq j \le T$.
    \item $|X \times Y \setminus \bigcup_{i=1}^T X_i \times Y_i| \le \eta^2|X||Y|$.
    \item $\frac{|X_i|}{|V_i|} \ge \frac{\eta^4}{100}\delta_X$ and $\frac{|Y_i|}{|V_i|} \ge \frac{\eta^2}{10}\delta_Y$ for all $i = 1, 2, \dots, T$. 
    \item $X_i$ and $Y_i$ are $(r,\eps)$-algebraically spread within $V_i+x_i$ and $V_i+y_i$, respectively, for all $i \in \mc{G}$.
    \item $\sum_{i \in \mc{G}} |X_i||Y_i| \ge \frac12|X||Y|$.
\end{enumerate}
\end{lemma}
\begin{proof}
Let $X^{(0)} = X$ and perform the following algorithm. For $t = 0, 1, \dots$: if $|X^{(t)}| \le \frac{\eta^2}{10}\delta_X |W|$, then terminate. Otherwise, let $X_t \subseteq X^{(t)}$ be $(r',\eps/5)$-algebraically spread within an affine subspace $W_t+x_t$ for $r' = \Omega(r\eps^{-1}\log(1/(\delta_Y\eta)))$, as given by \cref{claim:spread}. Note that $|X_t| \ge \frac{\eta^2}{10}\delta_X |W_t|$ and
\[\dim(W_t) \ge \dim(W) - O(r'\eps^{-1}\log(1/(\delta_X \eta))). \]
Now let $X^{(t+1)} = X^{(t)} \setminus X_t$. Let $T$ be the total number of iterations, so that
\[ X = X^{(T)} \cup X_0 \cup X_1 \cup \dots \cup X_{T-1}, \]
where $|X^{(T)}| \le \frac{\eta^2}{10}|X|$.

Next, fix some $t \in \{0, 1, \dots, T-1\}$. For $x \in W / W_t$ define $Y_{t,x} \coloneqq Y \cap (W_t+x)$. Note that over $x \in W / W_t$, the $Y_{t,x}$ form a partition of $Y$. Now for each $x \in W / W_t$ further partition
\[ Y_{t,x} = Y_{t,x}^{(T'_{t,x})} \cup Y_{t,x,0} \cup \dots \cup Y_{t,x,T'_{t,x}-1}, \]
using the algorithm in the first paragraph, where $|Y_{t,x}^{(T'_{t,x})}| \le \frac{\eta^2}{10}\delta_Y |W_t|$ and each $Y_{t,x,t'}$ for $0 \le t' \le T'_{t,x}-1$ is $(r, \eps)$-algebraically spread within some affine subspace $W_{t,x,t'}+y_{t,x,t'}$ where
\[ \dim(W_{t,x,t'}) \ge \dim(W_t) - O(r\eps^{-1}\log(1/(\delta_Y\eta))). \]
In particular, the $Y_{t,x,t'}$ satisfy $\frac{|Y_{t,x,t'}|}{|W_{t,x,t'}|} \geq \frac{\eta^2}{10} \delta_Y$. Finally, for each $y \in W_t / W_{t,x,t'}$ define $X_{t,x,t',y} \coloneqq X_t \cap (W_{t,x,t'}+x_t + y)$, and note that for every fixed $t, x, t'$ that $X_{t,x,t',y}$ partition $X_t$ over $y \in W_t / W_{t,x,t'}$.

Now we define the pieces $X_i \times Y_i$ in the lemma statement and $\mc{G}$. The pieces $X_i \times Y_i$ are all pieces \[ X_{t,x,t',y} \times Y_{t,x,t'} \subseteq (W_{t,x,t'}+x_t + y) \times (W_{t,x,t'} + y_{t,x,t'}) \] where $\frac{|X_{t,x,t',y}|}{|W_{t,x,t'}|} \ge \frac{\eta^4}{100}\delta_X$. This combined with the density lower bound $\frac{|Y_{t,x,t'}|}{|W_{t,x,t'}|} \geq \frac{\eta^2}{10} \delta_Y$ gives guarantee (4). Finally, $\mc{G}$ consists of those pieces with $\frac{|X_{t,x,t',y}|}{|W_{t,x,t'}|} \ge (1-3\eps/5)\frac{|X_t|}{|W_t|}$.

Let us now verify the remaining conclusions. (1) follows because
\begin{align*}
\dim(W_{t,x,t'}) &\ge \dim(W_t) - O(r\eps^{-1}\log(1/(\delta_Y\eta))) \\
&\ge \dim(W) - O(r\eps^{-2}\log(1/(\delta_X\delta_Y\eta))^2) - O(r\eps^{-1}\log(1/(\delta_Y\eta))),
\end{align*}
as desired. (2) follows by construction. To check (3), first define $\mc{I}$ to be the set of tuples $(t,x,t',y)$ with $\frac{|X_{t,x,t',y}|}{|W_{t,x,t'}|} < \frac{\eta^4}{100}\delta_X$. Note that by the construction 
\begin{align*}
|X \times Y| - \sum_i |X_i \times Y_i| = |X^{(T)}||Y| + \sum_{t,x} |X_t||Y_{t,x}^{(T'_{t,x})}| + \sum_{(t,x,t',y) \in \mc{I}} |X_{t,x,t',y}||Y_{t,x,t'}|.
\end{align*}
We bound this term by term. By construction, $|X^{(T)}||Y| \le \frac{\eta^2}{10}|X||Y|$. Also, $|Y_{t,x}^{(T'_{t,x})}| \le \frac{\eta^2}{10}\delta_Y|W_t|$, so $\sum_{x \in W/W_t} |Y_{t,x}^{(T'_{t,x})}| \le \frac{\eta^2}{10}\delta_Y|W|$. Overall,
\begin{equation}
\sum_{t,x} |X_t||Y_{t,x}^{(T'_{t,x})}| \le \sum_t |X_t| \cdot \frac{\eta^2}{10}\delta_Y|W| \le \frac{\eta^2}{10}|X||Y|. \label{eq:txbound}
\end{equation}
For the final term, we can write
\begin{align*}
    \sum_{(t,x,t',y) \in \mc{I}} |X_{t,x,t',y}||Y_{t,x,t'}| &\le \sum_{\substack{t,x,t' \\ y \in W_t / W_{t,x,t'}}} |Y_{t,x,t'}| \cdot \frac{\eta^4}{100}\delta_X |W_{t,x,t'}| \\
    &= \sum_{t,x,t'} |Y_{t,x,t'}| \cdot \frac{\eta^4}{100}\delta_X |W_t| \le \sum_t |Y| \cdot \frac{\eta^4}{100}\delta_X |W_t| \\
    &\le \sum_t |Y| \cdot \frac{\eta^4}{100} \cdot \frac{10}{\eta^2}|X_t| \le \frac{\eta^2}{10}|X||Y|.
\end{align*}
Here, the second to last inequality follows because $\frac{|X_t|}{|W_t|} \ge \frac{\eta^2}{10}\delta_X$. Combining these verifies item (3).

To verify (5), note that each $Y_{t,x,t'} \subseteq W_{t,x,t'} + y_{t,x,t'}$ is $(r,\eps)$-algebraically spread by construction. Additionally, note that for any affine subspace $V \subseteq W_{t,x,t'}$ with $\dim(V) \ge \dim(W_{t,x,t'}) - r$ it holds that
\[ \frac{|X_{t,x,t',y} \cap V|}{|V|} \le \frac{|X_t \cap V|}{|V|} \le (1+\eps/5)\frac{|X_t|}{|W_t|}, \]
because $\dim(V) \ge \dim(W_t) - O(r\eps^{-1}\log(1/(\delta_Y\eta)))$ and $X_t$ is $(r', \eps/5)$-algebraically spread inside $W_t + x_t$ for $r'=\Omega(r\eps^{-1}\log(1/(\delta_Y\eta)))$ by construction. Thus if $\frac{|X_{t,x,t',y}|}{|W_{t,x,t'}|} \ge (1-3\eps/5)\frac{|X_t|}{|W_t|}$ (as is true for the tuples in $\mc{G}$), then $X_{t,x,t',y}$ is $(r, \eps)$-algebraically spread.

For item (6), first consider fixing some $t, x, t'$, and let $\delta_t = \frac{|X_t|}{|W_t|}$ to simplify notation. Note that $\E_{y \in W_t/W_{t,x,t'}}\left[\frac{|X_{t,x,t',y}|}{|W_{t,x,t'}|} \right] = \delta_t$, and that $\frac{|X_{t,x,t',y}|}{|W_{t,x,t'}|} \le (1+\eps/5)\delta_t$ for all $y$, as argued above. Thus by the reverse Markov inequality (\cref{fct:rev-mark}), $\Pr_{y \in W_t/W_{t,x,t'}}\left[\frac{|X_{t,x,t',y}|}{|W_{t,x,t'}|} \ge (1-3\eps/5)\delta_t \right] \ge \frac34.$
Hence,
\[ \sum_{\substack{y \in W_t/W_{t,x,t'} \\ \frac{|X_{t,x,t',y}|}{|W_{t,x,t'}|} \ge (1-3\eps/5)\delta_t}} |X_{t,x,t',y}| \ge \frac34\frac{|W_t|}{|W_{t,x,t'}|} (1-3\eps/5)\delta_t|W_{t,x,t'}| \ge \frac23|X_t|.\]
We conclude the total size of pieces from $\mc{G}$ is at least
\[ \sum_{t,x,t'} \frac23|X_t||Y_{t,x,t'}| = \sum_t \frac23|X_t||Y| - \sum_{t,x} \frac23|X_t||Y_{t,x}^{(T'_{t,x})}| \ge \frac12|X||Y|, \]
where we have used that $\sum_t |X_t| \ge (1-\eta^2)|X|$ and $\sum_{t,x} |X_t||Y_{t,x}^{(T'_{t,x})}| \le \frac{\eta^2}{10}|X||Y|$ from \eqref{eq:txbound}.
\end{proof}
Iterating \cref{lemma:oneround2dim} proves \cref{lemma:2dim}.
\begin{proof}[Proof of \cref{lemma:2dim}]
Perform the following algorithm given $X \times Y \subseteq W \times W$. Consider $X_i \times Y_i \subseteq (V_i+x_i) \times (V_i+y_i)$ for $i = 1, \dots, T$ as given in \cref{lemma:oneround2dim}. Now, for each $i \in [T] \setminus \mc{G}$ invoke \cref{lemma:oneround2dim} recursively. Note that if $X \subseteq W$ is spread, then $X+x \subseteq W+x$ is spread with the same parameters. Thus, even though \cref{lemma:oneround2dim} is stated in terms of linear subspaces $W$, by shifting $X,Y$ we can apply the lemma to $(X_i+x_i) \times (Y_i+y_i) \subseteq V_i \times V_i$ to obtain a partition, then shift all of the sets back by $x_i, y_i$, respectively.
Terminate when the recursion depth is $O(\log(1/\eta))$, and remove any remaining pieces from the partition.

Let's check that all the conclusions of \cref{lemma:2dim} hold.  By conclusion (4) of \cref{lemma:oneround2dim} and iterating for $O(\log(1/\eta))$ recursive layers gives that $|X_i|/|V_i| \ge 2^{-O(\log(1/\eta)^2)}\delta_X$ and $|Y_i|/|V_i| \ge 2^{-O(\log(1/\eta)^2)}\delta_Y$. Thus (4) holds.
(1) holds because the recursion depth is $O(\log(1/\eta))$ combined with conclusion (1) of \cref{lemma:oneround2dim} -- the additional $O(r\eps^{-2}\log(1/\eta)^5)$ terms is due to the fact that $\delta_X$ and $\delta_Y$ may drop by $2^{-O(\log(1/\eta)^2)}$ during the algorithm. (2) and (5) hold by construction.  (3) holds because conclusion (3) of \cref{lemma:oneround2dim} implies that at most $\eta^2|X||Y|$ total size of pieces is thrown out as each recursive layer, so $O(\eta^2\log(1/\eta)|X||Y|)$ total. Additionally, by conclusion (5) of \cref{lemma:oneround2dim} the total size of pieces being recursed on after $O(\log(1/\eta))$ layers is only $2^{-O(\log(1/\eta))}|X||Y| \le \eta^2|X||Y|$. So the total amount size of pieces that are either thrown out or not partitioned is at most $O(\eta^2\log(1/\eta)|X||Y|) \le \eta|X||Y|$ as desired.
\end{proof}

Now we apply \cref{lemma:2dim} to prove \cref{thm:algopseudo}.
\begin{proof}[Proof of \cref{thm:algopseudo}]
Set $\eta = \eps\alpha\delta_D/16$ and let $\{V_i, X_i, x_i, Y_i, y_i\}_{i\in [T]}$ be as given by \cref{lemma:2dim} with the choice $\eps \to \eps/16$.
We will show some choice of $i$ satisfies the conclusions of the theorem; thus, (1), (3), (4) follow by construction. 
In order to verify (2) and (5), we want to compare the sizes of $A \cap S(X_i, Y_i, D)$ and $S(X_i, Y_i, D)$, while also comparing these to the ``original'' densities $\delta_X, \delta_Y$, and most importantly $\delta_D$. This motivates us to consider the quantity
\begin{align*}
    &\sum_{i \in [T]} \left(|A \cap S(X_i, Y_i, D)| - (1-4\eps)\alpha|S(X_i, Y_i, D)| - \eps\alpha\delta_D|X_i||Y_i|\right)
    \\ &\qquad\qquad\qquad\ge  |A| - \eta|X||Y| - (1-4\eps)\alpha|S(X, Y, D)| - \eps\alpha\delta_D|X||Y| \\
    &\qquad\qquad\qquad\ge \alpha\delta|W|^2 - \eps\alpha\delta_D |X||Y| - (1-4\eps)(1+\eps)\alpha\delta|W|^2 - \eps\alpha\delta_D|X||Y| \ge 0,
\end{align*}
where we have used that $\delta = \delta_X\delta_Y\delta_D$, and $|S(X, Y, D)| \le (1+\eps)\delta|W|^2$ by \cref{lem:relatesxyd} because one of $X$ or $Y$ is $(r, \eps/8)$-algebraically spread. By averaging, there exists an index $i$ for which
\[
|A \cap S(X_i, Y_i, D)| \ge (1-4\eps)\alpha|S(X_i, Y_i, D)| + \eps\alpha\delta_D|X_i||Y_i|.
\]
Note that if $X_i \subseteq V_i + x_i$ and $Y_i \subseteq V_i+y_i$ then $S(X_i, Y_i, D) = S(X_i, Y_i, D_i)$ for $D_i \coloneqq D \cap (V_i+x_i+y_i)$. Thus 
\[ |A \cap S(X_i, Y_i, D_i)| \ge (1-4\eps)\alpha|S(X_i, Y_i, D_i)| \ge (1-5\eps)\alpha \frac{|X_i||Y_i||D_i|}{|V_i|}, \]
where the final inequality follows from \cref{lem:2side} as both $X_i, Y_i$ are $(r, \eps/16)$-algebraically spread, which gives (5). Additionally  
\[ \eps\alpha\delta_D|X_i||Y_i| \le |A \cap S(X_i, Y_i, D_i)| \le |S(X_i, Y_i, D_i)| \le 2|X_i||Y_i| \frac{|D_i|}{|V_i|}, \]
because $X_i, Y_i$ are $(r, \eps/16)$-algebraically spread. Thus, $\frac{|D_i|}{|V_i|} \ge \eps\alpha\delta_D/2$, and so (2) is satisfied.
\end{proof}

\subsection{Von Neumann lemma}

In this subsection, we give suitable conditions under which a set $A \subseteq S(X,Y,Z)$ contains roughly as many corners as a random set of the same density. More specifically, let $\delta_X$ denote the density of some set $X$ within a subspace $W$, and similarly for $\delta_Y, \delta_D$. Suppose $A \subseteq S(X,Y,D)$ has size $|A| = \alpha \delta_X\delta_Y\delta_D|W|^2$. For a typical choice of $X,Y,D$, the size of $S(X,Y,D)$ will be roughly $\delta_X\delta_Y\delta_D|W|^2$ (see \cref{lem:2side}), so one should morally view $\alpha$ as the density of $A$ within its container.

Define the trilinear form $\Phi(f_1, f_2, f_3) \coloneqq \E_{x,y,z \in W}\left[f_1(x, y)f_2(y+z, y)f_3(x, x+z) \right]$. 
Observe that by the change of variable $z \to x+y+z$, $\Phi(f_1, f_2, f_3) = \E_{x,y,z \in W}\left[f_1(x, y)f_2(x+z, y)f_3(x, y+z) \right]$, so that $\Phi(\mbm{1}_A,\mbm{1}_A,\mbm{1}_A)$ counts the number of corners in $A$ (up to normalization). 
We will show that if $A$ is sufficiently pseudorandom, then $\Phi(\mbm{1}_A,\mbm{1}_A,\mbm{1}_A)$ is approximately as large as $\alpha^3 \delta_X^2 \delta_Y^2 \delta_D^2$, which is the number of corners in randomly chosen $A,X,Y,Z$ of the same densities. Concretely, we will require two grid norms related to $\mbm{1}_A$ be bounded, as well as requiring that $A$ has no columns which are too sparse.

\begin{lemma}\label{lem:VNL}
    Let $W \subseteq \mb{F}_2^n$ be a linear subspace, and let $X, Y, D \subseteq W$ be subsets of size $|X| = \delta_X |W|, |Y| = \delta_Y|W|, |D| = \delta_D|W|$. Additionally, let $A \subseteq S(X, Y, D)$ be a subset of size $|A| = \alpha \delta_X\delta_Y\delta_D|W|^2$. Let $\eps \in (0,1/10)$ and let $p = \Omega(\log(1/(\alpha \delta_D)) / \eps^4)$ be a positive integer. Define the functions $F_1 : Y \times D \to \{0, 1\}$ where $F_1(y, z) = \mbm{1}_A(y+z, y)$ and $F_2 : X \times D \to \{0, 1\}$ where $F_2(x, z) = \mbm{1}_A(x, x+z)$. Suppose the following conditions hold:
    \begin{enumerate}
        \item $\snorm{F_1}_{G(2,p)} < (1+\eps^2/36)\alpha \delta_X$,
        \item $\snorm{F_2}_{G(2,p)} < 2\alpha \delta_Y$,
        \item For all $y \in Y$, we have $\E_{x \in W}[\mbm{1}_A(x,y)] \geq (1-\eps^2/36)\alpha \delta_X \delta_D.$
    \end{enumerate}
    Then,
    $$
        \Phi(\mbm{1}_A,\mbm{1}_A,\mbm{1}_A) \geq (1 - 4\eps)\alpha^3 \delta_X^2 \delta_Y^2 \delta_D^2.
    $$
\end{lemma}

\begin{proof}
    For clarity, let $f \coloneqq \mbm{1}_A$ denote the indicator function for $A$. Additionally, let $S = S(W, Y, D) = \{(y+z, y) \in W \times W : y\in Y, z\in D\}$, and let $g = f - \alpha \delta_X \mbm{1}_S$. Then, we have
    \begin{align*}
        \Phi(f,f,f) = \Phi(f,g,f) + \alpha \delta_X \cdot \Phi(f, \mbm{1}_S, f).
    \end{align*}
    We will proceed by lower bounding the second term. Afterwards, we will upper bound the magnitude of the first term, showing it is ultimately dominated by the second.
    
    Using the observation that whenever $f(x,y) = 1$ and $f(x,x+z)=1$ we must have $x \in X, y\in Y, z\in D$, we can lower bound 
    \begin{align*}
        \Phi(f, \mbm{1}_S, f) &= \E_{x,y,z \in W}\left[f(x, y)\mbm{1}_S(y+z, y)f(x, x+z) \right] \\
        &= \E_{x,y,z \in W}\left[f(x, y)\mbm{1}_Y(y)\mbm{1}_D(z)f(x, x+z) \right] \\
        &= \E_{x,y,z \in W}\left[f(x, y)f(x, x+z) \right] \\ 
        &= \E_{x \in W} \left(\E_{y \in W} f(x, y)\right)^2 \ge \delta_X^{-1}\left(\E_{x,y \in W} f(x, y)\right)^2 = \alpha^2\delta_X \delta_Y^2\delta_D^2,
    \end{align*}
    where we used the Cauchy-Schwarz inequality in the last line. In particular, 
$$
    \alpha \delta_X \cdot \Phi(f, \mbm{1}_S, f)  \geq \alpha^3 \delta_X^2 \delta_Y^2 \delta_D^2.
$$
Thus, it remains to bound the magnitude of $\Phi(f,g,f)$. Let $k = 2\ceil{\log(1/(\alpha\delta_D))}$.
H\"{o}lder's inequality gives 
\begin{align*}
\left|\E_{x,y,z \in W} f(x,y)g(y+z,y)f(x,x+z) \right| &= \delta_X \delta_Y \delta_D \left|\E_{x \in X, y \in Y, z \in D} f(x,y)g(y+z,y)f(x,x+z) \right| \\
&\le \delta_X \delta_Y \delta_D \left(\E_{x \in X, y \in Y} f(x,y) \right)^{\frac{k-1}{k}} \\
&\qquad \cdot \left(\E_{x \in X, y \in Y}\Big|\E_{z \in D} g(y+z, y)f(x, x+z) \Big|^k \right)^{1/k} \\
&= \delta_X \delta_Y \delta_D(\alpha \delta_D)^{1-\frac{1}{k}} \cdot \left(\E_{\substack{x \in X, y \in Y \\ z_1, \dots, z_k \in D}} \prod_{i=1}^k g(y+z_i,y)f(x,x+z_i) \right)^{1/k}.
\end{align*}
By our choice of $k$, we have $(\alpha \delta_D)^{-1/k} \le 2$, so it suffices to bound the second factor by $2 \eps \alpha^2 \delta_X \delta_Y$.
Define the function $G : Y \times D \to [-1, 1]$ as $G(y, z) = g(y+z, y)$.
Now, apply the Cauchy-Schwarz inequality to get
\begin{align*} \left(\E_{\substack{x \in X, y \in Y \\ z_1, \dots, z_k \in D}} \prod_{i=1}^k g(y+z_i,y)f(x,x+z_i) \right)^{1/k}  &\le 
\left(\E_{\substack{y_1,y_2 \in Y \\ z_1,\dots,z_k \in D }}\prod_{i=1}^k g(y_1+z_i, y_1)g(y_2 + z_i, y_2)\right)^{1/(2k)}\\
&\qquad\cdot \left(\E_{\substack{x_1,x_2 \in X \\ z_1,\dots,z_k \in D }} \prod_{i=1}^k f(x_1, x_1+z_i)f(x_2, x_2+z_i) \right)^{1/(2k)} \\
&= 
\left(\E_{\substack{y_1,y_2 \in Y \\ z_1,\dots,z_k \in D }}\prod_{i=1}^k G(y_1, z_i)G(y_2, z_i)\right)^{1/(2k)}\\
&\qquad\cdot \left(\E_{\substack{x_1,x_2 \in X \\ z_1,\dots,z_k \in D }} \prod_{i=1}^k F_2(x_1, z_i)F_2(x_2,z_i) \right)^{1/(2k)} \\
&= \|G\|_{G(2,k)} \cdot \|F_2\|_{G(2,k)}. 
\end{align*}
The second hypothesis, along with the fact that grid norms are monotonic, bounds the second factor by $2\alpha \delta_Y$.

We will finish the proof by bounding the first factor by $\eps \alpha \delta_X$. 
To do this, we will use the third item of the hypothesis as well as \Cref{lem:spectral-pos} to reduce this quantity to the first hypothesis. 
First, note that
\[ \|F_1\|_{G(1,1)} = (\delta_{Y}\delta_D)^{-1}\E_{y, z \in W}[f(y+z,y)] = \alpha\delta_X \] by the definition of $\alpha$. Thus $G = F_1 - \|F_1\|_{G(1,1)}$ by the definition of $g$. 
Observe the third item of the hypothesis guarantees lower--boundedness on the rows of $G$; namely for all $y \in Y$, we have that
$$
    \E_{z \in D}[G(y,z)] = \delta_D^{-1} \E_{z \in W}[g(y+z, y)] \geq (1-\eps^2/36)\alpha \delta_X - \alpha \delta_X = -\eps^2 \alpha \delta_X/36. 
$$
Combining  with the first item of the hypothesis, the contrapositive of \cref{lem:spectral-pos} gives that
\[
    \|G\|_{G(2,k)} < \eps\alpha \delta_X,
\]
as desired.
\end{proof}

\subsection{Density increment}

Now we prove that if the assumptions of \Cref{lem:VNL} are not satisfied, we can pass to a subset $X' \times Y' \times Z'$ where (essentially) the density of $A$ in $S(X', Y', Z')$ increases by a constant factor. There are two ways to obtain this density increment: an appropriate grid norm is large, or there are too many sparse rows. 

Let $\alpha$ be such that $|A| = \alpha\delta_X\delta_Y\delta_D|W|^2,$ where as before we morally view $\alpha$ as the density of $A$ within $S(X,Y,D)$. To start,
let $\wte>0$ be a parameter and define $L = \{ y \in Y : \E_x[\mbm{1}_A(x,y)] < (1-\wte)\alpha\delta_X\delta_D\} \subseteq Y$ to be the set of sparse rows of $A$. We will later set $\wte = \eps^2/36$, motivated by \cref{lem:VNL}. First we obtain a density increment in the case where $L$ is large.
\begin{lemma}
\label{lemma:llarge}
Let $X, Y, D \subseteq W$ and $\eps_L \in (0,1/2)$. If $|L| \ge \eps_L|Y|$, then there is some subset $Y' \subseteq Y$ of density $\delta_{Y'} \geq \delta_Y/2$ such that
\[ \frac{|A \cap S(X,Y',D)|}{\delta_X\delta_{Y'}\delta_D|W|^2} \ge (1+\wte\eps_L/2)\alpha. \]
\end{lemma}

\begin{proof}
Let $L'$ be an arbitrary subset of $L$ of size $|L'|=\eps_L |Y|$ and set $Y' = Y \setminus L'$. We have
\begin{align*}
|A \cap S(X, Y', D)| \ge |A| - |L'||W|(1-\wte)\alpha\delta_X\delta_D &= \alpha\delta_X\delta_Y\delta_D|W|^2 - (1-\wte)\eps_L\alpha\delta_X\delta_Y\delta_D|W|^2 \\
&= (1-(1-\wte)\eps_L)\alpha\delta_X\delta_Y\delta_D|W|^2.
\end{align*}
Thus,
\[ \frac{|A \cap S(X, Y', D)|}{\delta_X\delta_{Y'}\delta_D|W|^2} \ge \frac{1-(1-\wte)\eps_L}{1-\eps_L}\alpha \ge (1+\wte\eps_L/2)\alpha. \]
This completes the proof.
\end{proof}

The next lemma says that if one of the grid norms considered in \cref{lem:VNL} is large, then $A$ admits a density increment.

\begin{lemma}
\label{lemma:densincr}
Let $k$ be a positive integer, $\eps \in (0,1)$, and set $\eps_s = \Theta(\eps)$ sufficiently small. Let $X, Y, D\subseteq W$ and $A \subseteq S(X, Y, D)$. Define the functions $F_1 : Y \times D \to \{0, 1\}$ where $F_1(y, z) = \mbm{1}_A(y+z, y)$ and $F_2 : X \times D \to \{0, 1\}$ where $F_2(x, z) = \mbm{1}_A(x, x+z)$. Suppose $X$ and $Y$ are both $(r, \eps_s)$-algebraically spread for 
\[ r \ge \Omega(\eps^{-7}\log(1/(\delta_X\delta_Y))\log(1/(\delta_X \delta_Y \delta_D\gamma))^7),
\]
where 
\[ \gamma \le (\alpha\delta_X\delta_Y)^{O(\eps^{-2}k\log(1/\alpha)^2 + \eps^{-1}k\log(1/(\delta_X\delta_Y)))}. \]
Then if $\snorm{F_1}_{G(2,k)} \ge (1+\eps/32)\alpha\delta_X,$
then there are $Y' \subseteq Y$ and $D' \subseteq D$ with $|Y'| \ge (\eps\alpha/2)^{O(\eps^{-1}k^2\log(1/\alpha))}|Y|$ and $|D'| \ge (\eps\alpha/2)^{O(\eps^{-1}\log(1/\alpha))}|D|$, and 
\[ |A \cap S(X, Y', D')| \ge (1+\Omega(\eps))\alpha\delta_X|Y'||D'|. \]
Similarly, if $\snorm{F_2}_{G(2,k)} \ge 2\alpha\delta_Y,$
then there are $X' \subseteq X$ and $D' \subseteq D$ with $|X'| \ge (\eps\alpha/2)^{O(\eps^{-1}k^2\log(1/\alpha))}|X|$ and $|D'| \ge (\eps \alpha/2)^{O(\eps^{-1}\log(1/\alpha))}|D|$, and
\[ |A \cap S(X', Y, D')| \ge (1+\Omega(\eps))\alpha\delta_Y|X'||D'|. \]
\end{lemma}
\begin{proof}
We only prove the former assertion, as the latter one has an identical proof. Set $\tau = (1+8\eps_s)\delta_X$. We would like to argue that the set $T \coloneqq \{(y,z) \in Y \times D : y + z \in X\}$ is $(\tau, \gamma)$-combinatorially spread. We will start with the set $T' \subseteq W \times W$, which we define to be $T' \coloneqq \{(y+z,y) : y,z \in W, y+z \in X\} = S(X, W, W)$. By \Cref{lem:relate} as well as the $(r,\eps_s)$-algebraic spreadness of $X$ for $r$ large enough, we know that $T'$ is $(\tau, \delta_Y \delta_D \gamma)$-combinatorially spread. This is enough to imply the desired combinatorial spreadness of $T$, since
\begin{align*}
    \E_{y \in Y, z \in D}[\mbm{1}_T(y,z)g_1(y)g_2(z)] &= (\delta_Y \delta_D)^{-1}\E_{y,z \in W}[\mbm{1}_{T'}(y+z,y)g_1(y)g_2(z)] \\
    &\leq (\delta_Y \delta_D)^{-1}\left(\tau \E_{y \in W}[g_1]\E_{z \in W}[g_2] + \delta_Y \delta_D\gamma \right) \\
    &= \tau \E_{y \in Y}[g_1]\E_{z \in D}[g_2] + \gamma.
\end{align*}
If we write our grid norm assumption in terms of $\tau$, we have
$$
    \|F_1\|_{G(2,k)} \geq (1+\eps/32)\alpha \delta_X = \left( \frac{1+\eps/32}{1+8\eps_s} \right) \cdot  \alpha \tau \geq (1+\eps/64) \alpha \tau.
$$
Thus applying \cref{thm:quasisifting2} and \cref{lem:extract-cor} with sufficiently small $\eps_s$. gives that there are functions $g_1: Y \to \{0, 1\}$ and $g_2: D \to \{0, 1\}$ satisfying:
\begin{align*}
    \E_{y \in Y, d \in D}[F_1(y,d)g_1(y)g_2(d)] &\geq (1-\eps_s)(1+\eps/64)\alpha \tau \E_{y \in Y}[g_1(y)] \E_{d \in D}[g_2(d)] \\
    &\geq (1+\eps/128) \alpha \tau \E_{y \in Y}[g_1(y)] \E_{d \in D}[g_2(d)] \\
    &\geq (1+\eps/128) \alpha \delta_X \E_{y \in Y}[g_1(y)] \E_{d \in D}[g_2(d)]
\end{align*}
and
\[ \E_{y \in Y}[g_1(y)] \ge (\eps\alpha/2)^{O(\eps^{-1}k^2\log(1/\alpha))} \enspace \quad\text{and}\quad \enspace \E_{d \in D}[g_2(d)] \ge (\eps\alpha/2)^{O(\eps^{-1}\log(1/\alpha))}. \]
Letting $Y'$ and $D'$ be the indicator functions of $g_1$ and $g_2$ respectively completes the proof.
\end{proof}

\subsection{Obtaining spreadness}
In this section we will use a density increment algorithm to reach a state $A \subseteq S(X, Y, D)$ where the conditions of the von Neumann Lemma (\cref{lem:VNL}) are satisfied, and both $X, Y$ are algebraically spread. A useful definition is the notion of asymmetric combinatorial spreadness. Intuitively, this says that the function $f$ does not admit a density increment onto somewhat larger subrectangles, where we control the densities of the rows and columns separately.

\begin{definition}[Asymmetric combinatorial spreadness]
\label{def:asymspread}
    We say that a function $f : \Omega_1 \times \Omega_2 \to [0,1]$ is \emph{$(s,t,\eps)$-combinatorially spread} if for all functions $g_1 : \Omega_1 \to \{0,1\}$ and $g_2 : \Omega_2 \to \{0,1\}$ with
    $$
        \E[g_1(x)] \geq 2^{-s} \quad\text{and}\quad \E[g_2(y)] \geq 2^{-t},
    $$
    it holds that
    $$ 
        \E[f(x,y)g_1(x)g_2(y)] \leq (1+\eps)\E[f]\E[g_1]\E[g_2].
    $$
\end{definition}

We will apply this definition with $s$ much larger than $t$. This corresponds to the fact that in our proof $\delta_D$ is much larger than $\delta_X$ and $\delta_Y$ throughout. Note also that unlike \Cref{def:combpseudo}, the $g_i$ are simply subsets of $\Omega_i$. Now, we specialize the above definition to define what it means for a set $A$ to be combinatorially spread within the container $S(X, Y, D)$.

\begin{definition}
\label{def:aspread}
    Let $A \subseteq \mb{F}_2^n \times \mb{F}_2^n$, $W \subseteq \mb{F}_2^n$ be a linear subspace, and $X, Y, D \subseteq W$. Let $f$ be the indicator function of $A \cap S(X, Y, D)$. We say that $A$ is \emph{$(s,t,\eps)$-combinatorially spread} in a container $S(X, Y, D)$ if the functions $F_1 : Y \times D \to \{0,1\}$, $F_2 : X \times D \to \{0,1\}$ defined as
    $$
        F_1(y,d) = f(y+d, y) \quad\text{and}\quad F_2(x,d) = f(x, x+d)
    $$
    are $(s,t,\eps)$-combinatorially spread. 
\end{definition}

Our main lemma says that if one repeatedly does a density increment followed by pseudorandomization, then we reach a state where $A$ is combinatorially spread, both $X, Y$ are algebraically spread, and $A$ has lower--bounded rows. Additionally, the densities of $X, Y, D$ have not dropped too much, and the dimension of the subspace we are working in has not decreased significantly.

\begin{lemma}\label{lem:obtaining_spreadness}
    Let $r,s,t \geq 1$ and $\eps \in (0,1/64)$. Suppose that $A \subseteq \mb{F}_2^n \times \mb{F}_2^n$ has size $|A| = \alpha 4^n$. Then, there exists a subspace $W \subseteq \mb{F}_2^n$ along with points $x,y \in \mb{F}_2^n / W$ and sets $X \subseteq W+x, Y \subseteq  W+y, D \subseteq W+x+y$ with sizes $|X| = \delta_X |W|, |Y| = \delta_Y |W|$, and $|D| = \delta_D |W|$ satisfying: 
    \begin{enumerate}
        \item $X, Y$ are $(r,5\eps^{1/2})$-algebraically spread in $W + x, W+ y$, respectively,
        \item $\delta_D \geq \bar{\delta}_D \coloneqq \exp\left(-O(t\eps^{-1}\log(1/\alpha) + \eps^{-2}\log(1/\alpha)^2 ) \right)$,
        \item $\min\{\delta_X, \delta_Y\} \geq \bar{\delta} \coloneqq \exp\left(-O\left(s\eps^{-1}\log(1/\alpha) + \eps^{-1}\log(1/(\eps \alpha \bar{\delta}_D))^2\log(1/\alpha)\right)\right)$,
        \item $\dim(W) \geq n - O(r\eps^{-3}\log(1/(\eps \alpha\bar{\delta}_D\bar{\delta}))^2\log(1/(\eps\alpha\bar{\delta}_D))\log(1/\alpha) + r\eps^{-3}\log(1/(\eps\alpha\bar{\delta}_D))^5\log(1/\alpha))$,
        \item $A$ is $(s, t, 5\eps^{1/2})$-combinatorially spread within the container $S(X, Y, D)$ (see \cref{def:aspread}),
        \item The rows of $A \cap S(X, Y, D)$ are lower bounded; namely, for all $y \in Y$, 
        $$
            \E_{x \in W}[\mbm{1}_{A \cap S(X, Y, D)}(x,y)] \geq (1-2\eps^{1/2})\alpha^* \delta_X \delta_D
        $$
        where $\alpha^* = \frac{|A \cap S(X, Y, D)|}{\delta_X \delta_Y \delta_D |W|^2}$ and $\alpha^* \geq \alpha$,
    \end{enumerate}
    as long as $r \geq \Omega(\eps^{-8}\log(1/\bar{\delta})^8)$ large enough.
\end{lemma} 

\begin{proof}
    Perform the following algorithm, beginning with $X, Y, D = \mb{F}_2^n$. As long as $A$ is not $(s+1,t,\eps)$-combinatorially spread in its container $S(X, Y, D)$, iteratively restrict to containers $S(X_1, Y, D_1)$ or $S(X, Y_1, D_1)$ to obtain a $(1+\eps)$ density increment (as is guaranteed by \cref{def:aspread}), followed by applying \Cref{thm:algopseudo} with $\eps \to 8\eps_s$ for sufficiently small $\eps_s \coloneqq \Theta(\eps)$ at each step to reestablish algebraic pseudorandomness of $X, Y$. At the end, remove all rows from $A$ that violate (6).

    We start by analyzing how $X, Y, D$ change in one iteration of the algorithm, before the row removal phase. If $A$ is not $(s+1, t, \eps)$-combinatorially spread, without loss of generality we can find subsets $X_1 \subseteq X$ and $D_1 \subseteq D$ with $|X_1| \geq 2^{-s-1}|X|$ and $|D_1| \geq 2^{-t}|D|$ where 
    $$
        \E_{x \in X, d \in D}[\mbm{1}_{X_1}(x)\mbm{1}_{D_1}(d)f(x,x+d)] \geq (1+\eps)\E_{x \in X, d \in D}[f(x,x+d)] \cdot \E_{x \in X}[\mbm{1}_{X_1}] \cdot \E_{d \in D}[\mbm{1}_{D_1}],
    $$
    or equivalently after normalizing,
    $$
        \E_{x \in X_1, d \in D_1}[f(x,x+d)] \geq (1+\eps)\E_{x \in X, d \in D}[f(x,x+d)].
    $$
    In other words, 
    $$
        \frac{|A \cap S(X_1, Y, D_1)|}{|X_1||D_1|} \geq (1+\eps)\frac{|A \cap S(X, Y, D)|}{|X||D|}.
    $$
    Let $\alpha_1 = \frac{|A \cap S(X_1, Y, D_1)|}{\delta_{X_1} \delta_Y \delta_{D_1} |W|^2}$, and note that $\alpha_1 \geq (1+\eps)\frac{|A \cap S(X, Y, D)|}{\delta_X \delta_Y \delta_D |W|^2}$ by the above inequality. Set $\delta = \delta_{X_1}\delta_Y \delta_{D_1}$, and recall that $\eps_s = \Theta(\eps)$ small enough. At this point, we will also require that $r \geq \Omega(\eps_s^{-8}(\log(1/(\eps \alpha_1 \delta_{D_1}))^2 + \log(1/\delta))^8)$ large enough. Now, we apply \Cref{thm:algopseudo} with $A \cap S(X_1, Y, D_1)$\footnote{In general, if working in a container $S(X, Y, D)$ with $X \subseteq W+x, Y \subseteq W+y, D \subseteq W+x + y$, we can shift $A$ by $(x,y)$ to work in a container which is contained in $W \times W$ as prescribed by \Cref{thm:algopseudo}, followed by shifting the new container back by $(x,y)$. The only conclusion which could be affected is (4), but this is not an issue since if $X \subseteq W$ is algebraically spread, then $X + x \subseteq W + x$ is algebraically spread with the same parameters.}, $r$, and $\eps_s$ to obtain sets 
    $X', Y', D'$ along with shifts $x, y$ so that
    \begin{enumerate}
        \item $\dim(W') \ge n - O(r\eps_s^{-2} \log(1/(\eps_s \alpha\delta))^2\log(1/(\eps_s\alpha_1\delta_{D_1})) + r\eps_s^{-2}\log(1/(\eps_s\alpha_1\delta_{D_1}))^5)$,
        \item $\frac{|D'|}{|W'|} \ge \eps_s\alpha_1\delta_{D_1}/2$,
        \item $\frac{|X'||Y'|}{|W'|^2} \geq e^{-O(\log(1/(\eps \alpha_1 \delta_{D_1}))^2)}\delta_{X_1} \delta_Y$,
        \item $X'$ and $Y'$ are $(r, \eps_s/16)$-algebraically spread in $W' + x, W' + y$, respectively, and
        \item $|A \cap S(X', Y', D')| \ge (1-5\eps_s)\alpha_1\delta_{X'}\delta_{Y'}\delta_{D'}|W'|^2$ where $\delta_{X'} = \frac{|X'|}{|W'|}$, $\delta_{Y'} = \frac{|Y'|}{|W'|}$, and $\delta_{D'} = \frac{|D'|}{|W'|}$.
    \end{enumerate} 
    
    For $\eps_s \leq O(\eps)$ small enough, (5) and our earlier lower bound on $\alpha_1$ give
    $$
        \frac{|A \cap S(X', Y', D')|}{\delta_{X'}\delta_{Y'}\delta_{D'}|W'|^2} \geq (1+\eps/2) \frac{|A \cap S(X, Y, D)|}{\delta_X \delta_Y \delta_D |W|^2}.
    $$
    If we can show that $|S(X', Y', D')|$ is roughly $\delta_{X'}\delta_{Y'}\delta_{D'}|W'|^2$, then we can use the left-hand side as a measure of progress, since the density of $A$ in $S(X', Y', D')$ cannot exceed 1. Thus, it will follow that this process ends in at most $O(\eps^{-1}\log(1/\alpha))$ many iterations, since $\alpha = \frac{|A \cap S(X, Y, D)|}{\delta_X \delta_Y \delta_D |W|^2}$. By (2) and (3), we know that 
    $$
        \delta_{X'}\delta_{Y'}\delta_{D'} \geq \eps_s \alpha_1 e^{-O(\log(1/(\eps \alpha_1 \delta_{D_1}))^2)} \delta_{X_1} \delta_Y \delta_{D_1}.
    $$
    Thus, our choice of $r$ guarantees that $r \geq \Omega(\eps_s^{-8} \log(1/(\delta_{X'}\delta_{Y'}\delta_{D'}))^8)$. In general, we have chosen $r$ sufficiently large to satisfy such an inequality for all iterations. Hence, \Cref{lem:2side} implies that $|S(X', Y', D')| \leq (1+\eps_s)\delta_{X'}\delta_{Y'}\delta_{D'}|W'|^2$. In particular,
    \[
        1 \ge \frac{|A \cap S(X', Y', D')|}{|S(X', Y', D')|} = \frac{|A \cap S(X', Y', D')|}{\delta_{X'}\delta_{Y'}\delta_{D'}|W'|^2}\cdot \frac{\delta_{X'}\delta_{Y'}\delta_{D'}|W'|^2}{|S(X', Y', D')|} \ge \frac{|A \cap S(X', Y', D')|}{\delta_{X'}\delta_{Y'}\delta_{D'}|W'|^2} \cdot \frac{1}{1+\eps_s},
    \]
    so $|A \cap S(X', Y', D')| / \delta_{X'}\delta_{Y'}\delta_{D'}|W'|^2$ cannot exceed $1+\eps_s$, and at each iteration this quantity increases by a factor of $1+\eps/2$. Thus, if we iterate this process for $O(\eps^{-1} \log(1/\alpha))$ many iterations, we obtain sets $X^*, Y^*, D^*$ along with a subspace $W^*$ and shifts $x^*, y^*$ so that $A$ is $(s+1, t, \eps)$-combinatorially spread in the container $S(X^*, Y^*, D^*)$. Additionally, $X^*, Y^*$ are $(r,\eps_s)$-algebraically spread in $W^* + x^*, W^* + y^*$ by construction, which verifies (1).
    
    We now verify the size lower bounds on $X^*, Y^*, D^*$ as well as the dimension bound on $W^*$. First, we lower bound $\delta_{D^*}$. At every iteration, the density of $D$ is decreasing by a factor of at most  $2^{-t}$ due to the combinatorial density increment, and then by an additionally factor of at most $\eps_s \alpha /2$ due to the application of \Cref{thm:algopseudo}, which gives
    $$
        \frac{|D^*|}{|W^*|} \geq (\eps_s \alpha2^{-t})^{O(\eps^{-1} \log(1/\alpha))} \geq \exp\left(-O(t\eps^{-1}\log(1/\alpha) + \eps^{-2}\log(1/\alpha)^2)\right) \reflectbox{~$\coloneqq$~} \bar{\delta}_D
    $$
    for our choice of $\eps_s$. Similarly, $X, Y$ decrease in density by a factor of at most $2^{-s-1}$ due to the combinatorial density increment, and then by an additionally factor of $e^{-O(\log(1/(\eps \alpha \bar{\delta}_D))^2)}$ due to the application of \Cref{thm:algopseudo}. Note that $\log(1/(\eps \alpha \bar{\delta}_D)) \leq O(t\eps^{-1}\log(1/\alpha) + \eps^{-2}\log(1/\alpha)^2)$, yielding
    \begin{align*}
        \frac{|X^*|}{|W^*|}, \frac{|Y^*|}{|W^*|} &\geq \left(2^{-O(s + \log(1/(\eps \alpha \bar{\delta}_D))^2)}\right)^{O(\eps^{-1}\log(1/\alpha))} \\
        &\geq \exp\left(s\eps^{-1}\log(1/\alpha) + \eps^{-1}\log(1/(\eps \alpha \bar{\delta}_D))^2\log(1/\alpha)\right) \reflectbox{~$\coloneqq$~} \bar{\delta}.
    \end{align*}      
    Finally, we can verify the dimension bound on $W^*$. At each iteration, the application of \Cref{thm:algopseudo} decreases the dimension of $W$ by \[ O(r\eps_s^{-2}\log(1/(\eps_s\alpha\delta^*))^2\log(1/\eps\alpha\bar{\delta}_D) + r\eps_s^{-2}\log(1/\eps\alpha\bar{\delta}_D)^5) \] for $\delta^* = \bar{\delta}_D\bar{\delta}^2$. This gives
    \begin{align*}
        \dim(W^*) &\geq n - O(r\eps^{-3}\log(1/\alpha)\log(1/(\eps \alpha\delta^*))^2\log(1/(\eps\alpha\bar{\delta}_D)) + r\eps^{-3}\log(1/\alpha)\log(1/(\eps\alpha\bar{\delta}_D))^5).
    \end{align*}
    Notice also that on the last iteration of the algorithm, we need spreadness for $r \geq \eps^{-8}\log(1/\delta^*)^8$ in order to control the size of $S(X^*, Y^*, D^*)$, which accounts for the lower bound on $r$ in the theorem statement.
    
    To complete the proof of the lemma, we must guarantee that $A \cap S(X^*, Y^*, D^*)$ has lower--bounded rows. To achieve this, we will simply remove all of the rows which violate (6). More formally, let $f$ be the indicator of $A \cap S(X^*, Y^*, D^*)$ and $\alpha^* = |A \cap S(X^*, Y^*, D^*)| /(\delta_{X^*} \delta_{Y^*} \delta_{D^*}|W^*|^2)$, and let $L = \{ y \in Y^* : \E_{x \in W^*}[f(x,y) < (1-\eps^{1/2})\alpha^* \delta_{X^*} \delta_{D^*}\}$. We may assume $|L| \leq 4\eps^{1/2}|Y^*|$, as otherwise \Cref{lemma:llarge} implies there exists a subset $Y^+ \subseteq Y^*$ with $\delta_{Y^+} \geq \delta_{Y^*}/2$ such that
    $$
        \frac{|A \cap S(X^*, Y^+, D^*)|}{\delta_{X^*} \delta_{Y^+} \delta_{Z^*}|W^*|^2} \geq (1+\eps) \alpha^*,
    $$ 
    which contradicts the $(s,t,\eps)$-combinatorial spreadness of $A$ in $S(X^*, Y^*, D^*)$. Thus, we define $Y^+ \coloneqq Y \setminus L$ so that $|Y^+| \geq (1-4\eps^{1/2})|Y^*|$. Define
    $$
        \alpha^+ \coloneqq \frac{|A \cap S(X^*, Y^+, D^*)|}{\delta_{X^*} \delta_{Y^+} \delta_{D^*}|W^*|^2}.
    $$
    Note that $\alpha^+ \geq \alpha^*$ since we only deleted sparse rows. Additionally, we have $\alpha^+ \leq (1+\eps)\alpha^+$ by $(s+1, t, \eps)$-combinatorial spreadness. Thus, in terms of the new density $\alpha^+$, all of the columns with $y \in Y^+$ satisfy
    $$
        \E_{x \in G}[\mbm{1}_{S(X^*, Y^+, D^*)}(x,y)] \geq (1-\eps^{1/2})\alpha^* \delta_X \delta_Y \geq (1-2\eps^{1/2})\alpha^+ \delta_X \delta_D.
    $$
    
    Consider the container $S(X^*, Y^+, D^*)$; the lemma will follow if we can show that $Y^+$ is $(r,\eps)$-algebraically spread in $W^* + y^*$, and that $A$ is $(s,t,\eps)$-combinatorially spread in this new container. Firstly, we have $\delta_{Y^+} \geq (1-4\eps^{1/2})\delta_{Y^*}$. The algebraic spreadness of $Y^*$ implies that its density on any codimension $r$ subspace of $W^*$ is bounded by $(1+\eps_s)\delta_{Y^*}$. Thus, it follows that $Y^+$ is $(r, \eps_s + 4\eps^{1/2})$-algebraically spread in $W^* + y^*$. We conclude by showing that $A$ is $(s, t, 5\eps^{1/2})$-combinatorially spread in the container $S(X^*, Y^+, D^*)$. For clarity, define $F_1 : Y^* \times D^* \to \{0,1\}$ and $F_2 : X^* \times D^* \to \{0,1\}$ with $F_1(y,d) = f(y+d,y)$ and $F_2(x,d) = f(x,x+d)$. The $(s,t,\eps)$-combinatorial spreadness of $F_2$ follows easily, since for any $X' \subseteq X^*$ with size $|X'| \geq 2^{-(s+1)}|X^*|$ and $D' \subseteq D^*$ with size $|D'| \geq 2^{-t}|D^*|$, we have
    $$
        |A \cap S(X', Y^+, D')| \leq |A \cap S(X', Y^*, D')| \leq (1+\eps) \alpha^* \delta_{X'}\delta_{Y^*} \delta_{D'} |W^*|^2.
    $$
    This combined with the fact that $\delta_{Y^+} \geq (1-4\eps^{1/2})\delta_{Y^*}$ and $\alpha^+ \geq \alpha^*$ gives
    $$
        |A \cap S(X', Y^+, D')| \leq (1+5\eps^{1/2})\alpha^+ \delta_{X'}\delta_{Y^+} \delta_{D'} |W^*|^2,
    $$
    which is equivalent to $F_2$ being $(s+1, t, 5\eps^{1/2})$-combinatorially spread. Showing spreadness of $F_1$ is only slightly more subtle. For $Y' \subseteq Y^+$ with size $|Y'| \geq 2^{-(s+1)}|Y^*|$ and $D' \subseteq D^*$ with size $|D'| \geq 2^{-t}|D^*|$, we again have
    $$
        |A \cap S(X^*, Y', D')| \leq (1+\eps) \alpha^* \delta_{X^*}\delta_{Y'} \delta_{D'} |W^*|^2.
    $$
    Since $|Y^+| \geq |Y^*|/2$, this means that $Y'$ must have density at least $2 \cdot 2^{-s-1} = 2^{-s}$ in $Y^+$, which means $F_1$ is $(s,t,5\eps^{1/2})$-combinatorially spread.
\end{proof}

\subsection{Completing the proof.} In this short section, we combine the previous pieces we have developed to establish \cref{thm:mainff}.

\begin{proof}[Proof of \cref{thm:mainff}]
Let $A \subseteq \mb{F}_2^n \times \mb{F}_2^n$ be a set of size $|A| = \alpha 4^n$. The proof proceeds by restricting $A$ to a large container $S(X, Y, D)$ where $A$ is combinatorially spread using \Cref{lem:obtaining_spreadness}. Then, we will argue that combinatorial spreadness is enough to ensure bounded grid norms using \Cref{lemma:densincr}. At that point, we can apply \Cref{lem:VNL} to show that $A$ contains many corners.

Let $\eps$ be a sufficiently small constant, and let 
$$
    r = O(\log(1/\alpha)^{156}),\quad\quad s = O(\log(1/\alpha)^{8}),\quad\quad t = O(\log(1/\alpha)^2)
$$
for large enough implicit constants. By applying \Cref{lem:obtaining_spreadness} with $r,s,t, \eps_s = O(\eps^4)$ small enough, there exists a subspace $W \subseteq \mb{F}_2^n$ along with shifts $x, y \in \mb{F}_2^n$ and sets $X \subseteq W + x$, $Y \subseteq W + y, D \subseteq W + x + y$ satisfying the following properties: 
\begin{enumerate}
    \item $X, Y$ are $(r, \eps)$-algebraically spread in $W + x, W + y$, respectively.
    \item $|D| \geq 2^{-O(\log(1/\alpha)^3)}|W|$
    \item $|X|, |Y| \geq 2^{-O(\log(1/\alpha)^{9})}|W|$
    \item $\dim(W) \geq n - O(r\log(1/\alpha)^{22}) \geq n - O(\log(1/\alpha)^{178})$.
    \item The set $A$ is $(s, t, O(\eps^2))$-combinatorially spread in the container $S(X, Y, D)$.
    \item The rows of $A \cap S(X, Y, D)$ are lower bounded; namely for all $y \in Y$,
    $$
        \E_{x \in W} [\mbm{1}_{A \cap S(X, Y, D)}(x,y)] \geq (1-O(\eps^2))\alpha^* \delta_X \delta_D
    $$
    where $\alpha^* = \frac{|A \cap S(X, Y, D)|}{\delta_X \delta_Y \delta_D |W|^2}$ and $\alpha^* \geq \alpha$.
\end{enumerate}

Let $f$ denote the indicator of $A \cap S(X, Y, D)$, and define the functions $F_1 \colon Y \times D \to \{0,1\}$ with $F_1(y,d) = f(y,y+d)$ and $F_2 \colon X \times D$ with $F_2(x,d) = f(x,x+d)$. We will argue that the combinatorial spreadness of $A$ in the container $S(X, Y, D)$ implies certain grid norms are bounded. In particular, we will show that
\[ \|F_1\|_{G(2,k)} \leq (1+\eps^2/36)\alpha^* \delta_X \quad\text{and}\quad \|F_2\|_{G(2,k)} \leq 2 \alpha^* \delta_Y\]
for $k = O(\log(\alpha^* \delta_D)/\eps^4) = O(\log(1/\alpha)^3)$ large enough. Without loss of generality, assume the first assumption does not hold. If we can verify that $X$ is sufficiently algebraically spread, then \Cref{lemma:densincr} implies that there are $Y' \subseteq Y$ and $D' \subseteq D$ with $|Y'| \ge (\eps\alpha/2)^{O(k^2\log(1/\alpha))}|Y|$ and $|D'| \ge (\eps\alpha/2)^{O(\log(1/\alpha))}|D|$, and
\[ |A \cap S(X, Y', D')| \ge (1+\Omega(\eps^2))\alpha\delta_X|Y'||D'|. \]
This, however contradicts the $(O(\log(1/\alpha)^{8}),O(\log(1/\alpha)^2),O(\eps^2))$-combinatorial spreadness of $A$ in the container $S(X, Y, D)$. 
To see why $X$ is sufficiently spread, note that we need
\[r \geq \Omega(\log(1/(\delta_X \delta_Y)) \log(1/(\delta_X \delta_Y \delta_D\gamma))^7) \geq \Omega(\log(1/\alpha)^{9} \log(1/\gamma)^7)\]
for
\[\gamma \leq (\alpha \delta_X \delta_Y)^{O(k \log(1/\alpha)^2 + k \log(1/(\alpha \delta_X \delta_Y)))} \leq (\alpha \delta_X \delta_Y)^{O(\log(1/\alpha)^{12})} \leq 2^{-O(\log(1/\alpha)^{21})}.\]
Our choice of $r$ suffices, since
\[r \geq \Omega(\log(1/\alpha)^{9}\log(1/\gamma)^7) \geq \Omega(\log(1/\alpha)^{156}).\]
Thus, the conditions of \Cref{lem:VNL} are met, which implies that
\[\Phi(f, f, f) \geq (1-4\eps)\alpha^3\delta_X^2\delta_Y^2\delta_D^2.\]
This implies that $A$ contains at least 
\[(1-4\eps)\alpha^3\delta_X^2\delta_Y^2\delta_D^2 |W|^3 \geq 2^{-O(\log(1/\alpha)^9)}|W|^3 \geq 2^{-O(\log(1/\alpha)^{178})} |\mb{F}_2^n|^3\]
many corners. 
\end{proof}

\section{Bohr Sets, Algebraic Spreadness, and Pseudorandomization}
\label{sec:bohr-prelim}

The vast majority of the remainder of the body of the paper is devoted to establishing \cref{thm:main}, an improved corners bound over general abelian groups. The proof in many regards closely follows that in the finite field model setting, but as is standard one is forced to work with Bohr sets throughout the analysis. 
\subsection{Bohr sets}
We now recall various standard material regarding Bohr sets. The influence of Bohr sets in additive combinatorics stems from seminal work of Bourgain \cite{Bou99}. For a textbook treatment, we refer the reader to \cite[Section~4.4]{TV10}. We first define a Bohr set of a finite abelian group $G$.

\begin{definition}[Bohr set]\label{def:Bohr-set}
Let $\eps \in \mb{R}^{+}$, $G$ be a finite abelian group, and $\Theta = (\Theta_1, \dots, \Theta_d)$ where $\Theta_i \in \wh{G}$ are additive homomorphisms from $G$ to $\mb{R}/\mb{Z}$. We define the \emph{Bohr set}
\[\Lambda = \Lambda_{\Theta,\eps} = \bigcap_{i=1}^{d}\Big\{x \in G:\norm{\Theta_i(x)}_{\mb{R}/\mb{Z}}\le \eps \Big\}, \]
where $\|x\|_{\mb{R}/\mb{Z}} = \min_{z \in \mb{Z}} |x-z|$.

For any real number $c>0$, we define the dilated Bohr set 
\[c\Lambda_{\Theta,\eps} = \Lambda_{\Theta,c\eps}.\]
\end{definition}

We will refer to $d$ as the dimension of the Bohr set and $\eps$ as the radius. We denote the radius of a Bohr set $B$ as $\nu(B)$.

We first require that the size of a Bohr set is lower bounded in terms of its radius and dimension. This appears as \cite[Lemma~4.20]{TV10}.
\begin{lemma}\label{lem:size-bounded}
Let $\Lambda = \Lambda_{\Theta,\eps}$ be a Bohr set of dimension $d$ and radius $\eps$. Then $|\Lambda|\ge \eps^{d}|G|$.
\end{lemma}

For the vast majority of our analysis we will operate with regular Bohr sets as introduced by Bourgain \cite{Bou99}. Heuristically, regular Bohr sets are those such that altering the radius parameter causes the size of the underlying Bohr set to vary in a predictable manner. 
\begin{definition}[Regular]\label{def:reg-Bohr}
A Bohr set $\Lambda = \Lambda_{\Theta,\eps}$ of dimension $d$ is \emph{regular} if for all $|c|\le 1/(100d)$, we have that 
\[1 - 100d|c|\le \frac{|(1+c)\Lambda|}{|\Lambda|}\le 1 + 100d|c|.\]
\end{definition}

A crucial feature of regular Bohr sets is that they may be constructed easily (at worst at the cost of passing to radius half the size). This was established by Bourgain \cite{Bou99} (see also \cite[Lemma~4.25]{TV10}).
\begin{lemma}\label{lem:regular}
For any Bohr set $\Lambda$, there exists $\alpha\in [1/2,1]$ such that $\alpha \Lambda$ is regular.
\end{lemma}

Furthermore we additionally have that regular Bohr sets are essentially ``shift--invariant'' when shifted by elements in a smaller Bohr set. We will use such inequalities repeatedly and without substantial comment. 
\begin{lemma}\label{lem:shift-invar}
Let $f$ be a $1$-bounded function and $\Lambda$ be a regular Bohr set of dimension $d$. If $|c|\le 1/(100d)$ and $n'\in c\Lambda$, then 
\[\E_{n\in \Lambda}f(n) = \E_{n\in \Lambda}f(n + n') + O(cd).\]
\end{lemma}
\begin{proof}
The above claim follows as
\[
\big|\E_{n\in \Lambda}f(n) - \E_{n\in \Lambda}f(n + n')\big|\le 2\E_{n\in \Lambda}\mbm{1}[n+n'\in (1+c)\Lambda \setminus \Lambda] = O(cd). \qedhere\]
\end{proof}

To simplify the presentation in the remainder of the paper, we define a $(d, \eta)$-small sequence of Bohr sets.
\begin{definition}[Small \& exact sequences]
\label{def:etasmall}
We say that $B_1, B_2, \dots$ is a \emph{$(d, \eta)$-small} sequence of Bohr sets if all $B_i$ have the same set of $d$ frequencies (so that $\rank(B_i) = d$), are all regular, and $\nu(B_{i+1})/\nu(B_i) \le \eta$ for $i = 1, 2, \dots$.
We say that the sequence is \emph{$(d, \eta)$-exact} if additionally $\nu(B_{i+1})/\nu(B_i) \in [\eta/2, \eta]$ for all $i = 1, 2, \dots$.
\end{definition}
Most of our lemma statements will involve a $(d, \eta)$-small sequence of Bohr sets as input hypotheses, and may output a $(d', \eta)$-exact sequence of Bohr sets. By \cref{lem:regular}, any Bohr set $B_1$ with $\rank(B_1) = d$ can be extended to a $(d, \eta)$-exact sequence of Bohr sets of arbitrary length.

We next require certain specialized Gowers grid norms which will be used throughout the analysis. These were introduced in the work of Mili\'cevi\'c \cite{Mil24}.
\begin{definition}[Gowers grid norm]
\label{def:gow-high}
Fix integers $k,\ell\ge 1$. Consider a triplet of Bohr sets $B_1$, $B_2$, and $B_3$ and a function $f:G\to \mb{R}$. We define the \emph{$(B_1, B_2, B_3, k, \ell)$-Gowers grid norm} to be
\[\snorm{f}_{(B_1,B_2,B_3, k,\ell)}^{k\ell} = \E_{\substack{x\sim B_1\\y_1,\ldots,y_k\sim B_2\\z_1,\ldots,z_\ell\sim B_3}}\prod_{\substack{i\in [k]\\j\in [\ell]}}f(x+y_i+z_j).\]
When the Bohr sets are clear from context, we will refer to this quantity more simply as the $(k,\ell)$-Gowers grid norm.
\end{definition}

Often, we will abuse terminology and refer to the Gowers grid norm of a set $X$ to mean the Gowers grid norm of its indicator function $\mbm{1}_X$.

We will require that the norms increase under passing to finer Bohr sets. The key inequality for this result will be the following Gowers--H\"{o}lder inequality which appears as \cite[Lemma~2.2]{FHHK24}.
\begin{lemma}\label{lem:gow-hold}
Fix integers $k,\ell\ge 1$ and finite sets $X,Y$. Consider functions $f_{ij}:X\times Y\to \mb{R}^{\ge 0}$. Then we have that
\[\E_{\substack{x_1,\dots,x_k\in X\\y_1,\dots,y_{\ell}\in Y}} \prod_{\substack{i\in [k]\\j\in [\ell]}}f_{ij}(x_i,y_j)\le \prod_{\substack{i \in [k] \\ j \in [\ell]}} \|f_{ij}\|_{G(k,\ell)} = \prod_{\substack{i\in [k]\\j\in [\ell]}}\Bigg(\E_{\substack{x_1,\dots,x_k\in X\\y_1,\dots,y_{\ell}\in Y}}\prod_{\substack{i'\in [k]\\j'\in [\ell]}}f_{ij}(x_{i'},y_{j'})\Bigg)^{1/(k\ell)}.\]
\end{lemma}
\begin{proof}
Observe via repeated application of rearranging and H\"{o}lder's inequality that
\begin{align*}
\E_{\substack{x_1,\dots,x_k\in X\\y_1,\dots,y_{\ell}\in Y}} \prod_{\substack{i\in [k]\\j\in [\ell]}}f_{ij}(x_i,y_j) &= \E_{\substack{x_1,\dots,x_k\in X}}\prod_{j\in [\ell]}\Big(\E_{y_1,\dots,y_{\ell}\in Y}\prod_{i\in [k]}f_{ij}(x_i,y_j)\Big) \\
&\le \prod_{j\in [\ell]}\Big(\E_{\substack{x_1,\dots,x_k\in X}}\Big(\E_{y_1,\dots,y_{\ell}\in Y}\prod_{i\in [k]}f_{ij}(x_i,y_j)\Big)^{\ell}\Big)^{1/\ell}\\
&=\prod_{j\in [\ell]}\Big(\E_{\substack{x_1,\dots,x_k\in X\\y_1,\dots,y_{\ell}\in Y}}\prod_{\substack{i\in [k]\\j'\in [\ell]}}f_{ij}(x_i,y_{j'})\Big)^{1/\ell} \\
&= \prod_{j\in [\ell]}\Big(\E_{y_1,\dots,y_{\ell}\in Y}\prod_{i\in [k]}\E_{x_1,\dots,x_k\in X}\prod_{j'\in [\ell]}f_{ij}(x_i,y_{j'})\Big)^{1/\ell}\\
&\le \prod_{\substack{i\in [k]\\j\in [\ell]}}\Big(\E_{y_1,\dots,y_{\ell}\in Y}\Big(\E_{x_1,\dots,x_k\in X}\prod_{j'\in [\ell]}f_{ij}(x_i,y_{j'})\Big)^{k}\Big)^{1/(k\ell)} \\
&= \prod_{\substack{i\in [k]\\j\in [\ell]}}\Bigg(\E_{\substack{x_1,\dots,x_k\in X\\y_1,\dots,y_{\ell}\in Y}}\prod_{\substack{i'\in [k]\\j'\in [\ell]}}f_{ij}(x_{i'},y_{j'})\Bigg)^{1/(k\ell)}. \qedhere
\end{align*}
\end{proof}

The following approximate monotonicity of Gowers grid norms in \cref{def:gow-high} will be crucial in our analysis.
\begin{lemma}\label{lem:monotone}
Let $\eta\in (0,1/(100d))$ and fix integers $k,\ell\ge 1$. Let $B_1, B_2, \dots$ be a $(d,\eta)$-small sequence of Bohr sets. Then for a function $f:G\to [0,1]$ we have that 
\[\snorm{f}_{(B_1,B_2,B_3,k,\ell)}\le \snorm{f}_{(B_1,B_4,B_5,k,\ell)} + O((\eta d)^{1/(k\ell)}).\]
\end{lemma}
\begin{proof}
Observe that it suffices to establish that 
\[\snorm{f}_{(B_1,B_2,B_3,k,\ell)}^{k\ell}\le \snorm{f}_{(B_1,B_4,B_5,k,\ell)}^{k\ell} + O(k\ell \cdot \eta d).\]
Next using \cref{lem:shift-invar}, we have that 
\begin{align*}
\snorm{f}_{(B_1,B_2,B_3,k,\ell)}^{k\ell} &= \E_{\substack{x\sim B_1\\y_1,\ldots,y_k\sim B_2\\z_1,\ldots,z_{\ell}\sim B_3}}\prod_{\substack{i\in [k]\\j\in [\ell]}}f(x + y_i + z_j) \\
&\le \E_{\substack{x\sim B_1\\y_1,\ldots,y_k\sim B_2\\z_1,\ldots,z_{\ell}\sim B_3\\y_1',\ldots,y_k'\sim B_4\\z_1',\ldots,z_{\ell}'\sim B_5}}\prod_{\substack{i\in [k]\\j\in [\ell]}}f(x + y_i + z_j + y_{i}' + z_{j}') + O(k\ell \cdot \eta d).
\end{align*}
Taking the expectation on $x,y_i,z_j$ outside and defining $f_{i,j}(y,z) = f(x+y_i + z_j + y + z)$ we observe that the final quantity is exactly as in \cref{lem:gow-hold}. Therefore we have that 
\begin{align*}
\E_{\substack{x\sim B_1\\y_1,\ldots,y_k\sim B_2\\z_1,\ldots,z_{\ell}\sim B_3\\y_1',\ldots,y_k'\sim B_4\\z_1',\ldots,z_{\ell}'\sim B_5}}&\prod_{\substack{i\in [k]\\j\in [\ell]}}f(x + y_i + z_j + y_{i}' + z_{j}')\le \E_{\substack{x\sim B_1\\y_1,\ldots,y_k\sim B_2\\z_1,\ldots,z_{\ell}\sim B_3}} \prod_{\substack{i\in [k]\\j\in [\ell]}}\Big(\E_{\substack{y_1',\ldots,y_k'\sim B_4\\z_1',\ldots,z_{\ell}'\sim B_5}}\prod_{\substack{i'\in [k]\\j'\in [\ell]}}f(x+y_i+z_j + y_{i'}' + z_{j'}')\Big)^{1/(k\ell)}\\
&\le \prod_{\substack{i\in [k]\\j\in [\ell]}}\Big(\E_{\substack{x\sim B_1\\y_1,\ldots,y_k\sim B_2\\z_1,\ldots,z_{\ell}\sim B_3}} \E_{\substack{y_1',\ldots,y_k'\sim B_4\\z_1',\ldots,z_{\ell}'\sim B_5}}\prod_{\substack{i'\in [k]\\j'\in [\ell]}}f(x+y_i+z_j + y_{i'}' + z_{j'}')\Big)^{1/(k\ell)}\\
&\le \prod_{\substack{i\in [k]\\j\in [\ell]}}\Big(\E_{x\sim B_1} \E_{\substack{y_1',\ldots,y_k'\sim B_4\\z_1',\ldots,z_{\ell}'\sim B_5}}\prod_{\substack{i'\in [k]\\j'\in [\ell]}}f(x+ y_{i'}' + z_{j'}') + O(k\ell \cdot \eta d )\Big)^{1/(k\ell)} \\
&= \E_{x\sim B_1} \E_{\substack{y_1',\ldots,y_k'\sim B_4\\z_1',\ldots,z_{\ell}'\sim B_5}}\prod_{\substack{i'\in [k]\\j'\in [\ell]}}f(x+ y_{i'}' + z_{j'}') + O(k\ell \cdot \eta d).
\end{align*}
Here we first applied \cref{lem:gow-hold}, then H\"{o}lder's inequality, and then \cref{lem:shift-invar}.
\end{proof}

The next several lemmas are somewhat analogous to the ones in the finite field section bounding the size of the container set $S(X, Y, D) = \{(x,y): x \in X, y \in Y, x+y \in D\}$ under various spreadness conditions on $X, Y, D$ (see \cref{lem:2side,lem:relatesxyd}). These lemmas are substantially more complicated in the setting of general abelian groups. To explain why, we first briefly describe the setup. We have Bohr sets $B_1, B_2$ with the same frequencies and $\nu(B_2)/\nu(B_1) \le \eta$. We will assume that $X, D \subseteq B_1$ and $Y \subseteq B_2$.

We first prove that if one of $X$ or $D$ has a bounded Gowers grid norm, then the size of the container $S(X, Y, D)$ is upper bounded. Below, one should think of $g$ as the indicator function of either $X$ or $D$.

\begin{lemma}\label{lem:upper-bound}
Fix an even integer $K\ge 1$. Let $B_1, B_2, \dots$ be a $(d,\eta)$-small sequence of Bohr sets. Consider a triplet of functions $f_1:B_1\to [0,1]$, $f_2:B_2\to [0,1]$, and $g:B_1\to [0,1]$, and suppose that $\snorm{g}_{(B_1,B_3,B_4,K,K)}\le \tau$, where $\tau \in [0, 1]$.

Then we have that 
\[\E_{\substack{x\sim B_1\\y\sim B_2}}f_1(x)f_2(y)g(x+y) \le \tau \cdot (1 + 3\eps)\cdot \E_{\substack{x\sim B_1\\y\sim B_2}}f_1(x)f_2(y) + O(e^{-\Omega(\eps K)} + \tau^{-K^2} \eta d).\]
\end{lemma}
\begin{proof}
Observe by \cref{lem:shift-invar} that
\begin{align*}
\E_{\substack{x\sim B_1\\y\sim B_2}}f_1(x)f_2(y)g(x+y) &\le \E_{\substack{x\sim B_1\\y\sim B_2}}\E_{\substack{x'\sim B_3\\y'\sim B_4}}f_1(x+x')f_2(y+y')g(x+y + x'+y') +O(\eta d).
\end{align*}
For fixed $x \in B_1, y \in B_2$, we can apply H\"{o}lder's inequality twice to get that 
\begin{align*}
\E_{\substack{x'\sim B_3\\y'\sim B_4}}&f_1(x+x')f_2(y+y')g(x+y + x'+y') \\
&\le \Big(\E_{\substack{x'\sim B_3}}f_1(x+x')\Big)^{(K-1)/K} \cdot \Big(\E_{x'\sim B_3}\Big(\E_{y'\sim B_4}f_2(y+y')g(x+y + x'+y')\Big)^{K}\Big)^{1/K}\\
&\le \Big(\E_{\substack{x'\sim B_3\\y'\sim B_4}}f_1(x+x')f_2(y+y')\Big)^{(K-1)/K} \cdot \Big(\E_{\substack{x_1',\ldots,x_K'\sim B_3\\y_1',\cdots,y_K'\sim B_4}}\prod_{\substack{i\in [K]\\j\in [K]}}g(x+y+x_i' + y_j')\Big)^{1/K^2}\\
&\le (1+\eps)\cdot \E_{\substack{x'\sim B_3\\y'\sim B_4}}f_1(x+x')f_2(y+y')\cdot \Big(\E_{\substack{x_1',\ldots,x_K'\sim B_3\\y_1',\cdots,y_K'\sim B_4}}\prod_{\substack{i\in [K]\\j\in [K]}}g(x+y+x_i' + y_j')\Big)^{1/K^2} + e^{-\Omega(\eps K)},
\end{align*}
where the bound in the last line is because either $\E_{\substack{x'\sim B_3\\y'\sim B_4}}f_1(x+x')f_2(y+y') \le e^{-\Omega(\eps K)}$, which makes the bound trivial, or $\Big(\E_{\substack{x'\sim B_3\\y'\sim B_4}}f_1(x+x')f_2(y+y')\Big)^{-1/K} \le 1+\eps$.

For $x \in B_1, y \in B_2$ define the event
\[\mc{E}(x,y) = \mbm{1}\Big[\Big(\E_{\substack{x_1',\ldots,x_K'\sim B_3\\y_1',\cdots,y_K'\sim B_4}}\prod_{\substack{i\in [K]\\j\in [K]}}g(x+y+x_i' + y_j')\Big)^{1/K^2}\ge \tau \cdot (1+\eps)\Big]\]
and observe that 
\begin{align*}
\E[\mc{E}(x,y)]&\le (\tau \cdot (1+\eps))^{-K^2} \cdot \E_{\substack{x \sim B_1, y \sim B_2 \\ x_1',\ldots,x_K'\sim B_3\\y_1',\cdots,y_K'\sim B_4}}\prod_{\substack{i\in [K]\\j\in [K]}}g(x+y+x_i' + y_j') \\
&\le e^{-\Omega(\eps K^2)}\tau^{-K^2}(\|g\|_{(B_1,B_3,B_4,K,K)}^{K^2} + O(\eta d)) \le e^{-\Omega(\eps K^2)} + O(\tau^{-K^2} \cdot \eta\cdot d).
\end{align*}
Combining this exceptional set bound with the earlier H\"{o}lder argument we have
\begin{align*}
\E_{\substack{x\sim B_1\\y\sim B_2}}&\E_{\substack{x'\sim B_3\\y'\sim B_4}}f_1(x+x')f_2(y+y')g(x+y + x'+y')\\
&\le \E_{\substack{x\sim B_1\\y\sim B_2}}(1-\mc{E}(x,y))\E_{\substack{x'\sim B_3\\y'\sim B_4}}f_1(x+x')f_2(y+y')g(x+y + x'+y') + e^{-\Omega(\eps K^2)} + O(\tau^{-K^2} \cdot \eta\cdot d)\\
&\le \E_{\substack{x\sim B_1\\y\sim B_2}}\tau \cdot (1+3\eps)\cdot\E_{\substack{x'\sim B_3\\y'\sim B_4}}f_1(x+x')f_2(y+y') + e^{-\Omega(\eps K)} + O(\tau^{-K^2} \cdot \eta\cdot d)\\
&=\tau \cdot (1+3\eps)\cdot \E_{\substack{x\sim B_1\\y\sim B_2}}f_1(x)f_2(y) + e^{-\Omega(\eps K)} + O(\tau^{-K^2} \cdot \eta\cdot d)
\end{align*}
as desired.
\end{proof}

Our next lemma gives an estimate of the container size in the setting when $X, Y$ have bounded Gowers grid norms. We note that this estimate is \emph{not} $\E[\mbm{1}_X]\E[\mbm{1}_Y]\E[\mbm{1}_D]$, and is instead $\E[\mbm{1}_X \mbm{1}_D]\E[\mbm{1}_Y]$. Later, we give more conditions under which $\E[\mbm{1}_X\mbm{1}_D] \approx \E[\mbm{1}_X]\E[\mbm{1}_D]$. Below, we encourage the reader to think of $f_1, f_2, g$ as the indicator functions of $X, Y, D$ respectively. 

\begin{lemma}\label{lem:conv-lower-bound}
Let $\eps\in (0,1/100)$ and fix an even integer $K\ge 2$. Let $B_1, B_2, \dots$ be a $(d,\eta)$-small sequence of Bohr sets. Consider a triplet of functions $f_1:B_1\to [0,1]$, $f_2:B_2\to [0,1]$, and $g:B_1\to [0,1]$. Suppose that $\snorm{f_1}_{(B_1,B_4,B_5,K,K)}\le (1+\eps) \cdot \E[f_1]$, $\snorm{f_2}_{(B_2,B_4,B_5,K,K)}\le (1+\eps) \cdot \E[f_2]$, $\eta \le \eps^3 \cdot (\E[f_1]\cdot \E[f_2] \cdot \E[g])^{O(K^2)}$, and $K\ge 100\eps^{-8}\log(2/(\E[f_1]\cdot \E[f_2] \cdot \E[g]))$.

Then we have that
\begin{multline*}
    \Big|\E_{\substack{x\sim B_1\\y\sim B_2}}f_1(x)f_2(y)g(x+y) - \E_{\substack{x\sim B_1\\y\sim B_2}}[f_1(x)g(x+y)]\cdot \E_{y\sim B_2}[f_2(y)]\Big| \\
    = O(\eps^{1/2}) \cdot \E[f_1]\E[f_2]\E[g] + O((\eta d)^{1/(2K)}+ e^{-\Omega(\eps^8K)}).
\end{multline*}
\end{lemma}
\begin{proof}
Observe that
\begin{align*}
\E_{\substack{x\sim B_1\\y\sim B_2}}f_1(x)f_2(y)g(x+y) &= \E_{\substack{x\sim B_1\\y\sim B_2}}f_1(x-y)f_2(y)g(x) + O(\eta d)\\
&= \E_{\substack{x\sim B_1\\y\sim B_2}}f_1(x-y)\wt{f_2}(y)g(x)+ \E_{\substack{x\sim B_1\\y\sim B_2}}[f_1(x)g(x+y)]\cdot \E_{y\sim B_2}[f_2(y)] + O(\eta d)\\
&= \E_{\substack{x\sim B_1\\y\sim B_2\\z\sim B_3}}f_1(x-y)\wt{f_2}(y + z)g(x + z) + \E_{\substack{x\sim B_1\\y\sim B_2}}[f_1(x)g(x+y)]\cdot \E_{y\sim B_2}[f_2(y)]+ O(\eta d)
\end{align*}
where $\wt{f_2}(y) = f_2(y) - \E_{\substack{y'\sim B_2}}[f_2(y')]$.

We let $L$ be a positive integer to be chosen later. Now observe that 
\begin{align*}
\Big|\E_{\substack{x\sim B_1\\y\sim B_2\\z\sim B_3}}&f_1(x-y)\wt{f_2}(y + z)g(x + z)\Big|\le \E_{\substack{x\sim B_1\\z\sim B_3}}g(x + z)\Big|\E_{y\sim B_2}f_1(x-y)\wt{f_2}(y + z)\Big|\\
&\le (\E_{\substack{x\sim B_1\\z\sim B_3}}g(x+z))^{(L-1)/L} \cdot \Big(\E_{\substack{x\sim B_1\\z\sim B_3}}\E_{y_1,\ldots,y_{L}\sim B_2}\prod_{j=1}^{L}f_1(x-y_j)\wt{f_2}(y_j+z)\Big)^{1/L}\\
&\le (1 + \eps) \cdot \E_{y\sim B_2}g(y) \cdot \Big(\E_{\substack{x_1,x_2\sim B_1\\y_1,\ldots,y_{L}\sim B_2}}\prod_{j=1}^{L}f_1(x_1 - y_j)f_1(x_2-y_j)\Big)^{1/(2L)} \\
&\qquad\qquad\qquad\cdot \Big(\E_{\substack{z_1,z_2\sim B_3\\y_1,\ldots,y_{L}\sim B_2}}\prod_{j=1}^{L}\wt{f_2}(y_j + z_1)\wt{f_2}(y_j+z_2)\Big)^{1/(2L)} + O(\eta d + e^{-\Omega(\eps L)}),
\end{align*}
where the final inequality follows from $\delta^{\frac{L-1}{L}} \le (1+\eps)\delta + e^{-\Omega(\eps L)}$, and an application of Cauchy-Schwarz.
In order to complete our analysis it suffices to bound the last two terms: 
\[
    \Big(\E_{\substack{x_1,x_2\sim B_1\\y_1,\ldots,y_{L}\sim B_2}}\prod_{j=1}^{L}f_1(x_1 - y_j)f_1(x_2-y_j)\Big)^{1/(2L)} \text{ and } \Big(\E_{\substack{z_1,z_2\sim B_3\\y_1,\ldots,y_{L}\sim B_2}}\prod_{j=1}^{L}\wt{f_2}(y_j + z_1)\wt{f_2}(y_j+z_2)\Big)^{1/(2L)}.
\]

We now derive a bound on the second term involving $\wt{f}_2$  -- the bound on the first term involving $f_1$ is similar. For the second term, observe that 
\[\min_{z\in B_3}\E_{y\sim B_2}[\wt{f_2}(y+z)]\ge -O(\eta d).\]
Furthermore observe that by H\"{o}lder's inequality that $\snorm{f_2}_{(B_2,B_4,B_5,K,2)}\le \snorm{f_2}_{(B_2,B_4,B_5,K,K)} \le (1+\eps)\cdot \E_{y\sim B_2,z\sim B_3}[f_2(y+z)] + O(\eta d)$.
Moreover,
\begin{align*}
\E_{\substack{z_1,z_2\sim B_3\\y_1,\ldots,y_{K}\sim B_2}}&\prod_{j=1}^{K}f_2(y_j + z_1)f_2(y_j+z_2) \\
&\le \E_{\substack{z_1,z_2\sim B_3\\y_1,\ldots,y_{K}\sim B_2\\ y'\sim B_4\\z'\sim B_5}}\prod_{j=1}^{K}f_2(y_j + z_1 + y'+z')f_2(y_j+z_2 + y'+z') + O(K^2\eta d)\\
&\le \E_{\substack{z_1,z_2\sim B_3\\y_1,\ldots,y_{K}\sim B_2}}\prod_{\substack{i\in [2]\\j\in [K]}}\Big(\E_{\substack{z_1',z_2'\sim B_5\\y_1',\ldots,y_K'\sim B_4}}\prod_{\substack{\ell\in [2]\\r\in [K]}}f_2(y_j+z_i + y_r'+z_{\ell}')\Big)^{1/(2K)} + O(K^2\eta d)\\
&\le \Big(\prod_{\substack{i\in [2]\\j\in [K]}}\E_{\substack{z_1,z_2\sim B_3\\y_1,\ldots,y_{K}\sim B_2}}\E_{\substack{z_1',z_2'\sim B_5\\y_1',\ldots,y_K'\sim B_4}}\prod_{\substack{\ell\in [2]\\r\in [K]}}f_2(y_j+z_i + y_r'+z_{\ell}')\Big)^{1/(2K)} + O(K^2\eta d)\\
&\le \E_{\substack{y\sim B_2\\z\sim B_3\\z_1',z_2'\sim B_5\\y_1',\ldots,y_K'\sim B_4}}\prod_{\substack{\ell\in [2]\\r\in [K]}}f_2(y+z + y_r'+z_{\ell}') + O(K^2\eta d) \\
&\le \snorm{f_2}_{(B_2,B_4,B_5,K,2)}^{2K} + O(K^2\eta d).
\end{align*}

Recall that we assumed that $\snorm{f_2}_{(B_2,B_4,B_5,K,2)} \le (1+\eps) \E[f_2]$. Also, note that \[ \E_{y \in B_3} \E_{x \in B_2} \wt{f_2}(x+y) \ge -O(\eta d) \ge -\eps^2/1000 \cdot \E[f_2]. \]
Therefore \cref{lem:spectral-pos} applies, and for $L = 10^{-2}\lfloor K \cdot \eps^4 \rfloor$ we get
\[\Big(\E_{\substack{z_1,z_2\sim B_3\\y_1,\ldots,y_{L}\sim B_2}}\prod_{j=1}^{L}\wt{f_2}(y_j + z_1)\wt{f_2}(y_j+z_2)\Big)^{1/(2L)} = O(\eps^{1/2}) \cdot \E[f_2] + O(K^2\eta d).\]

An identical analysis for $f_1$ yields that 
\[\Big(\E_{\substack{x_1,x_2\sim B_1\\y_1,\ldots,y_K\sim B_2}}\prod_{j=1}^K f_1(x_1 - y_j)f_1(x_2-y_j)\Big)\le \snorm{f_1}_{(B_1,B_4,B_5,K,2)}^{2K} + O(K^2 \eta d).\]
Combining these two bounds gives the desired result. 
\end{proof}

We will also consider the case where the both $X$ and $D$ have bounded Gowers grid norms. In this case, the size of the container is indeed as is expected. Below one should view $f_1,f_2$ as the indicator functions of $X,D$.

\begin{lemma}\label{lem:conv-lower-bound-2}
Let $\eps\in (0,1/100)$ and fix an even integer $K\ge 2$. Let $B_1, B_2, \dots$ be a $(d,\eta)$-small sequence of Bohr sets. Consider a triplet of functions $f_1:B_1\to [0,1]$, $f_2:B_1\to [0,1]$, and $g:B_2\to [0,1]$. Suppose that $\snorm{f_1}_{(B_1,B_4,B_5,K,K)}\le (1+\eps) \cdot \E[f_1]$, $\snorm{f_2}_{(B_1,B_4,B_5,K,K)}\le (1+\eps) \cdot \E[f_2]$, $\eta \le \eps^3 \cdot (\E[f_1]\cdot \E[f_2] \cdot \E[g])^{O(K^2)}$, and $K\ge 100\eps^{-8}\log(2/(\E[f_1]\cdot \E[f_2] \cdot \E[g]))$.

Then we have that
\[\Big|\E_{\substack{x\sim B_1\\y\sim B_2}}f_1(x)g(y)f_2(x+y) - \E[f_1]\E[f_2]\E[g]\Big| = O(\eps^{1/2}) \cdot \E[f_1]\E[f_2]\E[g] + O((\eta d)^{1/(2K)}+ e^{-\Omega(\eps^8 K)}).\]
\end{lemma}
\begin{proof}
Observe that 
\begin{align*}
\E_{\substack{x\sim B_1\\y\sim B_2}}f_1(x)g(y)f_2(x+y) &= \E_{\substack{x\sim B_1\\y\sim B_2}}f_1(x)g(y)(f_2(x+y) -\E[f_2]) + \E[f_1]\E[f_2]\E[g]\\
&= \E_{\substack{x\sim B_1\\y\sim B_2\\z\sim B_3}}f_1(x+z)g(y-z)\wt{f_2}(x+y)+ \E[f_1]\E[f_2]\E[g]+ O(\eta d)
\end{align*}
where $\wt{f_2}(z) = f_2(z) - \E_{\substack{z'\sim B_1}}[f_2(z')]$.

We let $L$ be a positive integer to be chosen later. Now observe that 
\begin{align*}
\Big|\E_{\substack{x\sim B_1\\y\sim B_2\\z\sim B_3}}&f_1(x+z)g(y-z)\wt{f_2}(x+y)\Big|\le \E_{\substack{y\sim B_2\\z\sim B_3}}g(y-z)\Big|\E_{x\sim B_1}f_1(x+z)\wt{f_2}(x+y)\Big|\\
&\le (\E_{\substack{y\sim B_2\\z\sim B_3}}g(y-z))^{(L-1)/L} \cdot \Big(\E_{\substack{y\sim B_2\\z\sim B_3}}\E_{x_1,\ldots,x_{L}\sim B_2}\prod_{j=1}^{L}f_1(x_j+z)\wt{f_2}(x_j+y)\Big)^{1/L}\\
&\le (1 + \eps) \cdot \E_{y\sim B_2}g(y) \cdot \Big(\E_{\substack{z_1,z_2\sim B_3\\x_1,\ldots,x_{L}\sim B_1}}\prod_{j=1}^{L}f_1(x_j+z_1)f_1(x_j+z_2)\Big)^{1/(2L)} \\
&\qquad\qquad\qquad\cdot \Big(\E_{\substack{y_1,y_2\sim B_2\\x_1,\ldots,x_{L}\sim B_1}}\prod_{j=1}^{L}\wt{f_2}(x_j + y_1)\wt{f_2}(x_j+y_2)\Big)^{1/(2L)} + O(\eta d + e^{-\Omega(\eps L)}).
\end{align*}
In order to complete our analysis it suffices to bound the last two terms: 
\[
    \Big(\E_{\substack{z_1,z_2\sim B_3\\x_1,\ldots,x_{L}\sim B_1}}\prod_{j=1}^{L}f_1(x_j+z_1)f_1(x_j+z_2)\Big)^{1/(2L)} \text{ and } \Big(\E_{\substack{y_1,y_2\sim B_2\\x_1,\ldots,x_{L}\sim B_1}}\prod_{j=1}^{L}\wt{f_2}(x_j + y_1)\wt{f_2}(x_j+y_2)\Big)^{1/(2L)}.
\] 

We now consider the second term involving $\wt{f}_2$ where we will need to invoke spectral positivity. For this term, observe that 
\[\min_{y\in B_2}\E_{x\sim B_1}[\wt{f_2}(x+y)]\ge -O(\eta d).\]
Furthermore observe that by H\"{o}lder's inequality that $\snorm{f_2}_{(B_2,B_4,B_5,K,2)}\le \snorm{f_2}_{(B_2,B_4,B_5,K,K)} \le (1+\eps)\cdot \E_{y\sim B_2\\z\sim B_3}[f_2(y+z)] + O(\eta d)$. 
Moreover,
\begin{align*}
\E_{\substack{y_1,y_2\sim B_2\\x_1,\ldots,x_{K}\sim B_1}}&\prod_{j=1}^{K}f_2(x_j + y_1)f_2(x_j+y_2) \\
&\le \E_{\substack{y_1,y_2\sim B_2\\x_1,\ldots,x_{K}\sim B_1\\ x'\sim B_4\\y'\sim B_5}}\prod_{j=1}^{K}f_2(x_j + y_1 + x'+y')f_2(x_j+y_2 + x'+y') + O(K^2\eta d)\\
&\le \E_{\substack{y_1,y_2\sim B_2\\x_1,\ldots,x_{K}\sim B_1}}\prod_{\substack{i\in [2]\\j\in [K]}}\Big(\E_{\substack{y_1',y_2'\sim B_5\\x_1',\ldots,x_K'\sim B_4}}\prod_{\substack{\ell\in [2]\\r\in [K]}}f_2(y_i+x_j + y_r'+x_{\ell}')\Big)^{1/(2K)} + O(K^2\eta d)\\
&\le \Big(\prod_{\substack{i\in [2]\\j\in [K]}}\E_{\substack{y_1,y_2\sim B_2\\x_1,\ldots,x_{K}\sim B_1}}\E_{\substack{y_1',y_2'\sim B_5\\x_1',\ldots,x_K'\sim B_4}}\prod_{\substack{\ell\in [2]\\r\in [K]}}f_2(y_i+x_j + y_r'+x_{\ell}')\Big)^{1/(2K)} + O(K^2\eta d)\\
&\le \E_{\substack{y\sim B_2\\x\sim B_1\\z_1',z_2'\sim B_5\\x_1',\ldots,x_K'\sim B_4}}\prod_{\substack{\ell\in [2]\\r\in [K]}}f_2(y+x + y_r'+x_{\ell}') + O(K^2 \eta d) \\
&\le \snorm{f_2}_{(B_2,B_4,B_5,K,2)}^{2K} + O(K^2\eta d).
\end{align*}

Therefore by \cref{lem:spectral-pos}, if we set $L = 10^{-2}\lfloor K \cdot \eps^4 \rfloor$ then 
\[\Big(\E_{\substack{y_1,y_2\sim B_2\\x_1,\ldots,x_{L}\sim B_1}}\prod_{j=1}^{L}\wt{f_2}(x_j + y_1)\wt{f_2}(x_j+y_2)\Big)^{1/(2L)} = O(\eps^{1/2}) \cdot \E[f_2] + O((\eta d)^{1/(2L)}).\]

An identical analysis for $f_1$ yields that 
\[\Big(\E_{\substack{z_1,z_2\sim B_3\\x_1,\ldots,x_K\sim B_1}}\prod_{j=1}^{K}f_1(x_j+z_1)f_1(x_j+z_2)\Big)\le \snorm{f_1}_{(B_1,B_4,B_5,K,2)}^{2K} + O(K^2 \eta d).\]
Combining these two bounds gives the desired result. 
\end{proof}

In this setting we also require a weaker notion of spreadness which we call $(B_1, B_2, \eps)$ $\ell_1$-spreadness. Informally, a set $D \subseteq B_1$ satisfies the $\ell_1$-spreadness property if, when $B_1$ is averaged on $B_2$ shifts, then for almost all the shifts, the density of $D$ doesn't deviate too far from its global density.
\begin{definition}[$\ell_1$-spreadness]
\label{def:l1spread}
Let $B_1, B_2$ be regular Bohr sets with the same frequencies and $\nu(B_2) \le \nu(B_1)$. We say that a function $f: B_1 \to [0, 1]$ is \emph{$(B_1, B_2, \eps)$ $\ell_1$-spread} if
\[ \E_{x \sim B_1} \left|\E_{y \sim B_2}[f(x+y)] - \E[f]\right| \le \eps \cdot \E[f]. \]
\end{definition}
We say that a set $D \subseteq B_1$ is $\ell_1$-spread if its indicator function is, according to \cref{def:l1spread}. The next lemma proves that $\E[f_1g] \approx \E[f_1]\E[g]$, which is useful given \cref{lem:conv-lower-bound}, holds as long as $f_1$ has small Gowers grid norm and $g$ is $\ell_1$-spread, which is a much weaker hypothesis on $g$ than having small Gowers grid norm.

\begin{lemma}\label{lem:product-spread}
Let $\eps\in (0,1/100)$ and fix an even integer $K\ge 2$. Let $B_1, B_2, \dots$ be a $(d,\eta)$-small sequence of Bohr sets. Consider a pair of functions $f:B_1\to [0,1]$ and $g:B_1\to [0,1]$. Suppose that $\snorm{f}_{(B_1,B_3,B_4,K,K)}\le (1+\eps) \cdot \E[f]$, $g$ is $(B_1,B_4,\eps)$ $\ell_1$-spread, $\eta \le \eps^3 \cdot (\E[f]\cdot \E[g])^{O(K^2)} \cdot d^{-O(1)}$, and $K\ge 100\eps^{-8}\log(2/(d \cdot \E[f]\cdot \E[g]))$.

Then we have that
\[\Big|\E_{\substack{x\sim B_1\\y\sim B_2}}f(x)g(x+y) - \E_{x\sim B_1}[f(x)]\cdot \E_{x\sim B_1}[g(x)]\Big| = O(\eps^{1/2}) \cdot \E[f]\E[g] + O((\eta d)^{1/(2K)}+ e^{-\Omega(\eps^8 K)}).\]
\end{lemma}
\begin{proof}
For the upper bound, first write
\[ \E_{\substack{x\sim B_1\\y\sim B_2}}f(x)g(x+y) \le \E_{\substack{x\sim B_1\\y\sim B_2}}f(x+y)g(x) + O(\eta d). \]
Now apply \cref{lem:upper-bound} with $f_1$ as $g$, $f_2 = 1$, and the $g$ in \cref{lem:upper-bound} as $f_1$ with $\tau = \snorm{f}_{(B_1,B_3,B_4,K,K)}$ (and the same $\eta$ and $K$), which gives the desired result.

For the lower bound, define
\begin{align*}
\mc{E}_1(x, y) &= \mbm{1}\Big[\Big(\E_{\substack{x_1',\ldots,x_K'\sim B_3\\y_1',\ldots,y_K'\sim B_4}}\prod_{\substack{i\in [K]\\j\in [K]}}f(x+y+x_i' + y_i') \ge ((1+2\eps) \cdot \E[f]\Big)^{K^2}\Big].
\end{align*}
To bound $\E[\mc{E}_1(x,y)]$, note that
\begin{align*}
&\mbm{1}\Big[\E_{\substack{x_1',\ldots,x_K'\sim B_3\\y_1',\ldots,y_K'\sim B_4}}\prod_{\substack{i\in [K]\\j\in [K]}}f(x+y+x_i' + y_i') \ge ((1+2\eps) \cdot \E[f])^{K^2}\Big] \\
& \qquad \qquad \qquad \le ((1+2\eps)\E[f])^{-K^2} \E_{\substack{x \in B_1 \\y \in B_2}} \E_{\substack{x_1',\ldots,x_K'\sim B_3\\y_1',\ldots,y_K'\sim B_4}} \prod_{\substack{i\in [K]\\j\in [K]}}f(x+y+x_i' + y_i') \\
& \qquad \qquad \qquad \le ((1+2\eps)\E[f])^{-K^2}\Big(\|f\|_{B_1,B_3,B_4,K,K} + O(\eta d) \Big) \le e^{-\Omega(\eps K^2)} + O(\E[f]^{-K^2} \eta d),
\end{align*}
because $\|f\|_{B_1,B_3,B_4,K,K} \le (1+\eps)\E[f]$.

Furthermore observe that if $\mc{E}(x,y) = 0$, then monotonicity of the $(K,K)$-Gowers grid norms (\cref{lem:monotone}) gives that $\E_{z\sim B_3}f(x+y + z)\le (1+3\eps) \cdot \E[f] + O(\eta d)$.
Now define
\[\mc{E}_2(x,y,y') =\mbm{1}\Big[\E_{\substack{z\sim B_3}}f(x+y+z) \le (1-\eps^{1/2}) \cdot \E[f]\Big] + \mbm{1}\Big[\E_{\substack{z\sim B_3}}g(x+y'+z)\le (1-\eps^{1/2}) \cdot \E[g]\Big]. \]
\cref{fct:rev-mark} (applied to $\min(\E_{z\sim B_3}f(x+y' + z),(1+3\eps)\cdot \E[f])$ implies that
\[ \E_{x \sim B_1, y \sim B_2}\Big[\mbm{1}\Big[\E_{\substack{z\sim B_3}}f(x+y+z) \le (1-\eps^{1/2}) \cdot \E[f]\Big]\Big] \le O(\eps^{1/2} + \eta d), \]
and the $(B_1, B_4, \eps)$ $\ell_1$-spreadness of $g$ and Markov's inequality imply that
\[ \E_{x \sim B_1, y' \sim B_2}\Big[\mbm{1}\Big[\E_{\substack{z\sim B_3}}g(x+y'+z)\le (1-\eps^{1/2}) \cdot \E[g]\Big]\Big] \le O(\eps^{1/2} + \eta d). \]
In total, we have concluded that
\[\E[\mc{E}_2(x,y,y')]\le O(\eps^{1/2} + \eta d).\]
Thus 
\begin{align*}
\E_{\substack{x\sim B_1\\y\sim B_2}}f(x)g(x+y) &\ge \E_{\substack{x\sim B_1\\y,y'\sim B_2\\z,z'\sim B_3}}f(x+y'+z)g(x+y+z') - O(\eta d)\\
&\ge \E_{\substack{x\sim B_1\\y,y'\sim B_2}}(1-\mc{E}_2(x,y,y')) (\E_{z\sim B_3}f(x+y'+z))(\E_{z'\sim B_4}g(x+y+z')) - O(\eta d)\\
&\ge (1-O(\eps^{1/2})) \cdot \E[f] \cdot \E[g] - O(\eta d)
\end{align*}
which implies the desired result. 
\end{proof}

\subsection{Algebraic spreadness and Gowers grid norms}
We now give the required analog of ``algebraic'' spreadness in the case of Bohr sets. 
\begin{definition}[Algebraic spreadness]
\label{def:bohrspread}
Let $B \subseteq G$ be a regular Bohr set. We say that $X \subseteq B$ is \emph{$(r, \eta_s, \eps)$-algebraically spread} within $B$ if for all regular Bohr sets $B' \subseteq B$ with $\rank(B') \le \rank(B) + r$ and radius $\nu(B') \ge \eta_s \nu(B)$, and for all $x \in G$ it holds that
\[ \frac{|(X - x) \cap B'|}{|B'|} \le (1+\eps)\frac{|X|}{|B|}. \]
\end{definition}

The key technical input from the work of Kelley and Meka \cite{KM23} is converting between algebraic spreadness and the guarantee that a function $f$ does not correlate with ``convolutions''. The statement we need appears with a careful proof in the work of Filmus, Hatami, Hosseini, and Kelman \cite{FHHK24} but not as a standalone statement. We reproduce a proof (assuming some known almost periodicity results) in \cref{sec:almost-period}.
\begin{theorem}\label{thm:ap+KM}
Let $\eta,\eps\in (0,1/2)$ and $k\ge 1$. Let $B_1, B_2, \dots$ be a $(d,\eta)$-small sequence of Bohr sets. Let $f_i:B_i\to [0,1]$ and $g:B_1\to [0,1]$.

Suppose that $\E_{x\sim B_i}[f_i(x)]\ge 2^{-k}$ and $\E_{x\sim B_1}[g(x)]\ge 2^{-k}$ and $\eta \le O(d^{-2} \cdot 2^{-O(k^2\log(1/\eps)/\eps)})$. Furthermore suppose that 
\[\E_{\substack{x\sim B_1\\y\sim B_2}}[f_1(x)f_2(y)g(x+y)]\ge (1+\eps) \cdot \E[g] \cdot \E[f_1] \cdot \E[f_2].\]

Then there exists a regular Bohr set $B'\subseteq B_1$ with radius $r'$ and dimension $d+d'$ such that 
\[r' \ge r_2 \cdot \eps \cdot d^{-4}\cdot \eta \cdot 2^{-O(k^2\log(1/\eps)/\eps)}\text{ and }d' = O(k^{8} \cdot \eps^{-9}) \]
and $x^{\ast}\in B_1$ such that 
\[\E_{x\sim B'}[g(x+x^{\ast})]\ge (1+\eps/2) \cdot \E_{x\sim B_1}[g(x)].\]
\end{theorem}

The last remaining technical issue for this section is connecting algebraic spreadness (\cref{def:bohrspread}) with the Gowers grid norms (\cref{def:gow-high}). This is essentially an immediate consequence of sifting. 
\begin{theorem}\label{thm:bohrkk}
Fix an integer $K\ge 1$, $\eta\in (0,1/2)$, and $\tau\in (0,1)$. Let $B_1, B_2, \dots$ be a $(d,\eta)$-small sequence of Bohr sets. Define $f:G\to [0,1]$ such that 
\[\snorm{f}_{(B_1,B_2,B_3,K,K)}\ge \tau.\]

Then for any $\eps > 0$ there exists $g_1:B_2\to [0,1]$, $g_2:B_3\to [0,1]$, and $x\in B_1$ such that
\[\E_{\substack{y\sim B_2\\z\sim B_3}}[g_1(y)g_2(z)f(x + y + z)]\ge (1-\eps) \cdot \tau \cdot \E_{\substack{y\sim B_2\\z\sim B_3}}[g_1(y)g_2(z)]\]
and $\E[g_i]\ge \eps \cdot \tau^{O(K)}$.
\end{theorem}
\begin{proof}
From our assumption, we immediately have 
\[\sup_{x\in B_1}\E_{\substack{y_1,\ldots,y_K\sim B_2\\z_1,\ldots,z_K\sim B_3}}\prod_{\substack{i\in [K]\\j\in [K]}}f(x+y_i+z_j)\ge \E_{\substack{x\sim B_1\\y_1,\ldots,y_K\sim B_2\\z_1,\ldots,z_K\sim B_3}}\prod_{\substack{i\in [K]\\j\in [K]}}f(x+y_i+z_j) \ge \tau^{K^2}.\]
Fix $x$ achieving the supremum and define $F(y,z) = f(x+y+z)$. We may apply \cref{thm:sift} and obtain $g_1:B_2\to [0,1]$ and $g_2:B_3\to [0,1]$ with $\E[g_i]\ge \eps \cdot \tau^{O(K)}$ and 
\[\E_{\substack{y\sim B_2\\z\sim B_3}}[g_1(y)g_2(z)f(x + y + z)]\ge (1-\eps) \cdot \tau \cdot \E_{\substack{y\sim B_2\\z\sim B_3}}[g_1(y)g_2(z)]. \qedhere\]
\end{proof}

\subsection{Pseudorandomization}

We note the easy claim that every set contains a somewhat large subset that is spread within a smaller Bohr set. This is a bit more complicated than the analogous \cref{claim:spread} because we need to work to make the sets that we increment onto also regular Bohr sets.

\begin{claim}
\label{claim:spreadbohr}
Let $r$ be an integer, $B \subseteq G$ be a regular Bohr set with $d = \rank(B)$, and $\eps, \eta_s \in (0, 1)$. For $X \subseteq B$ with $\delta \coloneqq |X|/|B|$ there is a regular Bohr set $B' \subseteq B$ and $x \in G$ such that $X' \coloneqq X \cap (x+B')$ is $(r, \eta_s, \eps)$-algebraically spread within $x+B'$, $\rank(B') \le d + O(r\eps^{-1}\log(1/\delta))$, $\frac{|X'|}{|B'|} \ge \delta$, and $\nu(B') \ge \nu(B) \cdot (\eps\delta\eta_s/(2d'))^{O(\eps^{-1}\log(1/\delta))}$, where $d' \coloneqq d + O(r\eps^{-1}\log(1/\delta))$.
\end{claim}
\begin{proof}
We proceed iteratively. Initialize $B^{(0)} = B$. If $X \cap B^{(0)}$ is $(r,\eta_s,\eps)$-algebraically spread within $B^{(0)}$, we are done. Otherwise, there must exist $x \in G$ and a Bohr set $B^{(0)'} \subseteq B^{(0)}$ with $\rank(B^{(0)'}) \le \rank(B^{(0)}) + r$ and $\nu(B^{(0)'}) \ge \eta_s \cdot \nu(B^{(0)})$ satisfying
\[
    \frac{|X \cap (x+B^{(0)'})|}{|B^{(0)'}|} \ge (1+\eps)\frac{|X \cap B^{(0)}|}{|B^{(0)}|}.
\]
Now, we want to simply iterate the argument, but we must be a bit careful because $B^{(0)'}$ is not necessarily regular (and we wish to output a regular Bohr set). Instead let $B^{(1)} = \eta'B^{(0)'}$ for some $\eta' \in [\frac{\eps\delta}{2Cd'}, \frac{\eps\delta}{Cd'}]$ for sufficiently large $C$ so that $B^{(1)}$ is regular. Let $X^{(0)'} = X \cap (x + B^{(0)'})$. Note that
\[ \E_{y \in B^{(0)'}} |X^{(0)'} \cap (y + x + B^{(1)})| = \frac{|B^{(1)}|}{|B^{(0)'}|} |X^{(0)'}| \pm O(\eta'd'|B^{(1)}|). \]
Thus there is some $y \in B^{(0)'}$ so that
\[ \frac{|X^{(0)'} \cap (y + x + B^{(1)})|}{|B^{(1)}|} \ge \frac{|X^{(0)'}|}{|B^{(0)'}|} - O(\eta'd') \ge (1+\eps/2)\delta. \]
Now set $X^{(1)} \coloneqq X^{(0)'} \cap (y + x + B^{(1)})$ and iterate the same process starting with $X^{(1)} \subseteq y+x+B^{(1)}$. The number of iterations is at most $O(\eps^{-1}\log(1/\delta))$. The conclusions follow because after $i$ iterations we know that $\rank(B^{(i)}) \le d + ri$, $\nu(B^{(i)}) \ge (\eps \delta \eta_s / (2 C d'))^{i} \nu(B)$, and $\frac{|X^{(i)}|}{|B^{(i)}|} \ge (1+\eps/2)^i \delta$.
\end{proof}

In the setting of general abelian groups we will pseudorandomize all three sets $X$, $Y$, and $D$. $X$ and $Y$ will be made to be algebraically spread (see \cref{def:bohrspread}), while $D$ will be made to be $\ell_1$-spread (see \cref{def:l1spread}) -- we do not see how to guarantee that $D$ is also algebraically spread.
This contrasts with the situation in the finite field setting where we only made $X$ and $Y$ algebraically spread. We make $D$ $\ell_1$-spread in order to establish that the size of $S(X, Y, D)$ is close to expected (see \cref{lem:conv-lower-bound,lem:product-spread}).

Our proof proceeds in two phases. First, we will prove a lemma which allows us to pseudorandomize both $X$ and $Y$, analogous to how \cref{thm:algopseudo} is proven. Then we will do a very particular partitioning algorithm to $D$ to ensure $\ell_1$-spreadness. This might ruin the algebraic spreadness of $X$ and $Y$ on a small fraction of pieces, which we then fix with recursion.

To start, we prove the analogue of \cref{lemma:2dim}, which said that we can partition $X \times Y$ into $X_i \times Y_i$, all of which are spread and not too small, except for a very small fraction. In this setting, we will not obtain a true partition, but instead a distribution over $X_i \times Y_i$ that on average covers each element of $X \times Y$ once (again, minus some small fraction). We also note that to interact with the density increment argument in \cref{sec:dsbohr} that the initial conditions for the lemma are a bit different. While in \cref{lemma:2dim} we started with $X, Y, D$ such that at least one of $X$ or $Y$ was algebraically spread, in this new setting we instead just have that $X$ is upper bounded in $(K, K)$-Gowers grid norm. The argument of \cref{lemma:2dim} still carries over even with these changes.

\begin{lemma}
\label{lemma:2dimbohr}
Let $r$ be a positive integer and $\eps, \eta, \eta_s, \beta > 0$. Let $B_1, B_2, \dots$ be a $(d,\eta)$-small sequence of Bohr sets. Let $X \subseteq B_1$, $Y \subseteq B_2$, and let $\delta_X = |X|/|B_1|$ and $\delta_Y = |Y|/|B_2|$.
Define \[ g \coloneqq O(\eps^{-1}\log(1/(\delta_X\delta_Y)) + \eps^{-1}\log(1/\beta)^2), \]
$d' = \rank(B_1) + O(rg^2\log(1/\beta))$, and assume that $\eta  \le (\delta_X\delta_Y/d')^{O(1)}2^{-O(\log(1/\beta)^2)}$.
Then there is a measure $\mu$ over rectangles $X_i \times Y_i \subseteq X \times Y$ such that:
\begin{enumerate}
    \item Each $X_i \times Y_i \in \supp(\mu)$ satisfies that $(X_i - x_i)$ is $(r,\eta_s,\eps)$-spread in $B_1^{(i)}$ for some $x_i$ and $(Y_i - y_i)$ is $(r,\eta_s,\eps)$-spread in $B_2^{(i)}$ for some $y_i$, where $B_2^{(i)} \subseteq B_1^{(i)}$ are regular Bohr sets with the same frequencies and $\nu(B_2^{(i)})/\nu(B_1^{(i)}) \in [\eta/2, \eta]$.
    \item Additionally, $\rank(B_1^{(i)}) \le \rank(B_1) + O(rg^2\log(1/\beta))$.
    \item Also, $\nu(B_1^{(i)}) \ge \eta_s^{O(g^2\log(1/\beta))}\nu(B_1)$.
    \item For any $1$-bounded function $f: X \times Y \to [-1,1]$ it holds that
    \[ \E_{x \in X, y \in Y} f(x, y) = \E_{X_i \times Y_i \sim \mu} \E_{x \in X_i, y \in Y_i} f(x, y) \pm O(\beta). \]
    \item For any $X_i \times Y_i \in \supp(\mu)$ it holds that $\frac{|X_i||Y_i|}{|B_1^{(i)}||B_2^{(i)}|} \ge e^{-O(\log(1/\beta)^2)}\delta_X\delta_Y$.
\end{enumerate}
\end{lemma}
Again, we prove this by first establishing a one round partitioning statement which is then used recursively.
\begin{lemma}
\label{lemma:oneroundbohr}
Let $r$ be a positive integer and $\eps, \eta, \eta_s, \beta > 0$. Let $B_1, B_2, \dots$ be a $(d,\eta)$-small sequence of Bohr sets. Let $X \subseteq B_1$, $Y \subseteq B_2$, and let $\delta_X = |X|/|B_1|$ and $\delta_Y = |Y|/|B_2|$. Define $g \coloneqq O(\eps^{-1}\log(1/(\delta_X\delta_Y\beta)))$, $d' = \rank(B_1) + O(rg^2)$, and assume that $\eta  \le (\beta\delta_X\delta_Y/d')^{O(1)}$. Then there is a measure $\mu$ over rectangles $X_i \times Y_i \subseteq X \times Y$, and an event $\mc{E}(X_i, Y_i)$ such that:
\begin{enumerate}
    \item Each $X_i \times Y_i \in \supp(\mu)$ satisfies that $(X_i - x_i) \subseteq B_1^{(i)}$ for some $x_i$ and $(Y_i - y_i) \subseteq B_2^{(i)}$ for some $y_i$, where $B_2^{(i)} \subseteq B_1^{(i)}$ are regular Bohr sets with the same frequencies and $\nu(B_2^{(i)})/\nu(B_1^{(i)}) \in [\eta/2, \eta]$.
    \item Additionally, $\rank(B_1^{(i)}) \le \rank(B_1) + O(rg^2)$.
    \item Also, $\nu(B_1^{(i)}) \ge \eta_s^{O(g^2)}\nu(B_1)$.
    \item For any $1$-bounded function $f: X \times Y \to [-1, 1]$ it holds that
    \[ \E_{x \in X, y \in Y} f(x, y) = \E_{X_i \times Y_i \sim \mu} \E_{x \in X_i, y \in Y_i} f(x, y) \pm O(\beta^2). \]
    \item For any $X_i \times Y_i \in \supp(\mu)$ it holds that $\frac{|X_i||Y_i|}{|B_1^{(i)}||B_2^{(i)}|} \ge \beta^{O(1)}\delta_X\delta_Y$.
    \item If $\mc{E}$ holds then $X_i - x_i \subseteq B_1^{(i)}$ and $Y_i - y_i \subseteq B_2^{(i)}$ are $(r, \eta_s, \eps)$-algebraically spread.
    \item $\E_{X_i \times Y_i \sim \mu}[\mc{E}(X_i, Y_i)] \ge 1/2$.
\end{enumerate}
\end{lemma}
\begin{proof}
Let $X^{(0)} = X$ and perform the following algorithm. For $t = 0, 1, \dots$: if $|X^{(t)}| \le \frac{\beta^2}{10}|X|$, then terminate. Otherwise, let $X_t \subseteq X^{(t)}$ be $(r', \eta_s', \eps/5)$-algebraically spread within $x_t + \mc{B}_t$ for $r' = \Omega(r\eps^{-1}\log(1/(\delta_Y\beta)))$ and $\eta_s^{O(\eps^{-1}\log(1/(\delta_Y\beta)))}$, as given by \cref{claim:spreadbohr}. Note that
\[\rank(\mc{B}_t) \le \rank(B_1) + O(r'\eps^{-1}\log(1/(\delta_X \beta))). \]
Now let $X^{(t+1)} = X^{(t)} \setminus X_t$. Let $T$ be the total number of iterations, so that
\[ X = X^{(T)} \cup X_0 \cup X_1 \cup \dots \cup X_{T-1}. \]
Now fix a $t \in \{0, 1, \dots, T-1\}$. Let $\mc{B}_t' = \eta'\mc{B}_t$ be regular for some $\eta' \in [\eta^2/2,\eta^2]$. Now for $x \in B_2$ define $Y_{t,x} \coloneqq Y \cap (x + \mc{B}_t')$, and partition
\[ Y_{t,x} = Y_{t,x}^{(T'_{t,x})} \cup Y_{t,x,0} \cup \dots \cup Y_{t,x,T'_{t,x}-1}, \]
using the algorithm in the first paragraph, where $|Y_{t,x}^{(T'_{t,x})}| \le \frac{\beta^2}{10}\delta_Y |\mc{B}_t'|$ and each $Y_{t,x,t'}$ for $0 \le t' \le T'_{t,x}-1$ is $(r, \eta_s, \eps)$-algebraically spread within some shifted regular Bohr set $y_{t,x,t'} + \mc{B}_{t,x,t'}'$ where
\[ \rank(\mc{B}_{t,x,t'}') \le \rank(\mc{B}_t) + O(r\eps^{-1}\log(1/(\delta_Y\beta))). \]
Let $\mc{B}_{t,x,t'}$ be a regular Bohr set with $\mc{B}_{t,x,t'}' = \eta''\mc{B}_{t,x,t'}$ for $\eta'' \in [\eta/2,\eta]$. Finally, for each $y \in \mc{B}_t$ define $X_{t,x,t',y} \coloneqq X_t \cap (x_t + y + \mc{B}_{t,x,t'})$ -- note that for every fixed $t, x, t'$ that the average of the indicator functions of $X_{t,x,t',y}$ is approximately the indicator function of $X_t$ in $\ell_1$ error, by \cref{lem:shift-invar} up to scaling.

Now we define the distribution $\mu$ over the pieces $X_i \times Y_i$ and the event $\mc{E}$. The pieces $X_i \times Y_i$ are all pieces \[ X_{t,x,t',y} \times Y_{t,x,t'} \subseteq (x_t + y + \mc{B}_{t,x,t'}) \times (y_{t,x,t'} + \mc{B}_{t,x,t'}') \] where $\frac{|X_{t,x,t',y}|}{|\mc{B}_{t,x,t'}|} \ge \frac{\beta^4}{100}\delta_X$. The probability mass in $\mu$ of such a piece is defined as: $\frac{|X_{t,x,t,y'}||Y_{t,x,t'}|}{|X||Y||\mc{B}_t'||\mc{B}_{t,x,t'}|}.$ $\mu$ may not have total probability mass $1$, so scale it appropriately so that it does -- later we argue that the scaling is only on the order of $1 \pm O(\beta^2)$.
Finally, $\mc{E}$ consists of the pieces with $\frac{|X_{t,x,t',y}|}{|\mc{B}_{t,x,t'}|} \ge (1-3\eps/5)\frac{|X_t|}{|\mc{B}_t|}$.

Let us now verify all the conclusions. (1) follows because by construction. (2) follows because
\begin{align*}
\rank(\mc{B}_{t,x,t'}) &\le \rank(\mc{B}_t) + O(r\eps^{-1}\log(1/(\delta_Y\beta))) \\
&\le \rank(\mc{B}_1) + O(r'\eps^{-1}\log(1/(\delta_X\beta))) + O(r\eps^{-1}\log(1/(\delta_Y\beta))),
\end{align*}
as desired. (3) follows similarly. Towards checking (4), first define $\mc{I}$ to be the set of tuples $(t,x,t',y)$ with $\frac{|X_{t,x,t',y}|}{|\mc{B}_{t,x,t'}|} \le \frac{\beta^4}{100}\delta_X$. 
For a $1$-bounded function $f: X \times Y \to \mb{R}$ we first write
\[ \E_{x \in X, y \in Y}f(x,y) = \sum_{t=0}^{T-1} \frac{|X_t|}{|X|} \E_{x \in X_t, y \in Y} f(x, y) \pm \frac{|X^{(T)}|}{|X|}. \] Also, $|X^{(T)}|/|X| \le \beta^2/10$ by construction. By \cref{lem:shift-invar} we know that
\begin{align*}
    \E_{x \in X_t, y \in Y} f(x, y) = \E_{x \in B_2} \delta_Y^{-1}\frac{|Y_{t,x}|}{|\mc{B}_t'|} \E_{x_1 \in X_t, y_1 \in Y_{t,x}} f(x_1, y_1) \pm \eta d' \delta_Y^{-1}.
\end{align*}
Continuing, we can bound
\[ \delta^{-1}\frac{|Y_{t,x}|}{|\mc{B}_t'|} \E_{x_1 \in X_t, y_1 \in Y_{t,x}} f(x_1, y_1) = \sum_{t'=0}^{T'_{t,x}-1} \delta_Y^{-1}\frac{|Y_{t,x,t'}|}{|\mc{B}_t'|} \E_{x_1 \in X_t, y_1 \in Y_{t,x,t'}} f(x_1, y_1) \pm \delta_Y^{-1}\frac{|Y_{t,x}^{(T')}|}{|\mc{B}_t'|}. \]
Now $\delta_Y^{-1}\frac{|Y_{t,x}^{(T')}|}{|\mc{B}_t'|} \le \frac{\beta^2}{10}$ by definition, so that error is acceptable. Finally, we can write
\[ \E_{x_1 \in X_t, y_1 \in Y_{t,x,t'}} f(x_1, y_1) = \E_{y \in \mc{B}_t} \frac{|\mc{B}_t|}{|X_t|}\frac{|X_{t,x,t',y}|}{|\mc{B}_{t,x,t'}|}\E_{x_1 \in X_{t,x,t',y}, y_1 \in Y_{t,x,t'}} f(x_1, y_1) + O\left(\eta d' \frac{|\mc{B}_t|}{|X_t|} \right). \]
Note that $\frac{|\mc{B}_t|}{|X_t|} \le O(\beta^{-2}\delta_X^{-1})$ by definition, so this error is acceptable. Finally, we have to bound the error contribution from tuples $(t,x,t',y) \notin \mc{I}$. Because $\frac{|\mc{B}_t|}{|X_t|} \le \frac{10}{\beta^2\delta_X}$, the total contribution is bounded by
\[ \E_y \left[\mbm{1}_{(t,x,t',y) \notin \mc{I}} \frac{|\mc{B}_t|}{|X_t|}\frac{|X_{t,x,t',y}|}{|\mc{B}_{t,x,t'}|} \right] \le \frac{10}{\beta^2\delta_X} \frac{\beta^4\delta_X}{100} \le \beta^2/10. \]
Thus the total accumulated error is $O(\beta^2 + \eta d' \delta_X^{-1}\delta_Y^{-1}\beta^{-2}) = O(\beta^2)$ as desired. This also argues that the scaling of $\mu$ that needs to be done is $1 + O(\beta^2)$ as claimed earlier.

(5) follows by by construction.

To verify (6), note that each $Y_{t,x,t'} \subseteq y_{t,x,t'} + \mc{B}_{t,x,t'}'$ is $(r, \eta_s, \eps)$-algebraically spread by construction. Note that for any shifted Bohr set $b + B$ with $\rank(B) \le \rank(\mc{B}_{t,x,t'}') + r$ it holds that
\[ \frac{|X_{t,x,t',y} \cap (b + B)|}{|B|} \le (1+\eps/5)\frac{|X_t|}{|\mc{B}_t|}, \]
because $\rank(B) \le \rank(\mc{B}_t) + r'$ and $X_t$ was $(r', \eta, \eps)$-algebraically spread by construction. Thus if $\frac{|X_{t,x,t',y}|}{|\mc{B}_{t,x,t'}|} \ge (1-3\eps/5)\frac{|X_t|}{|\mc{B}_t|}$ (as is true for the tuples in $\mc{E}$), then $X_{t,x,t',y} \subseteq x_t + y + \mc{B}_{t,x,t'}$ is $(r, \eta_s, \eps)$-algebraically spread.

For item (7), first consider fixing some $t, x, t'$, and let $\delta_t = \frac{|X_t|}{|\mc{B}_t|}$ to simplify notation. Note that
\[ \E_{y \in \mc{B}_t}\left[\frac{|X_{t,x,t',y}|}{|\mc{B}_{t,x,t'}|} \right] = \delta_t \pm O(\eta d' \delta_X^{-1}\delta_Y^{-1}\beta^{-2}) \ge (1-\eps/10)\delta_t \] as above, and that $\frac{|X_{t,x,t',y}|}{|\mc{B}_{t,x,t'}|} \le (1+\eps/5)\delta_t$ for all $y$, as argued above. Thus, by the reverse Markov inequality (\cref{fct:rev-mark}) we have
\[ \Pr_{y \in \mc{B}_t}\left[\frac{|X_{t,x,t',y}|}{|\mc{B}_{t,x,t'}|} \ge (1-3\eps/5)\delta_t \right] \ge 5/8. \]
Thus,
\[ \E_{y \in \mc{B}_t : \frac{|X_{t,x,t',y}|}{|\mc{B}_{t,x,t'}|} \ge (1-3\eps/5)\delta_t} |X_{t,x,t',y}| \ge \frac58(1-3\eps/5)\delta_t|\mc{B}_{t,x,t'}| \ge \frac35|X_t| \cdot \frac{|\mc{B}_{t,x,t'}|}{|\mc{B}_t|}. \]
Combining this with the fact with item (4) (that at most $O(\beta^2 + \eta d' \delta_X^{-1}\delta_Y^{-1}\beta^{-2})$ mass is thrown out) completes the proof.
\end{proof}

Iterating \cref{lemma:oneroundbohr} proves \cref{lemma:2dimbohr}.

\begin{proof}[Proof of \cref{lemma:2dimbohr}]
Perform the following algorithm given $X \times Y \subseteq B_1 \times B_2$. Return $X_i \times Y_i \subseteq (x_i + B_1^{(i)}) \times (y_i + B_2^{(i)})$ for $i = 1, \dots, T$ as given in \cref{lemma:oneroundbohr}. Now, for each $i \in [T]$ such that $\mc{E}(X_i, Y_i)$ does not hold, invoke \cref{lemma:oneroundbohr} recursively. Terminate when the recursion depth is $O(\log(1/\beta))$, and remove any remaining pieces from the partition.

Let's check that all the conclusions of \cref{lemma:2dimbohr} hold. First note that for all pieces $X' \times Y'$ in the partition that $\delta_{X'}\delta_{Y'} \ge 2^{-O(\log(1/\beta)^2)}\delta_X\delta_Y$. So these are the parameters $\delta_X$ and $\delta_Y$ that one applies \cref{lemma:oneroundbohr} with. Now, (1) holds by construction. (2), (3), (5) hold by applying (2), (3), (5) of \cref{lemma:oneroundbohr} respectively for recursion depth $O(\log(1/\beta))$ -- note that
\[ g \le O(\eps^{-1}\log(1/(\delta_{X'}\delta_{Y'}))) \le O(\eps^{-1}\log(1/(\delta_X\delta_Y)) + \eps^{-1}\log(1/\beta)^2). \] (4) holds by using item (4) of \cref{lemma:oneroundbohr}, plus that after $O(\log(1/\beta))$ rounds the total size of the ``bad event'' (not in $\mc{E}$) is at most $2^{-O(\log(1/\beta))} \le \beta^{O(1)}$.
\end{proof}

Applying \cref{lemma:2dimbohr} gives a way to split $S(X, Y, D)$ into $S(X_i, Y_i, D)$ where $X_i \times Y_i$ are a partition of measure of $X \times Y$ (minus some small error). Say that $X_i \subseteq x_i + B_1^{(i)}$ and $Y_i \subseteq y_i + B_2^{(i)}$ as in \cref{lemma:2dimbohr}. Then $S(X_i, Y_i, D) \approx S(X_i, Y_i, D_i)$ where $D_i \coloneqq D \cap (x_i + y_i + B_1^{(i)})$. While $X_i$ and $Y_i$ are algebraically spread, $D_i$ may not be $\ell_1$-spread, which we wanted to enforce. Thus, we design a scheme to partition $D_i$ into $\ell_1$-spread pieces, while not affecting the average measure of $X_i$ and $Y_i$ on those pieces.

\begin{lemma}
\label{lemma:dsplit}
Let $X, D \subseteq B_1$ and $Y \subseteq B_2$ where $B_1, B_2, \dots$ form a $(d,\eta)$-small sequence of Bohr sets. Assume that $X$ and $Y$ are $(r, \eta_s, \eps)$-algebraically spread. Let $\beta \in (0, 1)$ be a parameter, where $\eta \le (\eps\beta/d)^{O(1)}$. There is a distribution $\nu$ over $(x_i + B_1^{(i)}) \times (y_i + B_2^{(i)})$ satisfying the following properties:
\begin{enumerate}
\item $B_1^{(i)}, B_2^{(i)}, \dots$ form a $(d, \eta)$-exact sequence of Bohr sets.
\item $\nu(B_1^{(i)}) \ge \eta^{O(\log(1/(\eps\alpha\beta))^2)} \nu(B_1)$.
\item For all $(x_i + B_1^{(i)}) \times (y_i + B_2^{(i)}) \sim \mu$, it holds that either: 
    \begin{enumerate}
        \item $D_i \coloneqq D \cap (x_i + y_i + B_1^{(i)})$ is $(B_1^{(i)}, B_9^{(i)}, \eps)$ $\ell_1$-spread, or 
        \item $D_i/|B_1^{(i)}| \le \beta^3$.
    \end{enumerate}
\item For any $1$-bounded function $f: B_1 \times B_2 \to [-1, 1]$ it holds that
\[ \E_{x \in B_1, y \in B_2}[f(x, y)] = \E_{(x_i + B_1^{(i)}) \times (y_i + B_2^{(i)}) \sim \mu} \E_{\substack{x \sim x_i + B_1^{(i)} \\ y \sim y_i + B_2^{(i)}}}[f(x, y)] + O(\beta^3). \]
\item The probability over $(x_i + B_1^{(i)}) \times (y_i + B_2^{(i)}) \sim \mu$ that $X_i \coloneqq X \cap (x_i + B_1^{(i)})$ and $Y_i \coloneqq Y \cap (y_i + B_2^{(i)})$ are both $(r, \eta_s', 10\eps)$-algebraically spread for $\eta_s' \coloneqq \eta^{-O(\log(1/(\eps\alpha\beta))^2)} \eta_s$ is at least $1/2$.
\end{enumerate}
\end{lemma}
\begin{proof}
We first describe how to do one round of partitioning to $D$, which we then recurse heavily on.

\textbf{One round partitioning:} Let $\gamma = \beta^6$. We define a sequence of distributions $\mu^{(0)}, \mu^{(1)}$, $\dots$ over shifted regular Bohr sets $x + B' \subseteq B_1$. $\mu^{(0)}$ is defined to have all mass on $B_1$. Let $t$ be a time step, and for $x + B' \sim \mu^{(t)}$ say that $x+B'$ is \emph{bad} if $\frac{|D \cap (x+B')|}{|B'|} \ge \gamma$ and $D \cap (x+B')$ is not $(B', B'', \eps)$ $\ell_1$-spread where $B''$ has the same frequencies as $B'$ and $\nu(B'')/\nu(B') \in [\eta^8/500, \eta^8]$. If the probability over $x+B' \sim \mu$ that $x+B'$ is bad is at least $1/2$, define $\mu^{(t+1)}$ as follows. For the not bad $x+B' \sim \mu$, keep the same mass in $\mu^{(t)}$. Otherwise, replace the mass of $x+B'$ with $x+x'+B''$ where $x' \sim B'$ is uniform.

We prove that this process terminates within $T = O(\eps^{-2}\log(1/\gamma))$ steps. To prove this, consider the potential function 
\[ \Phi^{(t)} \coloneqq \E_{x + B' \sim \mu^{(t)}}\left[-\log\left(\gamma + \frac{|D \cap (x+B')|}{|B'|}\right) \right]. \]
Note that $\Phi^{(0)} \ge -\log(\gamma+1) \ge -\gamma$ and $\Phi^{(t)} \le \log(1/\gamma)$ always. We will use the following inequality in our analysis: for any $\delta_0, \delta_1 > 0$ it holds that
\begin{equation}
-\log \delta_1 = -\log \delta_0 - \log\left(1 + \left(\frac{\delta_1}{\delta_0}-1\right) \right) \ge -\log \delta_0 + 1 - \frac{\delta_1}{\delta_0} + \Omega\left(\min\left\{1, \left(\frac{\delta_1}{\delta_0} - 1\right)^2 \right\} \right). \label{eq:logineq}
\end{equation}
We now lower bound $\Phi^{(t+1)}$. Consider a bad $x+B'$ and let $\delta \coloneqq \frac{|(x+B') \cap D|}{|B'|}$. For $x' \sim B'$ define $\delta_{x'} \coloneqq \frac{|(x+x'+B'') \cap D|}{|B''|}$ -- we know that $\E_{x'}[\delta_{x'}] = \delta \pm O(\eta d)$. Additionally, because $D \cap (x+B')$ was not $(B', B'', \eps)$ $\ell_1$-spread we know that
\begin{equation} \E_{x'}\left[\mbm{1}_{\delta_{x'} < \delta}(\delta - \delta_{x'}) \right] \ge \eps
\delta/2 - O(\eta d). \label{eq:notspread} \end{equation}
Combining this with \eqref{eq:logineq} gives that
\begin{align*}
\E_{x'}[-\log(\gamma + \delta_{x'})] + \log(\gamma+\delta) &\ge \E_{x'}\left[1 -\frac{\gamma+\delta_{x'}}{\gamma+\delta} + \Omega\left(\min\left\{1, \frac{(\delta-\delta_x)^2}{(\gamma+\delta)^2} \right\} \right) \right] \\
&\ge -O(\eta d \cdot \gamma^{-1}) + \Omega(\eps^2) \ge \Omega(\eps^2),
\end{align*}
where we have used that $\gamma + \delta \le 2\delta$ and \eqref{eq:notspread}. Thus by definition $\Phi^{(t+1)} \ge \Phi^{(t)} + \Omega(\eps^2)$ (at least $1/2$ of $x+B' \sim \mu^{(t)}$ was bad), so the total number of iterations is bounded by $T = O(\eps^{-2}\log(1/\gamma))$ as desired.

\textbf{Recursive partitioning:} Given $\mu^{(T)}$ where the one round partitioning scheme terminated, consider each bad $x+B' \sim \mu^{(T)}$. For each of these, apply the one round partitioning scheme to it recursively. Do this until the recursion depth is $O(\log(1/\gamma))$.

\textbf{Constructing $\mu$:} We now describe the distribution $\mu$. Let $\mu'$ be the distribution over $x+B'$ after doing the recursive partitioning described in the above paragraph. For $x+B' \sim \mu'$ put the following mass in $\mu$: pick $y \sim B_2$ uniformly, and put $(x-y+B') \times (y+B'')$ in $\mu$, where $B''$ is chosen to be regular, with the same frequencies as $B'$, and $\nu(B'')/\nu(B') \in [\eta/2, \eta]$.

\textbf{Verifying conclusions:} We prove that the construction described satisfies all the desired conclusions. (1) follows by construction. (2) follows because the recursion depth of the recursive partitioning is $O(\log(1/\beta))$ and the number of iterations in the one round partition is $O(\eps^{-2}\log(1/\beta))$, each of which drops the radius by $\eta^{O(1)}$. (3) follows by construction, after throwing out the bad pieces (of which the total mass is at most $\beta^{100}$, because of the recursion).

We now verify (4). It suffices to verify (4) under a single partition of $B_1$ in the one-round partitioning scheme, as we can then apply induction. This amounts to verifying that
\begin{align*}
    \E_{\substack{x \sim B_1 \\ y \sim B_2}} \E_{\substack{x' \in x-y+B_9 \\ y' \in y+B_{10}}} f(x', y') = 
    \E_{\substack{x \sim B_1 \\ y \sim B_2 \\ x' \sim B_9 \\ y' \sim B_{10}}} f(x-y+x', y+y') = \E_{\substack{x \sim B_1 \\ y \sim B_2}} f(x, y) + O(\eta d),
\end{align*}
as desired. Because there are $O(\log(1/\beta)^2)$ layers of recursion and the total size of bad pieces is at most $\beta^4$ by the recursion, the result follows.

(5) follows because the expectation density of $X_i$ is $\delta_X + O(\eta d)$ by (4) applied to $f = \mbm{1}_X$. Thus the spreadness of $X$ and the reverse Markov inequality (\cref{fct:rev-mark}) imply that the probability over $\mu$ that $|X_i|/|B_1^{(i)}| \ge (1-8\eps)\delta_X$ is at least $3/4$. The same reasoning applies to $Y$, and thus (5) follows by a union bound.
\end{proof}

We now have the necessary tools to start with $A \subseteq S(X, Y, D)$ and pass to a subset $A' \subseteq S(X', Y', D')$ where $X', Y'$ are algebraically spread and $D'$ is $\ell_1$-spread. The benefit of this is that for example, the size of the container $S(X', Y', D')$ is close to what is expected.

\begin{theorem}
\label{thm:bohrpseudo}
Let $X, D \subseteq B_1$ and $Y \subseteq B_2$ where $B_1, B_2, \dots$ form a $(d,\eta)$-small sequence of Bohr sets. Define $\delta_Y = |Y|/|B_2|$, $\delta_D = |D|/|B_1|$, and $\tilde{\delta}_X \coloneqq \|X\|_{(B_1, B_2, B_3, K, K)}$. Let $A \subseteq S(X, Y, D)$ with $|A| \ge \alpha\tilde{\delta}_X\delta_Y\delta_D|B_1||B_2|$. Then there are shifted Bohr sets $x + B_1' \subseteq B_1$ and $y + B_2' \subseteq B_2$ where $B_1', B_2', \dots$ form a $(d',\eta)$-exact sequence of Bohr sets, and $X', Y', D'$ satisfying:
\begin{enumerate}
    \item $X' \subseteq X \cap (x+B_1')$, $Y' \subseteq Y \cap (y+B_2')$, and $D' \subseteq D \cap (x+y+B_1')$,
    \item $d' \le \rank(B_1) + O(rg^2\log(1/(\eps\alpha\delta_D))^2)$, and $\nu(B_1') \ge \eta_s^{O(g^2\log(1/(\eps\alpha\delta_D))^2)}\nu(B_1)$, for
    \[ g = O(\log(1/(\delta_X\delta_Y)) + \log(1/(\eps\alpha\delta_D))^3), \]
    \item $X'$ and $Y'$ are $(r, \eta_s, \eps_s)$-algebraically spread,
    \item $D'$ is $(B_1', B_9', \eps_s)$ $\ell_1$-spread,
    \item $|D'|/|B_1'| \ge \Omega(\eps\alpha\delta_D)$, and $\frac{|X'||Y'|}{|B_1'||B_2'|} \ge 2^{-O(\log(1/(\eps_s\alpha\delta_D))^3)}\tilde{\delta}_X\delta_Y$, and
    \item $A' \coloneqq A \cap S(X', Y', D')$ satisfies $|A'| \ge (1-\eps)\alpha|S(X', Y', D')|$.
\end{enumerate}
\end{theorem}
\begin{proof}
We define a sequence of distributions $\mu^{(0)}, \dots, \mu^{(T)}$ over triples $(X_i, Y_i, D_i)$ such that there are $B_1^{(i)}, B_2^{(i)}$ and $x_i, y_i$ such that $X_i \subseteq x_i + B_1^{(i)}$, $Y_i \subseteq y_i + B_2^{(i)}$, and $D_i \subseteq x_i+y_i+B_1^{(i)}$. Define $\mu^{(0)}$ to be identically $(X, Y, D)$.

At a time step $t$, first define $\mu^{(t)'}$ to be the result of replacing each $(X', Y', D') \sim \mu^{(t)}$ which does not satisfy (3) in \cref{thm:bohrpseudo} with the distribution given by applying \cref{lemma:2dimbohr} with the choice $\beta = (\eps\alpha\delta_D)^{O(1)}$, where $\delta_D = |D|/|B_1|$ (the original density of $D$), but throwing out any subpiece where the densities of $X''$ or $Y''$ has decreased from that of $X'$ or $Y'$ by more than a $\beta^{O(1)}$ factor (recall that it cannot increase significantly because $X', Y'$ are algebraically spread). Formally, \cref{lemma:2dimbohr} gives a distribution over subsets $X'' \times Y'' \subseteq X' \times Y'$ where $X'' \subseteq x'' + B_1''$ and $Y'' \subseteq y'' + B_2''$ for some regular Bohr sets $B_1'', B_2''$. In this case, put the triple $(X'', Y'', D'')$ in the distribution where $D'' \coloneqq D' \cap (x''+y''+B_1'')$.

Now define $\mu^{(t+1)}$ by replacing each $(X', Y', D') \sim \mu^{(t)'}$ which does not satisfy (4) in \cref{thm:bohrpseudo} with the distribution given by \cref{lemma:dsplit}.

Consider $\mu^{(T)}$ for $T = O(\log(1/
(\eps_s\alpha\delta_D)))$. We want to extract a piece $(X', Y', D')$ where $X', Y', D'$ are all not too small, and $A \cap S(X', Y', D')$ is relatively large. This motivates the calculation
\begin{align*}
    & \E_{(X', Y', D') \sim \mu^{(T)}}\left[\frac{1}{|X'||Y'|}\left(|A \cap S(X', Y', D')| - (1-\eps)\alpha|S(X', Y', D')| - \frac{\eps}{2}\alpha\delta_D|X'||Y'|\right) \right] \\
    & \qquad \qquad \qquad \qquad \ge \frac{1}{|X||Y|}|A| - \frac{1}{|X||Y|}(1-\eps)\alpha|S(X, Y, D)| - \frac{\eps}{2}\alpha\delta_D - O(\beta T) - 2^{-T} \\
    & \qquad \qquad \qquad \qquad \ge \frac{\eps}{3}\alpha\delta_D.
\end{align*}
Here we have repeatedly used item (4) of \cref{lemma:2dimbohr} and \cref{lemma:dsplit}. Also, we threw out at most $\beta^{O(1)}$ extra mass from small pieces when going from $\mu^{(t)'}$ to $\mu^{(t+1)}$. Finally, the $2^{-T}$ is because the number of iterations is $T = O(\log(1/(\eps\alpha\delta_D)))$, and the amount of mass in level $T$ that is still \emph{bad} is at most $2^{-T}$ by item (5) of \cref{lemma:dsplit}.
Thus there is a piece $(X', Y', D')$ where
\[ |A \cap S(X', Y', D')| \ge (1-\eps)\alpha|S(X', Y', D')| + \frac{\eps}{3}\alpha\delta_D|X'||Y'|. \]
Let us verify the hypotheses for this choice of $X', Y', D'$. Indeed, (1), (2), (3), (4), (6) follow by construction. For (5), the bound on $\frac{|X'||Y'|}{|B_1'||B_2'|}$ follows by \cref{lemma:2dimbohr} item (5). For $D'$, note that
\[ \frac{\eps}{3}\alpha\delta_D|X'||Y'| \le |S(X', Y', D')|| \le 2\delta_{D'}|X'||Y'| \]
where we have used \cref{lem:upper-bound} and the spreadness of $X'$ and $Y'$.
\end{proof}

\section{Density Increment by Reduction to the Grid Norm}
\label{sec:dsbohr}
\subsection{Density increment}
We now carry out the requisite density increment argument over general abelian groups. The analysis here corresponds in a certain sense to combining \cref{lem:VNL}, \cref{lemma:llarge}, and \cref{lemma:densincr}; however the analysis is rather more technical in the case of general abelian groups.

The first technical issue is that the parametrization used for corners over finite fields 
\[(x',y'),~(x',-x'+z'),\text{ and }(-y'+z',y')\]
directly is unusable over general group. This is due to the fact that we are assuming our set of corners lives $B_1\times B_2$ and thus any naive parameterization fails to gives that each coordinate varies over $B_1\times B_2$ uniformly. This is fixed via parameterizing corners in the form 
\[(x+x',y+y'),~(x+x',y-x'+z'),\text{ and }(x-y' + z',y+y').\]
In this parametrization we will always have that $x\sim B_1$, $y\sim B_2$, $x'\sim B_3$, $y'\sim B_4$, and $z' \sim B_5$. 

In our analysis $B_1, B_2, \dots$ will be a $(d,\eta)$-exact sequence of Bohr sets. We will have $3$ majorants, corresponding to rows, columns, and diagonals which satisfy 
\[\E_{x\sim B_1}[\mbm{1}_X(x)] = \delta_X,~\E_{y\sim B_2}[\mbm{1}_Y(y)] = \delta_Y,\text{ and }\E_{z\sim B_1}[\mbm{1}_D(z)] = \delta_D.\]
Furthermore we will assume that the indicator of our set $f:G\times G\to \{0,1\}$ satisfies 
\begin{align*}
\E_{\substack{x\sim B_1\\y\sim B_2}}f(x,y) &= \alpha \cdot \delta_X \delta_Y\delta_D.
\end{align*}
We additionally assume that our set possesses few corners; precisely we will assume that 
\begin{align*}
\sup_{\substack{x\in B_1\\y\in B_2}}\E_{\substack{x'\sim B_3\\y'\sim B_4\\z'\sim B_5}}f(x+x',y+y')f(x+x',y-x'+z')f(x-y'+z',y+y') &\le  2^{-3}\cdot \alpha^3 \cdot \delta_X^2 \delta_Y^2\delta_D^2.
\end{align*}

We fix a set of constant size parameters 
\[1/C\ll \eps_{s}\ll \eps_{R}\ll \eps_{L}\ll \eps, \]
where $\ll$ means ``much less'', so that for example $\eps_L$ is chosen to be sufficiently small in terms of $\eps$, etc. Also, by ``constant size'' we mean that these will be chosen to be absolute constants independent of $\alpha$ or $\delta_X,\delta_Y,\delta_D$. 

We define $K = C\lceil \log(4/(\delta_X\delta_Y\delta_D))\rceil$ and assume that $\eta \le e^{-O(K^3)} \cdot d^{-O(1)}.$ The crucial input pseudorandomness condition on the container functions $(X,Y,D)$ will be that 
\begin{align*}
\snorm{\mbm{1}_X}_{(B_1,B_8,B_9, K, K)}&\le (1+\eps_{s})\delta_X \\
\snorm{\mbm{1}_Y}_{(B_2,B_8,B_9, K, K)}&\le (1+\eps_{s})\delta_Y \\
\E_{x\sim B_1}|\E_{z\sim B_{10}}\mbm{1}_D(x+z) - \delta_D|&\le \eps_s\cdot \delta_D.
\end{align*}
Intuitively, this means that $X, Y$ are algebraically spread and that $D$ is $\ell_1$-spread. This asymmetry between $X,Y,$ and $D$ will persist throughout much of the analysis and cause substantial headaches. 

The first step in our analysis is noting that when passing to smaller Bohr sets various grid norm guarantees are likely to persist.
\begin{lemma}\label{lem:high-moment}
Fix $i\in \{3,4,5,6,7\}$ and for $x \in B_1$ denote 
\[\mc{E}_{X,i}(x) = \mbm{1}\Big[\snorm{\mbm{1}_X(x+\cdot)}_{(B_i,B_8,B_9, K, K)}\ge (1+2\eps_{s})\delta_X\Big].\]
Denote $\mc{E}_{Y,i}$ analogously. Then 
\[\E_{x\sim B_1}[\mc{E}_{X,i}(x)]\le e^{-\Omega(\eps_s K^2)}\]
and analogously for $\mc{E}_{Y,i}$.
\end{lemma}
\begin{proof}
We may observe that
\begin{align*}
\snorm{\mbm{1}_X}_{(B_1,B_8,B_9, K, K)}^{K^2} &=\E_{\substack{x\sim B_1\\y_1,\ldots,y_K\sim B_8\\z_1,\ldots,z_K\sim B_9}}\prod_{\substack{j\in [K]\\\ell\in [K]}}\mbm{1}_X(x + y_j + z_\ell)\\&= \E_{\substack{x\sim B_1,x'\sim B_i\\y_1,\ldots,y_K\sim B_8\\z_1,\ldots,z_K\sim B_9}}\prod_{\substack{j\in [K]\\\ell\in [K]}}\mbm{1}_X(x + x'+y_j + z_\ell) \pm O(K^2 \eta d)\\
&=\E_{x\sim B_1} \snorm{\mbm{1}_X(x + \cdot)}_{(B_i,B_8,B_9,K,K)}^{K^2}\pm O(K^2 \eta d)\\
&\ge \E_{x\sim B_1} \mc{E}_{X,i}(x) \cdot ((1+2\eps_s) \cdot \delta_X)^{K^2}-O(K^2 \eta d).
\end{align*}
Recall that we assume that $\snorm{\mbm{1}_X}_{(B_1,B_8,B_9,K,K)} \le (1+\eps_s) \delta_X$.
Rearranging and using $\eta \le e^{-O(K^3)} \cdot d^{-O(1)}$ we immediately have the desired bound. 
\end{proof}

We define $(x,y)\in B_1\times B_2$ to be \emph{well--conditioned} if 
\[x\in \bigcap_{4\le j\le 7}\on{supp}(1-\mc{E}_{X,i})\text{ and }y\in \bigcap_{4\le j\le 7}\on{supp}(1-\mc{E}_{Y,i}).\]

Next we define different types of dense rectangles, which are cheap ways to obtain a density increment. Otherwise, we are able to guarantee that on most of the space that the densities of the objects we care about such as $X, Y, D$ are relatively unchanged when passing to the smaller Bohr sets.

We say  $(x,y)$ gives a \emph{dense} rectangle of Type 1 if $(x,y)$ is well--conditioned and 
\begin{align*}
\E_{\substack{x'\sim B_3\\y'\sim B_4}}f(x+x',y+y')&\ge (1+\eps_R)\alpha \delta_X\delta_Y\cdot \E_{z\sim B_3}\mbm{1}_D(x+y+z)\\
\E_{z\sim B_3}\mbm{1}_D(x+y+z) &\ge \eps_s^2 \cdot \alpha^2 \cdot \delta_D.
\end{align*}

We say  $(x,y)$ gives a \emph{dense} rectangle of Type 2 if $(x,y)$ is well--conditioned and 
\begin{align*}
\E_{\substack{x'\sim B_3\\z'\sim B_5}}f(x+x',y-x'+z')&\ge (1+\eps_R)\alpha \delta_X\delta_Y\cdot \E_{z'\sim B_5}\mbm{1}_D(x+y+z')\\
\E_{z'\sim B_5}\mbm{1}_D(x+y+z') &\ge \eps_s^2 \cdot \alpha^2 \cdot \delta_D.
\end{align*}

We say  $(x,y)$ gives a \emph{dense} rectangle of Type 3 if $(x,y)$ is well--conditioned and 
\begin{align*}
\E_{\substack{y'\sim B_4\\z'\sim B_5}}f(x-y'+z',y+y')&\ge (1+\eps_R)\alpha \delta_X\delta_Y\cdot \E_{z'\sim B_5}\mbm{1}_D(x+y+z')\\
\E_{z'\sim B_5}\mbm{1}_D(x+y+z') &\ge \eps_s^2 \cdot \alpha^2 \cdot \delta_D.
\end{align*}

We now prove that if $(x,y)$ gives a dense rectangle this is sufficient to obtain a density increment. 
\begin{lemma}\label{lem:dense-rectangle}
Suppose that $(x,y)$ gives a dense rectangle of Type 1, 2, or 3. Then there exists $i\in\{3,4,5,6,7\}$ and $(x^{\ast},y^{\ast})$ such that the following all hold:
\begin{align*}
\E_{\substack{x'\sim B_i\\y'\sim B_{i+1}}}f(x^{\ast}+x',y^{\ast}+y')&\ge (1+\eps_R/2)\cdot \alpha \delta_X\delta_Y\cdot \E_{z'\sim B_i}\mbm{1}_D(x^{\ast}+y^{\ast} + z')\\
\E_{z'\sim B_i}\mbm{1}_D(x^{\ast}+y^{\ast} + z') &\ge \eps_s^4 \cdot \alpha^{3} \cdot \delta_D\\
\snorm{\mbm{1}_X(x^{\ast} + \cdot)}_{(B_i,B_8,B_9, K, K)}&\le (1+3\eps_{s})\delta_X\\
\snorm{\mbm{1}_Y(y^{\ast} + \cdot)}_{(B_{i+1},B_8,B_9, K, K)}&\le (1+3\eps_{s})\delta_Y.
\end{align*}
\end{lemma}
\begin{proof}
We observe that the Type 1 case is immediate by definition. The Type 2 and Type 3 cases are essentially identical; we handle the Type 2 case.

The input condition to our lemma implies that
\begin{align*}
\E_{\substack{x'\sim B_3\\z'\sim B_5}}f(x+x',y-x'+z')&\ge (1+3\eps_R/4)\alpha \delta_X\delta_Y\cdot \E_{z'\sim B_5}\mbm{1}_D(x+y+z') + \delta_X\delta_Y \cdot \eps_s^{3} \alpha^2 \delta_D.
\end{align*}
Combining the above with \cref{lem:shift-invar} and the choice of $\eta$ now gives 
\begin{align*}
\E_{\substack{x'\sim B_3\\z'\sim B_5}}&\Big(\E_{\substack{x''\sim B_6\\ y''\sim B_7}}f(x+x'+x'',y-x'+z'+y'') - (1+2\eps_R/3)\alpha \delta_X\delta_Y\cdot \mbm{1}_D(x+y+z' + x'' + y'')\Big) \\
&\qquad\qquad\qquad\ge \delta_X\delta_Y \cdot \eps_s^{3} \alpha^2 \delta_D.
\end{align*}
Let
\begin{align*}
    \mc{E}(x',z') &= \mbm{1}\Big[\snorm{\mbm{1}_X(x+x'+\cdot)}_{(B_6,B_8,B_9,K,K)}\le (1+3\eps_s)\delta_X\Big] \\
    &\qquad \cdot \mbm{1}\Big[\snorm{\mbm{1}_Y(y-x'+z'+\cdot)}_{(B_7,B_8,B_9,K,K)}\le (1+3\eps_s)\delta_X\Big].
\end{align*}
Then via the proof of \cref{lem:high-moment}, we have that 
\begin{align*}
\E_{\substack{x'\sim B_3\\z'\sim B_5}}\mc{E}(x',z')&\Big(\E_{\substack{x''\sim B_6\\ y''\sim B_7}}f(x+x'+x'',y-x'+z'+y'')- (1+3\eps_R/5)\alpha \delta_X\delta_Y\cdot \mbm{1}_D(x+y+z' + x'' + y'')\Big) \\
&\qquad\qquad\qquad\ge \delta_X\delta_Y \cdot \eps_s^{3} \alpha^2 \delta_D.
\end{align*}
Let 
\[\mc{E}_2(x',z') = \mbm{1}\Big[\E_{\substack{x''\sim B_6\\y''\sim B_7}}D(x+y+z'+x''+y'')\ge \eps_s^4 \cdot \alpha^3 \delta_D\Big].\]
Via \cref{lem:upper-bound}, we have that 
\begin{align*}
\E_{\substack{x'\sim B_3\\z'\sim B_5}}\mc{E}(x',z')\mc{E}_2(x',z')&\Big(\E_{\substack{x''\sim B_6\\ y''\sim B_7}}f(x+x'+x'',y-x'+z'+y'')\\
&\qquad\qquad -(1+3\eps_R/5)\alpha \delta_X\delta_Y\cdot \mbm{1}_D(x+y+z' + x'' + y'')\Big)>0.
\end{align*}
Taking $x'$ and $z'$ such that $\mc{E}(x',z')\mc{E}_2(x',z') = 1$ and the inner expression is strictly positive then completes the proof. 
\end{proof}

A major technical issue when handling the case of well--conditioned rectangles is that one has no matching lower bounds on the density of $X$, $Y$, and $D$. We now use the reverse Markov inequality (\cref{fct:rev-mark}) to ensure that the density is preserved with at least ``99\%'' probability.
\begin{lemma}\label{lem:density-drop}
For $i\in \{3,4,5,6,7\}$, define
\[\mc{E}_{\on{Small},X,i}(x) = \mbm{1}\Big[\E_{x'\sim B_i}\mbm{1}_X(x+x')\le (1-\eps_{s}^{1/2})\delta_X\Big].\]
Define $\mc{E}_{\on{Small},Y,i}$ analogously. Then we have that 
\[\E_{x\sim B_1}[\mc{E}_{\on{Small},X,i}(x)] = O(\eps_{s}^{1/2}) \]
and similarly for $\mc{E}_{\on{Small},Y,i}$.

Furthermore define 
\[\mc{E}_{\on{Special},D,i}(z) = \mbm{1}\Big[\E_{z'\sim B_i}\Big|\E_{z''\sim B_{10}}\mbm{1}_D(z+z'+z'')-\delta_D\Big|\ge \eps_s^{1/2}\cdot \delta_D\Big].\]
Then \[\E_{z\sim B_1}[\mc{E}_{\on{Special},D,i}(z)] = O(\eps_{s}^{1/2}). \]
\end{lemma}
\begin{proof}
We first handle the case of $X$ (with $Y$ being completely analogous) and then handle the case of $D$. 

For bounding $\mc{E}_{\on{Small},X,i}$, observe that 
\begin{align*}
\E_{x'\sim B_i}\mbm{1}_X(x+x') &\le \E_{\substack{x'\sim B_i\\ x_1\sim B_8\\x_2\sim B_9}}\mbm{1}_X(x+x' + x_1 + x_2) + O(\eta d)\le \snorm{\mbm{1}_X(x+\cdot)}_{(B_i,B_8,B_9,K,K)} + O(\eta d).
\end{align*}
By applying \cref{lem:high-moment}, we have that 
\[\E_{x\sim B_1}(1-\mc{E}_{X,i}(x))\E_{x'\sim B_i}\mbm{1}_X(x+x') \ge (1-\eps_s) \delta_X,\]
and for $x$ such that $\mc{E}_{X,i}(x) = 0$ we have that $\E_{x'\sim B_i}\mbm{1}_X(x+x') \le (1+3\eps_s)\delta_X$.

Applying \cref{fct:rev-mark} to $(1-\mc{E}_{X,i}(x))\E_{x'\sim B_i}\mbm{1}_X(x+x')$, we have that 
\[\Pr_{x\sim B_1}\Big[(1-\mc{E}_{X,i}(x))\E_{x'\sim B_i}\mbm{1}_X(x+x')\le (1-\eps_{s}^{1/2})\cdot \delta_X\Big] = O(\eps_s^{1/2}).\]
This combined with \cref{lem:high-moment} immediately gives the desired result for $\mc{E}_{\on{Small},X,i}$.

We now handle the event for $D$. The result follows via Markov's inequality and the bound that \[\E_{\substack{z\sim B_1\\z'\sim B_i}}\Big|\E_{z''\sim B_{10}}\mbm{1}_D(z+z'+z'')-\delta_D\Big|\le 2\eps_s\delta_D,\]
giving the desired result. 
\end{proof}

We now define the notion of a great pair. We say $(x,y)$ is a \emph{great pair} if $(x,y)$ is well--conditioned,
\[x\in \bigcap_{4\le j\le 7}\on{supp}(1-\mc{E}_{\on{Small},X,j}), ~y\in \bigcap_{4\le j\le 7}\on{supp}(1-\mc{E}_{\on{Small},Y,j}),\text{ and }x+y\in \bigcap_{4\le j\le 7}\on{supp}(1-\mc{E}_{\on{Special},D,j}) \]
and that
\begin{align*}
\E_{\substack{x'\sim B_3\\y'\sim B_4}}f(x+x',y+y')&\ge (1-\eps_{R}^{1/2})\alpha \delta_X\delta_Y\delta_D\\
\E_{\substack{x'\sim B_3\\z'\sim B_5}}f(x+x',y-x' + z')&\ge (1-\eps_{R}^{1/2})\alpha \delta_X\delta_Y\delta_D\\
\E_{\substack{y'\sim B_4\\z'\sim B_5}}f(x-y' + z',y+y')&\ge (1-\eps_{R}^{1/2})\alpha \delta_X\delta_Y\delta_D.
\end{align*}

The existence of a ``great pair'' will allow us to completely localize our analysis to within a certain subset of rectangles where we have appropriately lower--bounded marginals. We will derive the existence of a ``great pair'' via an averaging argument and noting that the mass in ``misbehaving'' pairs is bounded via the condition of having no dense rectangles. 

\begin{lemma}\label{lem:great-pair}
If $f$ has no dense rectangles (of Type 1, 2, or 3), then there exists a great pair $(x,y)$.
\end{lemma}
\begin{proof}
By assumption, we have that 
\[\E_{\substack{x\sim B_1\\y\sim B_2}}f(x,y) = \alpha \delta_X\delta_Y\delta_D.\]
Define 
\[\mc{E}_1(x,y) = \prod_{i\in\{3,4,5,6,7\}}(1-\mc{E}_{X,i}(x))(1-\mc{E}_{Y,i}(y)).\]
Via \cref{lem:high-moment}, we immediately have that 
\begin{align*}
\E_{\substack{x\sim B_1\\y\sim B_2}} \mc{E}_1(x,y)\E_{\substack{x'\sim B_3\\y'\sim B_4\\z'\sim B_5}}&[f(x+x',y+y') + f(x+x', y-x' + z') + f(x-y'+z',y+y')] \\
&= 3\alpha \delta_X\delta_Y\delta_D + O(\eta d) + O( e^{\Omega(-\eps_sK^2)}).
\end{align*}

Let $G:B_1\to [0,1]$ be such that $\E[G(z)]\le \tau$. Then 
\begin{align*}
\E_{z\sim B_1}G(z)\E_{z'\sim B_i}\mbm{1}_D(z+z') &=\E_{z\sim B_1}G(z)\E_{\substack{z'\sim B_i\\z''\sim B_9}}\mbm{1}_D(z+z'+z'') + O(\eta d)\\
&\le \E_{\substack{z\sim B_1\\z'\sim B_i}}G(z)\Big|\E_{z''\sim B_9}\mbm{1}_D(z+z'+z'') - \delta_D\Big|+ \delta_X \cdot \tau + O(\eta d)\\
&\le (\tau + \eps_s) \cdot \delta_D + O(\eta d).
\end{align*}

We next define 
\[\mc{E}_2(x,y) = \prod_{4\le j\le 7}(1-\mc{E}_{\on{Special},D,j}(x+y)).\]
Via combining \cref{lem:density-drop} and that there are no dense rectangles, we have that 
\begin{align*}
\E_{\substack{x\sim B_1\\y\sim B_2}} \mc{E}_1(x,y)\mc{E}_2(x,y)\E_{\substack{x'\sim B_3\\y'\sim B_4\\z'\sim B_5}}&[f(x+x',y+y') + f(x+x', y-x' + z') + f(x-y'+z',y+y')] \\
&\ge 3(1-\eps_R^{3/4})\alpha \delta_X\delta_Y\delta_D.
\end{align*}
We are using implicitly here that if the density of $D$ is especially small on the relevant shift of a Bohr set (e.g. $\le \eps_s^3 \cdot \alpha^2 \cdot \delta_D$) then \cref{lem:upper-bound} removed the contribution. 

Finally we may define 
\[\mc{E}_3(x,y) = \prod_{4\le j\le 7}(1-\mc{E}_{\on{Small},X,j}(x)) \cdot \prod_{4\le j\le 7}(1-\mc{E}_{\on{Small},Y,j}(y)).\]
Here by the assumption of no dense rectangle and \cref{lem:density-drop}, we immediately have that 
\begin{align*}
\E_{\substack{x\sim B_1\\y\sim B_2}} \mc{E}_1(x,y)\mc{E}_2(x,y)\mc{E}_3(x,y)\E_{\substack{x'\sim B_3\\y'\sim B_4\\z'\sim B_5}}&[f(x+x',y+y') + f(x+x', y-x' + z') + f(x-y'+z',y+y')] \\
&\ge 3(1-2\eps_R^{3/4})\alpha \delta_X\delta_Y\delta_D.
\end{align*}
Furthermore by the assumption of no dense rectangle, if $\mc{E}_1(x,y)\mc{E}_2(x,y)\mc{E}_3(x,y)$ hold then 
\begin{align*}
\E_{\substack{x'\sim B_3\\y'\sim B_4}}[f(x+x',y+y')] &\le (1 + 2\eps_R) \cdot \delta_X\delta_Y\delta_Z\\
\E_{\substack{x'\sim B_3\\z'\sim B_5}}[f(x+x', y-x' + z')]&\le (1 + 2\eps_R) \cdot \delta_X\delta_Y\delta_Z\\
\E_{\substack{y'\sim B_4\\z'\sim B_5}}[f(x-y'+z',y+y')]&\le (1 + 2\eps_R) \cdot \delta_X\delta_Y\delta_Z.
\end{align*}
The existence of a great pair then follows via the reverse Markov inequality (\cref{fct:rev-mark}).
\end{proof}

For the remainder of the analysis in this section we will fix a pair $(x,y)$ which is ``great''.

The next step in our analysis is analogous to the ``degree--regularization'' procedures in the finite field section where low density rows are removed (see \Cref{lemma:llarge}). We define 
\[Y_{-}(y') = \mbm{1}_Y(y+y') \mbm{1}[\E_{z'\sim B_5}f(x-y'+z',y+y')\ge (1-\eps_L) \cdot \alpha \cdot \delta_{X}\delta_D].\]
We say that $f$ is \emph{poor} in the third coordinate with respect to the direction $Y$ if 
\[\E_{y'\sim B_4}[(\mbm{1}_Y(y+y') - Y_{-}(y'))] \ge \eps_{L} \cdot \delta_{Y}.\] We now derive a density increment provided that $f$ is poor in the third coordinate with respect to the direction $Y$.
\begin{lemma}\label{lem:dens-diag-increment}
Suppose that $f$ is poor in the third coordinate with respect to the direction $Y$. Then there exists $(x^{\ast},y^{\ast})$ and a function $g:B_{7}\to \{0,1\}$ such that:
\begin{align*}
\E_{\substack{x''\sim B_6\\y''\sim B_7}}g(y'')\cdot f(x^{\ast}+x'',y^{\ast}+y'')&\ge (1+\eps_L^2/2) \cdot \alpha \delta_X\cdot \E_{y''\sim B_7}g(y'')\cdot \E_{x''\sim B_6}\mbm{1}_D(x^{\ast} + y^{\ast} + x'')\\
\E_{y''\sim B_7}[g(y'')] &\ge \eps_L^2 \cdot \alpha^2 \cdot \delta_Y\\
\E_{x''\sim B_6}\mbm{1}_D(x^{\ast} + y^{\ast} + x'') &\ge \eps_L^2 \cdot \alpha^2 \cdot \delta_D\\
g(y'')&\le \mbm{1}_Y(y^{\ast} + y'')\\
\snorm{\mbm{1}_X(x^{\ast} + \cdot)}_{(B_6,B_8,B_9, K, K)}&\le (1+3\eps_{s})\delta_X\\
\snorm{\mbm{1}_Y(y^{\ast} + \cdot)}_{(B_7,B_8,B_9, K, K)}&\le (1+3\eps_{s})\delta_Y.
\end{align*}
\end{lemma}
\begin{proof}
By assumption we have that 
\[\E_{\substack{y'\sim B_4\\z'\sim B_5}}f(x-y'+z',y+y')\ge (1-\eps_R^{1/2}) \cdot \alpha \cdot \delta_X\delta_Y\delta_D.\]
We define $g_1:B_4\to \{0,1\}$ such that $g_1(y')\le Y_{-}(y')$ and $\E_{y'\sim B_4}g_1(y') = \eps_{L} \cdot \delta_Y$. We have that 
\[\E_{\substack{y'\sim B_4\\z'\sim B_5}}f(x-y'+z',y+y')(\mbm{1}_Y(y+y') - g_1(y'))\ge (1 + 3\eps_L^2/4) \cdot \alpha \cdot \delta_X\delta_D \cdot \E_{y'\sim B_4}[\mbm{1}_Y(y+y') - g_1(y')].\] 
This implies that if $g_2(y') = \mbm{1}_Y(y+y') - g_1(y')$ then 
\[\E_{\substack{y'\sim B_4\\z'\sim B_5}}\E_{\substack{x''\sim B_6\\y''\sim B_7}}f(x-y'+z'+x'',y+y'+y'')g_2(y' + y'')\ge (1 + 5\eps_L^2/8) \cdot \alpha \cdot \delta_X\delta_D \cdot \E_{y'\sim B_4}[g_2(y')].\] 
We now set
\begin{align*}
\mc{E}_1(y',z') = \mbm{1}[\snorm{\mbm{1}_X(x-y'+z'+\cdot)}_{(B_6,B_8,B_9,K,K)}&\le (1+3\eps_{s})\delta_X \\
\vee \snorm{\mbm{1}_Y(y+y'+\cdot)}_{(B_7,B_8,B_9,K,K)}&\le (1+3\eps_{s})\delta_Y].
\end{align*}
Via the analysis in \cref{lem:dense-rectangle}, we immediately have that 
\[\E_{\substack{y'\sim B_4\\z'\sim B_5}}\mc{E}_1(y',z')\E_{\substack{x''\sim B_6\\y''\sim B_7}}f(x-y'+z'+x'',y+y'+y'')g_2(y' + y'')\ge (1 + 9\eps_L^2/16) \cdot \alpha \cdot \delta_X\delta_D \cdot \E_{y'\sim B_4}[g_2(y')].\] 

We next define 
\begin{align*}
\mc{E}_2(z') &=\mbm{1}\big[\E_{z''\sim B_6}\mbm{1}_D(x+y+z'+z'')\le \alpha^2 \cdot \eps_L^2 \cdot \delta_D\big].
\end{align*}
Observe that 
\begin{align*}
\E_{\substack{y'\sim B_4\\z'\sim B_5}}&\mc{E}_1(y',z')\mc{E}_2(z')\E_{\substack{x''\sim B_6\\y''\sim B_7}}f(x-y'+z'+x'',y+y'+y'')g_2(y' + y'')\\
&\le \E_{\substack{y'\sim B_4\\z'\sim B_5}}\mc{E}_1(y',z')\mc{E}_2(z')\E_{\substack{x''\sim B_6\\y''\sim B_7}}\mbm{1}_X(x-y'+z'+x'')\mbm{1}_Y(y+y'+y'')\mbm{1}_D(x+y+z'+x''+y'')\\
&\le 2\alpha^2 \cdot \eps_L^2 \cdot \delta_X\delta_Y\delta_D.
\end{align*}
Thus we may assume that 
\begin{align*}
\E_{\substack{y'\sim B_4\\z'\sim B_5}}&\mc{E}_1(y',z')(1-\mc{E}_2(z'))\E_{\substack{x''\sim B_6\\y''\sim B_7}}f(x-y'+z'+x'',y+y'+y'')g_2(y' + y'')\\
&\ge (1 + 17\eps_L^2/32) \cdot \alpha \cdot \delta_X\cdot \E_{z'\sim B_5}\mbm{1}_D(x+y+z') \cdot \E_{y'\sim B_4}[g_2(y')].
\end{align*}
We then define
\begin{align*}
\mc{E}_3(y') &=\mbm{1}[\E_{y''\sim B_7}g_2(y'+y'')\le \eps_L^3 \cdot \alpha^2 \cdot \delta_Y].
\end{align*}
We may observe that 
\begin{align*}
\E_{\substack{y'\sim B_4\\z'\sim B_5}} &\mc{E}_1(y',z')(1-\mc{E}_2(z'))\mc{E}_3(y')\E_{\substack{x''\sim B_6\\y''\sim B_7}}f(x-y'+z' + x'',y+y' +y'')g_2(y'+y'')\\
&\le \E_{\substack{y'\sim B_4\\z'\sim B_5}} \mc{E}_1(y',z')\mc{E}_3(y')\E_{\substack{x''\sim B_6\\y''\sim B_7}}\mbm{1}_X(x-y'+z'+x'')\mbm{1}_D(x+y+z' + x'' + y'')g_2(y'+y'').
\end{align*}
We are now in position to use \cref{lem:upper-bound} on the inner sum. Applying \cref{lem:upper-bound} and using the definition of $\mc{E}_3$ (and using $\mc{E}_1$ to guarantee spreadness), we find that 
\begin{align*}
\E_{\substack{y'\sim B_4\\z'\sim B_5}} \mc{E}_1(y',z')\mc{E}_3(y')&\E_{\substack{x''\sim B_6\\y''\sim B_7}}\mbm{1}_X(x-y'+z'+x'')\mbm{1}_D(x+y+z' + x'' + y'')g_2(y'+y'')\\
&\le 2 \alpha^2 \cdot \eps_L^3 \cdot \delta_X\delta_Y\delta_D. 
\end{align*}
Thus we have that 
\begin{align*}
\E_{\substack{y'\sim B_4\\z'\sim B_5}} &\mc{E}_1(y',z')(1-\mc{E}_2(z'))(1-\mc{E}_3(y'))\Big(\E_{\substack{x''\sim B_6\\y''\sim B_7}}f(x-y'+z' + x'',y+y' +y'')g_2(y'+y'')\\
&- (1 + \eps_L^2/2) \cdot \alpha \cdot \delta_X\cdot g_2(y' + y'')\mbm{1}_D(x+y+z' + x'' + y'')\Big)>0.
\end{align*}
Choosing $y'$ and $z'$ such that the expression in the brackets is positive while $(1-\mc{E}_1(y',z')(1-\mc{E}_2(y',z'))(1-\mc{E}_3(y'))\neq 0$ gives the result. 
\end{proof}

In a completely analogous manner we may define that $f$ is poor in the third coordinate in the $Z$ direction, that $f$ is poor in the second coordinate in the $X$ or $Z$ direction, or that $f$ is poor in the first coordinate in the $X$ or $Y$ direction. At various points in the proof we will require that $f$ is not poor in certain senses and we will derive these as the analysis proceeds and then handle the analog of \cref{lem:dens-diag-increment}.

Recall that
\[Y_{-}(y') = \mbm{1}_Y(y+y') \mbm{1}[\E_{z'\sim B_5}f(x-y'+z',y+y')\ge (1-\eps_L) \cdot \alpha \cdot \delta_{X}\delta_D]\]
and define 
\[f^{-}(y',z') = f(x-y'+z',y+y') \cdot Y_{-}(y').\]

Let $\alpha^{\ast}$ be such that 
\[\alpha^{\ast} \delta_{X} \E_{y'\sim B_4}Y_{-}(y') \cdot \E_{z'\sim B_5}\mbm{1}_D(x+y+z') = \E_{\substack{y'\sim B_4\\z'\sim B_5}}f^{-}(y',z');\]
observe by construction that $\alpha^{\ast} = \alpha (1\pm \eps_L^{1/2})$.
Finally we define 
\[h(y',z') = \alpha^{\ast} \delta_{X}Y_{-}(y')\mbm{1}_D(x+y+z').\]

We will consider the difference between the following pair of counting operators. We will consider corners of the form 
\begin{align*}
\E_{\substack{x'\sim B_3\\y'\sim B_4\\z'\sim B_5}}&(f(x+x',y+y')L^{-}(y'))(f(x+x',y-x'+z'))(f(x-y'+z',y+y')Y_{-}(y')) \\
&= \E_{\substack{x'\sim B_3\\y'\sim B_4\\z'\sim B_5}}(f(x+x',y+y')Y_{-}(y'))(f(x+x',y-x'+z'))(f^{-}(y',z'))
\end{align*}
and compare this to 
\[\E_{\substack{x'\sim B_3\\y'\sim B_4\\z'\sim B_5}}(f(x+x',y+y')Y_{-}(y'))(f(x+x',y-x'+z'))(h(y',z')).\]

Due to the assumption that our set has few corners, we immediately have that 
\begin{align*}
\E_{\substack{x'\sim B_3\\y'\sim B_4\\z'\sim B_5}}&(f(x+x',y+y')Y_{-}(y'))(f(x+x',y-x'+z'))(f(x-y'+z',y+y')Y_{-}(y'))\\
&\le \E_{\substack{x'\sim B_3\\y'\sim B_4\\z'\sim B_5}}f(x+x',y+y')f(x+x',y-x'+z')f(x-y'+z',y+y')\le 2^{-3}\cdot \alpha^{3}\delta_X^2\delta_Y^2\delta_Z^2.
\end{align*}
On the other hand, we have that 
\begin{align*}
\E_{\substack{x'\sim B_3\\y'\sim B_4\\z'\sim B_5}}&(f(x+x',y+y')Y_{-}(y'))(f(x+x',y-x'+z'))(h(y',z'))\\
& = \alpha^{\ast}\delta_X \cdot \E_{\substack{x'\sim B_3\\y'\sim B_4\\z'\sim B_5}}(f(x+x',y+y')Y_{-}(y'))(f(x+x',y-x'+z'))\\
& = \alpha^{\ast}\delta_X \cdot \E_{x'\sim B_3}(\E_{y'\sim B_4}f(x+x',y+y')Y_{-}(y'))(\E_{z'\sim B_5}f(x+x',y-x'+z')).
\end{align*}

We first observe that as $(x,y)$ is not poor in the $x$ direction in the second coordinate; for all but $O(\eps_L\delta_X|B_3|)$ many $x'\in\on{supp}(X(x+\cdot))$ we have that 
\[\E_{z'\sim B_5}f(x+x',y-x'+z')\ge (1-\eps_L)\cdot \alpha \cdot \delta_Y\delta_D.\]
In the case where $(x,y)$ is poor in the direction $x$ in the second coordinate, we may use \cref{lem:untilt-2} (with $g_2 = \mbm{1}_D(x+y+\cdot)$ and $g_1$ being the indicator of the complement of the sparse $x$ coordinates).

Thus there exists $X_{-}:B_3\to \{0,1\}$ with $X_{-}(\cdot)\le \mbm{1}_X(x+\cdot)$ and $\E_{x'\sim B_3}[X_{-}(x')]\ge (1-O(\eps_L))\delta_X$ such that our count is lower bounded by 
\[\alpha^2 \delta_X\delta_Y\delta_D \cdot (4/5) \cdot \E_{\substack{x'\sim B_3\\y'\sim B_4}}f(x+x',y+y')Y_{-}(y')X_{-}(x').\]

Now observe that if 
\[\E_{\substack{x'\sim B_3\\y'\sim B_4}}f(x+x',y+y')(\mbm{1}_Y(y+y')-Y_{-}(y'))\ge \frac{1}{10} \cdot \alpha\delta_X\delta_Y\delta_D,\]
then we have a massive density increment as $\E_{y'\sim B_4}[(\mbm{1}_Y(y+y')-Y_{-}(y'))]\le \eps_L \cdot \delta_Y$. This case immediately gives a density increment of the required type. 

Analogously we have that if
\[\E_{\substack{x'\sim B_3\\y'\sim B_4}}f(x+x',y+y')(\mbm{1}_X(x+x')-X_{-}(x'))\ge \frac{1}{10} \cdot \alpha\delta_X\delta_Y\delta_D\]
we have a massive density increment as $\E_{x'\sim B_3}[(\mbm{1}_X(x+x')-X_{-}(x'))]\le \eps_L \cdot \delta_X$. We handle processing this density increment at the end. 

Now observe that 
\begin{align*}
X_{-}(x')Y_{-}(y') &\ge \mbm{1}_X(x+x')\mbm{1}_Y(y+y') - \mbm{1}_Y(y+y')(\mbm{1}_X(x+x')-X_{-}(x')) \\
&\qquad - \mbm{1}_X(x+x')(\mbm{1}_Y(y+y')-Y_{-}(y')),
\end{align*}
and this is $(\mbm{1}_X(x+x')-X_{-}(x'))(\mbm{1}_Y(y+y')-Y_{-}(y'))\ge 0$. Thus if we do not have a massive density increment, we have that 
\[\E_{\substack{x'\sim B_3\\y'\sim B_4}}f(x+x',y+y')X_{-}(x')Y_{-}(y')\ge \alpha \cdot \delta_X\delta_Y\delta_Z \cdot 3/4;\]
observe here that $1 - 2(1/10) = 4/5>3/4$. Thus we obtain a lower bound (or a suitable density increment) of the form 
\begin{align*}
\E_{\substack{x'\sim B_3\\y'\sim B_4\\z'\sim B_5}}&(f(x+x',y+y')Y_{-}(y'))(f(x+x',y-x'+z'))(h(y',z'))\ge \frac{3}{4}\cdot \alpha^3 \delta_X^2\delta_Y^2\delta_D^2.
\end{align*}

Therefore we may assume that 
\[\Big|\E_{\substack{x'\sim B_3\\y'\sim B_4\\z'\sim B_5}}(f(x+x',y+y')Y_{-}(y'))(f(x+x',y-x'+z'))(f^{-}(y',z')-h(y',z'))\Big|\ge \frac{\alpha^3 \delta_X^2\delta_Y^2\delta_D^2}{2}.\]

We now handle the one deferred density increment.
\begin{lemma}\label{lem:one-off-density}
Suppose that there exists $g:B_3\to \{0,1\}$ with $\E_{x'\sim B_3}[g(x')] = \eps_L \cdot \delta_X$, $g(x')\le \mbm{1}_X(x+x')$, and 
\[\E_{\substack{x'\sim B_3\\y'\sim B_4}}f(x+x',y+y')g(x') \ge 2^{-4} \cdot \alpha \delta_X\delta_Y\delta_Z.\]

Then there exists $(x^{\ast},y^{\ast})$ and a function $h:B_{5}\to \{0,1\}$ such that:
\begin{align*}
\E_{\substack{x''\sim B_5\\y''\sim B_4}}h(x'')\cdot f(x^{\ast}+x'',y^{\ast}+y'')&\ge 2 \cdot \alpha \delta_Y \cdot \E_{x''\sim B_5}h(x'')\cdot \E_{y''\sim B_4}\mbm{1}_D(x^{\ast} + y^{\ast} + y''); \\
\E_{x''\sim B_5}[h(x'')] &\ge \eps_L^2 \cdot \alpha^2 \cdot \delta_D\cdot \delta_X; \\
\E_{y''\sim B_4}\mbm{1}_D(x^{\ast} + y^{\ast} + y'') &\ge \eps_L^2 \cdot \alpha^2 \cdot \delta_D; \\
h(x'')&\le \mbm{1}_X(x^{\ast} + x''); \\
\snorm{\mbm{1}_Y(y^{\ast} + \cdot)}_{(B_4,B_8,B_9, K, K)}&\le (1+3\eps_{s})\delta_Y.
\end{align*}
\end{lemma}
\begin{proof}
Observe that the input relation may be written as  
\[\E_{\substack{x'\sim B_3\\y'\sim B_4\\x''\sim B_5}}f(x+x'+x'',y+y')g(x'+x'') \ge 2^{-4} \cdot \alpha \delta_X\delta_Y\delta_D - O(\eta d).\]
The crucial claim to prove is that 
\[\E_{\substack{x'\sim B_3\\y'\sim B_4\\x''\sim B_5}}g(x'+x'')\mbm{1}_D(x+y+x'+x''+y') = (1\pm \eps_s^{\Omega(1)})\cdot \E_{x'\sim B_3}[g(x')] \cdot \delta_D.\]
To prove this, it suffices to note that 
\begin{align*}
\Big|\E_{\substack{x'\sim B_3\\y'\sim B_4\\x''\sim B_5}}g(x'+x'')&(\mbm{1}_D(x+y+x'+x''+y')-\delta_D)\Big| = \Big|\E_{\substack{x'\sim B_3\\y'\sim B_4\\x''\sim B_5}}g(x'+x'')(\mbm{1}_D(x+y+x'+y')-\delta_D)\Big|\\
&= \Big|\E_{\substack{x'\sim B_3\\y'\sim B_4\\x''\sim B_5\\y''\sim B_6}}g(x'+x'')(\mbm{1}_D(x+y+x'+y'+y'')-\delta_D)\Big|\\
&\le \E_{\substack{x'\sim B_3\\y'\sim B_4}}\big|\E_{x''\sim B_5}\mbm{1}_X(x'+x'')\big| \cdot \big|\E_{y''\sim B_6}\mbm{1}_D(x+y+x'+y'+y'')-\delta_D\big|\\
&\le 2\delta_X \cdot \E_{\substack{x'\sim B_3}} \cdot \big|\E_{y''\sim B_6}\mbm{1}_D(x+y+x'+y'')-\delta_D\big|+\eps_s \delta_X\delta_D \le 4\eps_s \cdot \delta_X\delta_D.
\end{align*}

Thus we may write 
\[\E_{\substack{x'\sim B_3\\y'\sim B_4\\x''\sim B_5}}f(x+x'+x'',y+y')g(x'+x'') \ge 3\alpha \cdot \delta_Y \cdot \E_{\substack{x'\sim B_3\\y'\sim B_4\\x''\sim B_5}}g(x'+x'')\mbm{1}_D(x+y+x'+x''+y')+ 2^{-5} \cdot \alpha \delta_X\delta_Y\delta_D.\]

We define 
\[\mc{E}_1(x') = \mbm{1}\big[\E_{x''\sim B_5}g(x'+x'') \le 2^{-7}\cdot \alpha^2 \cdot \delta_D \cdot \delta_X\big].\]
Then 
\begin{align*}
\E_{\substack{x'\sim B_3}}&\mc{E}_1(x') \E_{\substack{y'\sim B_4\\x''\sim B_5}}f(x+x'+x'',y+y')g(x'+x'')\\
&\le \E_{\substack{x'\sim B_3}}\mc{E}_1(x') \E_{\substack{y'\sim B_4\\x''\sim B_5}}\mbm{1}_Y(y+y')g(x'+x'')\le 2^{-6} \cdot \alpha^2 \cdot \delta_D\delta_X\delta_Y.
\end{align*}

Thus 
\begin{align*}
\E_{\substack{x'\sim B_3\\y'\sim B_4}}&(1-\mc{E}_1(x')) \E_{\substack{y'\sim B_4\\x''\sim B_5}}f(x+x'+x'',y+y')g(x'+x'') \\
&\ge 3\alpha \cdot \delta_Y \cdot \E_{\substack{x'\sim B_3\\y'\sim B_4\\x''\sim B_5}}g(x'+x'')\mbm{1}_D(x+y+x'+x''+y')+ 2^{-6} \cdot \alpha \delta_X\delta_Y\delta_D.
\end{align*}
We define 
\[\mc{E}_2(x') = \mbm{1}\big[\snorm{\mbm{1}_X(x'+\cdot)}_{(B_5,B_8,B_9,K,K)}\le (1+3\eps_s)\cdot \delta_X\big]\]
and we have that 
\begin{align*}
\E_{\substack{x'\sim B_3}}&(1-\mc{E}_1(x')) \mc{E}_2(x')\E_{\substack{y'\sim B_4\\x''\sim B_5}}f(x+x'+x'',y+y')g(x'+x'') \\
&\ge 3\alpha \cdot \delta_Y \cdot \E_{\substack{x'\sim B_3\\y'\sim B_4\\x''\sim B_5}}g(x'+x'')\mbm{1}_D(x+y+x'+x''+y')+ 2^{-7} \cdot \alpha \delta_X\delta_Y\delta_D.
\end{align*}
We finally define 
\[\mc{E}_3(x') = \mbm{1}\big[\E_{\substack{y'\sim B_4\\x''\sim B_5}}\mbm{1}_D(x+y+x'+y'+x'') \le 2^{-9}\cdot \alpha^2 \cdot \delta_D \big].\]
Then 
\begin{align*}
\E_{\substack{x'\sim B_3}}&\mc{E}_2(x')\mc{E}_3(x')\E_{\substack{y'\sim B_4\\x''\sim B_5}}f(x+x'+x'',y+y')g(x'+x'')\\
&\le \E_{\substack{x'\sim B_3}}\mc{E}_2(x')\mc{E}_3(x')\E_{\substack{y'\sim B_4\\x''\sim B_5}}\mbm{1}_D(x+x'+x''+y+y')\mbm{1}_X(x'+x'')\mbm{1}_Y(y+y')\le 2^{-8}\cdot \alpha^2 \cdot \delta_X\delta_Y\delta_D.
\end{align*}

Thus 
\begin{align*}
\E_{\substack{x'\sim B_3}}&(1-\mc{E}_1(x')) \mc{E}_2(x')(1-\mc{E}_3(x'))\Big(\E_{\substack{y'\sim B_4\\x''\sim B_5}}f(x+x'+x'',y+y')g(x'+x'')\\
&\qquad\qquad\qquad- 3\alpha \cdot \delta_Y \cdot g(x'+x'')\mbm{1}_D(x+y+x'+x''+y')\Big)>0
\end{align*}
and choosing $x'$ such that $(1-\mc{E}_1(x')) \mc{E}_2(x')(1-\mc{E}_3(x')) = 0$ and the internal bracket is positive gives the result.
\end{proof}

We now let $k_{\ast} = 2^{20}\lceil \log(1/(\alpha\delta_D))\rceil$. Observe that by H\"{o}lder's inequality, we have
\begin{align*}
\Big|\E_{\substack{x'\sim B_3\\y'\sim B_4\\z'\sim B_5}}&(f(x+x',y+y')Y_{-}(y'))(f(x+x',y-x'+z'))(f^{-}(y',z')-h(y',z'))\Big|^{k_{\ast}}\\
&\le (\E_{\substack{x'\sim B_3\\y'\sim B_4}}f(x+x',y+y')Y_{-}(y'))^{k_{\ast} - 1} \\
&\qquad\qquad\cdot \E_{\substack{x'\sim B_3\\y'\sim B_4}}f(x+x',y+y')Y_{-}(y')(\E_{z'\sim B_5}(f(x+x',y-x'+z'))(f^{-}(y',z')-h(y',z')))^{k_{\ast}}\\
&\le (2\alpha \delta_X\delta_Y\delta_D)^{k_{\ast} - 1} \cdot \E_{\substack{x'\sim B_3\\y'\sim B_4}}(\E_{z'\sim B_5}(f(x+x',y-x'+z'))(f^{-}(y',z')-h(y',z')))^{k_{\ast}}.
\end{align*}
Via rearranging, this implies that 
\[4^{-k_{\ast}}\alpha^{2k_{\ast}+1}(\delta_X\delta_Y\delta_D)^{k_{\ast} + 1}\le \E_{\substack{x'\sim B_3\\y'\sim B_4}}(\E_{z'\sim B_5}(f(x+x',y-x'+z'))(f^{-}(y',z')-h(y',z')))^{k_{\ast}}.\]
Due to the choice of $k_{\ast}$, this in fact implies that 
\[8^{-k_{\ast}}\alpha^{2k_{\ast}}(\delta_X\delta_Y)^{k_{\ast} + 1}\delta_D^{k_{\ast}}\le \E_{\substack{x'\sim B_3\\y'\sim B_4}}(\E_{z'\sim B_5}(f(x+x',y-x'+z'))(f^{-}(y',z')-h(y',z')))^{k_{\ast}}.\]
We remark here that this final step is crucial; we will be able to absorb the loss of a few factors of $\delta_D$ but will not be able to absorb corresponding losses in $\delta_X$ and $\delta_Y$. 

We next observe that 
\begin{align*}
\Big(\E_{\substack{x'\sim B_3\\y'\sim B_4}}&(\E_{z'\sim B_5}(f(x+x',y-x'+z'))(f^{-}(y',z')-h(y',z')))^{k_{\ast}}\Big)^2\\
&= \E_{z_1',\ldots,z_{k_{\ast}}'\sim B_5}\Big(\E_{x'\sim B_3}\prod_{j=1}^{k_{\ast}}f(x+x',y-x'+z_j')\Big)\cdot \Big(\E_{y'\sim B_4}\prod_{j=1}^{k_{\ast}}(f^{-}(y',z_j')-h(y',z_j'))\Big)\\
&\le \Big(\E_{z_1',\ldots,z_{k_{\ast}}'\sim B_5}\E_{x_1',x_2'\sim B_3}\prod_{j=1}^{k_{\ast}}f(x+x_1',y-x_1'+z_j')f(x+x_2',y-x_2'+z_j')\Big)\\
&\qquad\qquad \cdot \Big(\E_{z_1',\ldots,z_{k_{\ast}}'\sim B_5}\E_{y_1',y_2'\sim B_4}\prod_{j=1}^{k_{\ast}}(f^{-}(y_1',z_j')-h(y_1',z_j'))(f^{-}(y_2',z_j')-h(y_2',z_j'))\Big).
\end{align*}

Therefore we either have that at least one of the following holds:
\begin{align*}
2^{k_{\ast}}\alpha^{2k_{\ast}}\delta_X^{2k_{\ast}}\delta_{Y}^{2}\delta_D^{k_{\ast}}&\le\Big(\E_{z_1',\ldots,z_{k_{\ast}}'\sim B_5}\E_{x_1',x_2'\sim B_3}\prod_{j=1}^{k_{\ast}}f(x+x_1',y-x_1'+z_j')f(x+x_2',y-x_2'+z_j')\Big)\\
128^{-k_{\ast}}\alpha^{2k_{\ast}}\delta_X^2\delta_{Y}^{2k_{\ast}}\delta_D^{k_{\ast}}&\le\Big(\E_{z_1',\ldots,z_{k_{\ast}}'\sim B_5}\E_{y_1',y_2'\sim B_4}\prod_{j=1}^{k_{\ast}}(f^{-}(y_1',z_j')-h(y_1',z_j'))(f^{-}(y_2',z_j')-h(y_2',z_j'))\Big).
\end{align*}

We will proceed under the assumption that the latter case holds; the former is strictly simpler. 

Observe that $f^{-}(y,z) - h(y,z)$ is supported only on $\on{supp}(Y_{-}(\cdot)) \times \on{supp}(\mbm{1}_D(x+y+\cdot))$. Furthermore for $y\in \on{supp}(Y_{-}(\cdot))$, we have that 
\[\E_{z}f^{-}(y,z)\ge (1-\eps_L)\cdot \alpha \cdot \delta_X \delta_D.\]

Therefore we exactly have the required setup to apply \cref{lem:spectral-pos}. Letting $p = 2^{1000}\cdot k_{\ast}$, we have that 
\[(1+ 1/128)^{2p}\alpha^{2p}\delta_X^2\delta_{Y}^{2p}\delta_D^{p}\le\E_{y_1',y_2'\sim B_4}\E_{z_1',\ldots,z_{p}'\sim B_5}\prod_{j=1}^{p}f^{-}(y_1',z_j')f^{-}(y_2',z_j').\]

We now seek to apply \cref{thm:quasisifting2}. In order to do so, we will need to properly bound the size of the associated container coming from $(X,Y_{-},D)$. We claim that 
\[\E_{\substack{y'\sim B_4\\z'\sim B_5}}\mbm{1}_X(x-y'+z')Y_{-}(y')\mbm{1}_D(x+y+z')  = (1\pm \eps_L^{\Omega(1)})\cdot \delta_X\delta_Y\delta_D.\]

To prove this, first observe that we have that 
\[0\le \E_{\substack{y'\sim B_4\\z'\sim B_5}}\mbm{1}_X(x-y'+z')(\mbm{1}_Y(y+y') -Y_{-}(y'))\mbm{1}_D(x+y+z') \le \eps_L^{\Omega(1)} \cdot \delta_X\delta_Y\]
where the first is by definitions and the second via \cref{lem:upper-bound}.

Therefore it suffices to bound 
\[\E_{\substack{y'\sim B_4\\z'\sim B_5}}\mbm{1}_X(x-y'+z')\mbm{1}_Y(y+y')\mbm{1}_D(x+y+z') = (1+\eps_s^{\Omega(1)})\delta_X\delta_Y\delta_D.\]
This however is an immediate consequence of \cref{lem:conv-lower-bound-2}.

We next observe that one may obtain suitable upper bounds on 
\[\E_{\substack{y'\sim B_3\\z'\sim B_5}}[A(y')B(z')\mbm{1}_X(x-y'+z')]\]
via \cref{lem:upper-bound} for any functions $A:B_3\to [0,1]$ and $B:B_5\to [0,1]$. 

Therefore we may apply \cref{thm:quasisifting2}. We obtain functions $g_1(y')\le \mbm{1}_Y(y+y')$ and $g_2(z') \le \mbm{1}_D(x+y+z')$ such that:
\begin{align*}
\E_{y'\sim B_4}[g_1(y')]&\ge \delta_{Y} \cdot (\alpha/2)^{O(k_\ast^2\log(1/\alpha))} = \delta_{Y} \cdot (1/2)^{O(\log(1/(\alpha\delta_D))^2\log(1/\alpha)^2)}; \\
\E_{z'\sim B_5}[g_2(z')]&\ge \delta_D\cdot (\alpha/2)^{O(\log(1/\alpha))} = \delta_D\cdot (1/2)^{O(\log(1/\alpha)^2)}; \\
\E_{\substack{y'\sim B_4\\z'\sim B_5}}[f(x-y'+z',y+y')g_1(y')g_2(z')]&\ge (1+2\eps) \cdot \alpha \delta_X \cdot \E_{y'\sim B_4}[g_1(y')] \cdot \E_{z'\sim B_5}[g_2(z')].
\end{align*}

The above density increment is on a tilted rectangle. We now obtain the desired density increment via a further averaging argument. 
\begin{lemma}\label{lem:output-dense}
Suppose that $(x,y)$ is a good pair. Furthermore consider $g_1:B_4\to \{0,1\}$ and $g_2:B_5\to \{0,1\}$ such that 
\begin{align*}
\E_{y'\sim B_4}[g_1(y')]&\ge \delta_{Y} \cdot (1/2)^{O(\log(1/(\alpha\delta_D))^2\log(1/\alpha)^2)}; \\
\E_{z'\sim B_5}[g_2(z')]&\ge \delta_D\cdot (1/2)^{O(\log(1/\alpha)^2)}; \\
g_1(y')&\le \mbm{1}_Y(y+y'); \\
g_2(z')&\le \mbm{1}_D(x+y+z'); \\
\E_{\substack{y'\sim B_4\\z'\sim B_5}}f(x-y'+z',y+y')g_1(y')g_2(z')&\ge (1+2\eps) \cdot \alpha \delta_X \cdot \E_{y'\sim B_4}[g_1(y')] \cdot \E_{z'\sim B_5}[g_2(z')].
\end{align*} 
Then there exist $x^{\ast}$, $y^{\ast}$, and $h_1:B_{7}\to \{0,1\}$ and $h_2:B_6\to \{0,1\}$ such that:
\begin{align*}
h_1(y')&\le \mbm{1}_Y(y^{\ast}+y'); \\
h_2(z')&\le \mbm{1}_D(x^{\ast}+y^{\ast}+z'); \\
\E_{y'\sim B_7}[h_1(y')]&\ge \delta_{Y} \cdot (\alpha/2)^{O(k_\ast^2\log(1/\alpha))} = \delta_{Y} \cdot (1/2)^{O(\log(1/(\alpha\delta_D))^2\log(1/\alpha)^2)}; \\
\E_{z'\sim B_6}[h_2(z')]&\ge \delta_D\cdot (1/2)^{O(\log(1/\alpha)^2)}; \\
\E_{\substack{x'\sim B_6\\y'\sim B_7}}f(x^{\ast} + x',y^{\ast}+y')h_1(y')h_2(x' + y')&\ge (1+\eps) \cdot \alpha \delta_X \cdot \E_{y'\sim B_7}[h_1(y')]\cdot \E_{z'\sim B_6}[h_2(z')]; \\
\snorm{\mbm{1}_X(x^{\ast}+\cdot)}_{(B_6,B_8,B_9,K,K)}&\le (1+4\eps_s)\cdot \delta_X.
\end{align*}
\end{lemma}
\begin{proof}
Let $z' = z' + x'' + y''$ and $y' = y' + y''$ with $x''\sim B_6$ and $y''\sim B_7$. We then have that 
\begin{multline*}
    \E_{\substack{y'\sim B_4\\z'\sim B_5\\x''\sim B_6\\y''\sim B_7}}f(x-y'+z' + x'',y+y' + y'')g_1(y'+y'')g_2(z'+x'' + y'') \\
    \ge (1+2\eps) \cdot \alpha \delta_X \cdot \E_{y'\sim B_4}[g_1(y')] \cdot \E_{z'\sim B_5}[g_2(z')] - O(\eta d).
\end{multline*}

We define 
\[\mc{E}_1(y',z') = \mbm{1}\big[\snorm{\mbm{1}_X(x-y'+z'+\cdot)}_{(B_6,B_8,B_9,K,K)}\le (1+4\eps_s)\cdot \delta_X\big]\]
and by the proof of \cref{lem:high-moment} we have that 
\begin{align*}
\E_{\substack{y'\sim B_4\\z'\sim B_5}}\mc{E}_1(y',z')&\E_{\substack{x''\sim B_6\\y''\sim B_7}}f(x-y'+z' + x'',y+y' + y'')g_1(y'+y'')g_2(z'+x'' + y'')\\
&\ge (1+2\eps) \cdot \alpha \delta_X \cdot \E_{y'\sim B_4}[g_1(y')] \cdot \E_{z'\sim B_5}[g_2(z')] - O(\eta d).
\end{align*}

We then define 
\[\mc{E}_2(y',z') = \mbm{1}\Big[\E_{\substack{x''\sim B_6}}g_2(z'+x'') \le 2^{-3}\eps^2\alpha^2 \cdot \E_{z\sim B_5}[g_2(z)]\Big].\]
Observe that 
\begin{align*}
\E_{\substack{y'\sim B_4\\z'\sim B_5}}&\mc{E}_1(y',z')\mc{E}_2(y',z')\E_{\substack{x''\sim B_6\\y''\sim B_7}}f(x-y'+z' + x'',y+y' + y'')g_1(y'+y'')g_2(z'+x'' + y'')\\
&\le \E_{\substack{y'\sim B_4\\z'\sim B_5}}\mc{E}_1(y',z')\mc{E}_2(y',z')\E_{\substack{x''\sim B_6\\y''\sim B_7}}\mbm{1}_X(x-y'+z'+x'')g_1(y'+y'')g_2(z'+x'' + y'')\\
&\le 2\E_{\substack{y'\sim B_4\\z'\sim B_5}} \E_{\substack{x''\sim B_6\\y''\sim B_7}}\delta_X \cdot \Big(\eps^2\alpha^2/8 \cdot \E_{z\sim B_5}[g_2(z)]\Big) \cdot g_1(y' + y'')\\
&\le 2^{-2}\eps^2 \cdot \alpha^2 \cdot \delta_X \cdot \E_{y'\sim B_4}[g_1(y')] \cdot \E_{z'\sim B_5}[g_2(z')]
\end{align*}
where we have used \cref{lem:upper-bound} and the definition of $\mc{E}_2(y',z')$.

We may similarly define  
\[\mc{E}_3(y',z') = \mbm{1}[\E_{\substack{y''\sim B_7}}g_1(y'+y'') \le 2^{-3}\eps^2\alpha^2 \cdot \delta_D^2 \cdot \E_{z\sim B_4}[g_1(z)]].\]
Using the same analysis, except replacing the final internal bound on the $g_2$ expectation by $g_1$, we obtain that 
\begin{align*}
\E_{\substack{y'\sim B_4\\z'\sim B_5}}&\mc{E}_1(y',z')\mc{E}_3(y',z')\E_{\substack{x''\sim B_6\\y''\sim B_7}}f(x-y'+z' + x'',y+y' + y'')g_1(y'+y'')g_2(z'+x'' + y'')\\
&\le \E_{\substack{y'\sim B_4\\z'\sim B_5}}\mc{E}_1(y',z')\mc{E}_3(y',z')\E_{\substack{x''\sim B_6\\y''\sim B_7}}\mbm{1}_X(x-y'+z' + x'')g_1(y'+y'')\\
&\le 2^{-2}\eps^2 \cdot \alpha^2 \cdot \delta_X \delta_D \cdot \cdot \E_{y'\sim B_4}[g_1(y')] \cdot \E_{z'\sim B_5}[g_2(z')].
\end{align*}

This implies that 
\begin{align*}
\E_{\substack{y'\sim B_4\\z'\sim B_5}}&\mc{E}_1(y',z')(1-\mc{E}_2(y',z'))(1-\mc{E}_3(y',z')) \\
&\cdot \Big(\E_{\substack{x''\sim B_6\\y''\sim B_7}}f(x-y'+z' + x'',y+y' + y'')g_1(y'+y'')g_2(z'+x'' + y'') \\
&\qquad - (1+9\eps/8) \cdot \alpha \cdot \delta_X \cdot \E_{\substack{x''\sim B_6\\y''\sim B_7}}[g_1(y'+y'')g_2(z'+x'')]\Big)>0.
\end{align*}
The result then follows by choosing a $(y',z')$ pair such that the LHS is strictly positive. (Observe this forces $\mc{E}_1(y',z')(1-\mc{E}_2(y',z'))(1-\mc{E}_3(y',z')) = 1$ which are precisely the conditions we require.) 
\end{proof}

We briefly end this section with a discussion of the case of 
\begin{align*}
2^{k_{\ast}}\alpha^{2k_{\ast}}\delta_X^{2k_{\ast}}\delta_{Y}^{2}\delta_D^{k_{\ast}}&\le\Big(\E_{z_1',\ldots,z_{k_{\ast}}'\sim B_5}\E_{x_1',x_2'\sim B_3}\prod_{j=1}^{k_{\ast}}f(x+x_1',y-x_1'+z_j')f(x+x_2',y-x_2'+z_j')\Big)
\end{align*}
where certain slight differences arise. Observe that in this case there is no need to invoke spectral positivity. The analysis to obtain the appropriate upper-regularity on the container is exactly as before; the critical issue is that the density increment coming from \cref{thm:quasisifting2} now has a different form. 

We obtain functions $g_1(x')\le \mbm{1}_X(x+x')$ and $g_2(z') \le \mbm{1}_D(x+y+z')$ such that:
\begin{align*}
\E_{x'\sim B_3}[g_1(x')]&\ge \delta_{X} \cdot (\alpha/2)^{O(k_\ast^2\log(1/\alpha))} = \delta_{X} \cdot (1/2)^{O(\log(1/(\alpha\delta_D))^2\log(1/\alpha)^2)}; \\
\E_{z'\sim B_5}[g_2(z')]&\ge \delta_D\cdot (\alpha/2)^{O(\log(1/\alpha))} = \delta_D\cdot (1/2)^{O(\log(1/\alpha)^2)}; \\
\E_{\substack{x'\sim B_3\\z'\sim B_5}}[f(x+x',y-x'+z')g_1(x')g_2(z')]&\ge (1+2\eps) \cdot \alpha \delta_Y \cdot \E_{x'\sim B_3}[g_1(x')] \cdot \E_{z'\sim B_5}[g_2(z')].
\end{align*}

We now derive an increment on an untilted rectangle. A key differing feature is that $x$ will now range in a smaller Bohr set than $y$. 
\begin{lemma}\label{lem:untilt-2}
Suppose that $(x,y)$ is a good pair. Furthermore consider $g_1:B_3\to \{0,1\}$ and $g_2:B_5\to \{0,1\}$ such that 
\begin{align*}
\E_{x'\sim B_3}[g_1(x')]&\ge \delta_{X} \cdot (1/2)^{O(\log(1/(\alpha\delta_D))^2\log(1/\alpha)^2)}; \\
\E_{z'\sim B_5}[g_2(z')]&\ge\delta_D\cdot (1/2)^{O(\log(1/\alpha)^2)}; \\
g_1(x')&\le \mbm{1}_X(x+x'); \\
g_2(z')&\le \mbm{1}_D(x+y+z'); \\
\E_{\substack{x'\sim B_3\\z'\sim B_5}}[f(x+x',y-x'+z')g_1(x')g_2(z')]&\ge (1+2\eps) \cdot \alpha \delta_Y \cdot \E_{x'\sim B_3}[g_1(x')] \cdot \E_{z'\sim B_5}[g_2(z')].
\end{align*}
Then there exists $(x^{\ast},y^{\ast})$ and $h_1:B_6\to \{0,1\}$ and $h_2:B_5\to \{0,1\}$ such that:
\begin{align*}
\E_{x'\sim B_6}[h_1(x')]&\ge  \delta_{X} \cdot (1/2)^{O(\log(1/(\alpha\delta_D))^2\log(1/\alpha)^2)}; \\
\E_{z'\sim B_5}[h_2(z')]&\ge \delta_D\cdot (1/2)^{O(\log(1/\alpha)^2)}; \\
h_1(x')&\le \mbm{1}_X(x^{\ast}+x'); \\
h_2(z')&\le \mbm{1}_D(x^{\ast}+y^{\ast}+z'); \\
\E_{\substack{x'\sim B_6\\z'\sim B_5}}[f(x^{\ast} + x',y^{\ast} + y')h_1(x')h_2(y')]&\ge (1+\eps) \cdot \alpha \delta_Y \cdot \E_{x'\sim B_3}[h_1(x')] \cdot \E_{z'\sim B_5}[h_2(z')]; \\
\snorm{Y(y^{\ast}+\cdot)}_{(B_6,B_8,B_9,K,K)}&\le (1+4\eps_s) \cdot \delta_Y.
\end{align*}
\end{lemma}
\begin{proof}
Observe that the final condition implies that
\[\E_{\substack{x'\sim B_3\\z'\sim B_5\\\ell\sim B_6}}[f(x+x'+\ell,y-x'+z')g_1(x'+\ell)g_2(z'+\ell)]\ge (1+2\eps) \cdot \alpha \delta_Y \cdot \E_{x'\sim B_3}[g_1(x')] \cdot \E_{z'\sim B_5}[g_2(z')] - O(\eta d).\]
We define
\begin{align*}
\mc{E}_1(x') &= \mbm{1}\Big[\snorm{\mbm{1}_Y(y-x'+\cdot)}_{(B_5,B_8,B_9,K,K)}\le (1+3\eps_s)\delta_Y\Big]
\end{align*}
and by the proof of \cref{lem:high-moment} we have that 
\begin{align*}
&\E_{\substack{x'\sim B_3}}\mc{E}_1(x')\E_{\substack{z'\sim B_5\\\ell\sim B_6}}[f(x+x'+\ell,y-x'+z')g_1(x'+\ell)g_2(z'+\ell)] \\
~& \qquad \qquad \qquad \ge (1+2\eps) \cdot \alpha \delta_Y \cdot \E_{x'\sim B_3}[g_1(x')] \cdot \E_{z'\sim B_5}[g_2(z')] - O(\eta d).
\end{align*}
We now define 
\[\mc{E}_2(x') = \mbm{1}[\E_{\ell\sim B_6}[g_1(x' + \ell)]\le \eps^2 \cdot 2^{-O(\log(1/\alpha)^3)} \cdot \E_{x''\sim B_3}[g_1(x'')]].\]
We now have that 
\begin{align*}
\E_{\substack{x'\sim B_3}}&\mc{E}_1(x')\mc{E}_2(x')\E_{\substack{z'\sim B_5\\\ell\sim B_6}}[f(x+x'+\ell,y-x'+z')g_1(x'+\ell)g_2(z'+\ell)]\\
&\le \E_{\substack{x'\sim B_3}}\mc{E}_1(x')\mc{E}_2(x')\E_{\substack{z'\sim B_5\\\ell\sim B_6}}[\mbm{1}_Y(y-x'+z')g_1(x'+\ell)\mbm{1}_D(x+y+z'+\ell)]\\
&\le 2\delta_Y\delta_D \cdot \E_{\substack{x'\sim B_3}}\mc{E}_1(x')\mc{E}_2(x')\E_{\substack{z'\sim B_5\\\ell\sim B_6}}[g_1(x'+\ell)] + O(e^{-\Omega(\eps_s K)})\\
&\le 2\delta_Y\delta_D \cdot \eps^2 \cdot 2^{-O(\log(1/\alpha)^3)} \cdot \E_{x''\sim B_3}[g_1(x'')] +  O(e^{-\Omega(\eps_s K)}) + O(\eta d).
\end{align*}
Here we have applied \cref{lem:upper-bound}. Therefore we have that 
\begin{align*}
\E_{\substack{x'\sim B_3}}&\mc{E}_1(x')(1-\mc{E}_2(x'))\Big(\E_{\substack{z'\sim B_5\\\ell\sim B_6}}[f(x+x'+\ell,y-x'+z')g_1(x'+\ell)g_2(z'+\ell)]\\
& -(1+7\eps/4) \cdot \alpha \delta_X \cdot g_1(x'+\ell)g_2(z'+\ell)]\Big)>0.
\end{align*}
Thus there exists $x'$ with $\mc{E}_1(x')(1-\mc{E}_2(x')) = 1$ such that 
\[\E_{\substack{z'\sim B_5\\\ell\sim B_6}}[f(x+x'+\ell,y-x'+z')g_1(x'+\ell)g_2(z'+\ell)]\ge (1+7\eps/4) \cdot \alpha \delta_X \cdot \E_{\substack{z'\sim B_5\\\ell\sim B_6}}[g_1(x'+\ell)g_2(z'+\ell)].\]
This immediately gives the desired result. 
\end{proof}

\subsection{Completing the proof}
We now tie together the final loose ends; this material is little more than chaining various lemmas to complete the proof. We will in fact prove the stronger counting version of \cref{thm:main}, which we will need later to prove \cref{cor:3d,corollary:4nof}.
\begin{theorem}
\label{thm:countmain}
Let $G$ be a finite abelian group and $A \subseteq G \times G$ with $|A| = \alpha|G|^2$. Then
\[ \Pr_{x, y, d \in G}\left[(x, y), (x, y+d), (x+d, y) \in A \right] \ge 2^{-O(\log(1/\alpha)^{600})}. \]
\end{theorem}
\begin{proof}
Initially let $|A| = \alpha|G|^2$.
We will maintain Bohr sets $B_1, B_2$ of the same frequencies and $\nu(B_2)/\nu(B_1) \le \eta$, and $X \subseteq x + B_1$, $Y \subseteq y + B_2$, and $D \subseteq x+y+B_1$. The density of $A \cap S(X, Y, D)$ within $S(X, Y, D)$ will increase over the course of the procedure. Throughout this proof, let $\delta_X = |X|/|B_1|$, $\delta_Y = |Y|/|B_2|$, and $\delta_D = |D|/|B_1|$.

We will start describing how large various parameters are over the course of the density increment. We will prove that $\delta_D \ge 2^{-O(\log(1/\alpha)^3)}$ and $\delta_X, \delta_Y \ge 2^{-O(\log(1/\alpha)^{10})}$. Also we will choose $K = O(\log(1/\alpha)^{30})$, $\eta = 2^{-O(\log(1/\alpha)^{100})}$, and $\eta_s = \eta^C$ for sufficiently large constant $C$.
Finally we will choose $r = O(\log(1/\alpha)^{350})$.

Before a density increment, we will ensure the following pseudorandomness properties on $X, Y, D$. We will maintain that $X, Y$ are $(r, \eta_s, \eps_s)$-algebraically spread, and that $D$ is $(B_1, B_9, \eps_s)$ $\ell_1$-spread. Let us argue that this implies that $\|\mbm{1}_X\|_{(B_1, B_8, B_9, K, K)} \le (1+O(\eps_s))\delta_X$ and similarly $\|\mbm{1}_Y\|_{(B_2, B_8, B_9, K, K)} \le (1+O(\eps_s))\delta_Y$. Indeed, if the former fails, applying \cref{thm:bohrkk} and then \cref{thm:ap+KM} gives that $X$ admits a density increment onto a Bohr set with rank increase at most $O(\log(1/\delta_X^K)^8) \le O(\log(1/\alpha)^{350})$.

To apply the results in the section we must check that \cref{lem:upper-bound} provides enough combinatorial spreadness to apply \cref{thm:quasisifting2}. This holds for the choice of $K$ and $\eta$. Thus, we may apply the results in this section to conclude that either:
\[ \Pr_{x, d \in B_1, y \in B_2}\left[(x, y), (x+d, y), (x, y+d) \in A \right] \ge \alpha^3\delta_X^2\delta_Y^2\delta_D^2/10, \]
or $A$ admits a density increment of the form described in \cref{lem:output-dense} or \cref{lem:untilt-2}. It is worth pointing out that in \cref{lem:untilt-2}, that now $Y, D$ are in the larger Bohr set and $X$ is in the smaller one, but this can be handled by switching the roles of $X$ and $Y$ in the analysis each time this happens. In this density increment, the size of $\delta_D$ drops by $2^{-O(\log(1/\alpha)^2)}$ as stated, so the total drop is $2^{-O(\log(1/\alpha)^3)}$ over $O(\log(1/\alpha))$ steps. The drop of $\delta_X$ and $\delta_Y$ is $2^{-O(\log(1/(\alpha\delta_D))^2\log(1/\alpha)^2)} = 2^{-O(\log(1/\alpha)^8)}$ so a total of $2^{-O(\log(1/\alpha)^9)}$ over $O(\log(1/\alpha))$ steps. The radii of the Bohr sets $B_1, B_2$ also drop by $\eta^{O(1)}$ as stated.

After a density increment we apply \cref{thm:bohrpseudo} to pseudorandomize $X, Y, D$. Let us first discuss how this affects $\delta_D$ and $\delta_X, \delta_Y$. $\delta_D$ drops by an additional $O(\eps\alpha)$ per step, which does not affect the lower bound. $\delta_X$ and $\delta_Y$ drop by $2^{-O(\log(1/\delta_D)^3)} = 2^{-O(\log(1/\alpha)^9)}$, for a total of $2^{-O(\log(1/\alpha)^{10})}$ over $O(\log(1/\alpha))$ steps.

Now we track how the rank and radius of $B_1, B_2$ change. Note in \cref{thm:bohrpseudo} that $g = O(\log(1/\alpha)^9)$, so the rank of $B_1$ increases by $O(rg^2\log(1/\delta_D)^2) = O(\log(1/\alpha)^{400})$. The radius decreases of $\eta_s^{O(g^2\log(1/\delta_D)^2)} = 2^{-O(\log(1/\alpha)^{150})}$. Thus, the number of corners in $A$ at the end is at least
\[ \alpha^3\delta_X^2\delta_Y^2\delta_D^2|B_1||B_2| \ge \nu(B_2)^{O(\rank(B_2))}|G|^2 \ge 2^{-O(\log(1/\alpha)^{600})}|G|^2, \]
where we have applied \cref{lem:size-bounded}.
\end{proof}

\section{Coloring Bounds for 3-Dimensional Corners}
\label{sec:3dcorners}

In this section we prove \cref{cor:3d}. The proof requires the definition of a cylinder intersection, which is known to be related to the corners problem and communication complexity.
\begin{definition}[Cylinder intersection]
\label{def:cylinder}
Let $G$ be an abelian group and $S_{XY}, S_{YZ}, S_{XZ} \subseteq G \times G$. Then the cylinder intersection $\mc{I}(S_{XY}, S_{YZ}, S_{XZ})$ is defined as
\[ G \times G \times G \supseteq \mc{I}(S_{XY}, S_{YZ}, S_{XZ}) \coloneqq \{ (x,y,z) \in G \times G \times G : (x,y) \in S_{XY}, (y,z) \in S_{YZ}, (x,z) \in S_{XZ}\}. \]
\end{definition}

\cref{cor:3d} follows by inducting on the following statement. In the statement below, one should think of $*$ as representing points that are uncolored, where the number of uncolored points is at most $\widetilde{O}(|G|^2)$.
\begin{lemma}
\label{lemma:3d}
Let $A = \mc{I}(S_{XY}, S_{YZ}, S_{XZ}) \subseteq G \times G \times G$ be a cylinder intersection. Let $f: A \to [L] \cup \{*\}$ be a coloring of $A$, and let $U = |f^{-1}(*)|$. If $A$ contains no monochromatic 3D corners with colors in $[L]$, then there are subsets $A' \subseteq A$ and $S_{XY}' \subseteq S_{XY}, S_{YZ}' \subseteq S_{YZ}, S_{XZ}' \subseteq S_{XZ}$ such that:
\begin{enumerate}
    \item $A' = \mc{I}(S_{XY}', S_{YZ}', S_{XZ}')$ is a cylinder intersection.
    \item There is a color $c \in [L]$ such that $|A' \cap f^{-1}(\{c, *\})| \le U + |G|^2$. Informally, almost all elements of $A'$ are in one of $L-1$ colors.
    \item There is a universal constant $C$ such that for $\delta \coloneqq \frac{|A| - U}{L|G|^3}$, it holds that
    \[ \frac{|A'|}{|G|^3} \ge e^{-O(\log(1/\delta)^C)}. \]
\end{enumerate}
\end{lemma}
\begin{proof}
By the Pigeonhole principle, there is a color $c \in [L]$ and $g \in G$ such that the set
\[ T = \{(x,y,z) \in G \times G \times G : (x,y,z) \in A, f((x,y,z)) = c, x+y+z = g\} \]
is large, specifically that $|T| \ge \frac{|A| - U}{L|G|}$. Let $S_{XY}'$, $S_{YZ}'$, and $S_{XZ}'$ be the projections of $T$ onto the $XY$, $XZ$, and $YZ$ faces, formally
\[ S_{XY}' \coloneqq \{(x,y) \in G \times G : \exists z \in G \text{ such that } (x,y,z) \in T\}, \]
and similarly for $S_{YZ}'$ and $S_{XZ}'$. Because $T \subseteq A$ we know that $S_{XY}' \subseteq S_{XY}$, $S_{YZ}' \subseteq S_{YZ}$, $S_{XZ}' \subseteq S_{XZ}$ and thus $A' \subseteq A$.

We now check item (2). This amounts to checking that if a point in $A'$ is colored $c$, then it must lie in $T$, and then noting that $|T| \le |G|^2$. Indeed, assume that $(x,y,z) \in A'$, so that $(x,y,g-x-y)$, $(g-y-z,y,z)$, $(x,g-x-z,z)$ are in $T$. If $x+y+z \neq g$, then $(x,y,z)$ forms a 3D corner with these points (for $d = g-x-y-z$). Thus, $(x,y,z)$ cannot have color $c$.

To prove (3) we first prove the claim that the size of $A'$ is at least the number of 2D corners in $S_{XY}' \subseteq G \times G$. Indeed, assume that $S_{XY}'$ contains a 2D corner $(x,y)$, $(x,y+d)$, $(x+d,y) \in S_{XY}'$. Then $(x,y,g-x-y)$, $(x,y+d,g-x-y-d)$, $(x+d,y,g-x-y-d) \in T$ by the definition of $S_{XY}'$, so $(x,y,g-x-y-d) \in A'$. Each triple $(x,y,d)$ generates a distinct point, so the claim is proven. Because $|S_{XY}'| = |T| \ge \delta|G|^2$, (3) follows by \cref{thm:countmain} (the counting version of \cref{thm:main}).
\end{proof}

From here we can conclude the proof of \cref{cor:3d}.
\begin{proof}[Proof of \cref{cor:3d}]
Suppose by contradiction there exists a coloring of $G \times G \times G$ with $L = c \log \log \log |G|$ colors so that there are no monochromatic 3D corners.
We will iteratively apply \cref{lemma:3d} to restrict our attention to increasingly smaller subsets of $G\times G\times G$ with additional colors removed from consideration (by replacing any occurrences of them with $*$).
Our contradiction will arise from running out of colors before coloring all points in $G\times G\times G$.

Let $A_0 = G \times G \times G$. We will define $A_t$ for $t = 1, \dots, L$, and let $\delta_t = |A_t|/|G|^3$, so that $\delta_0 = 1$. We will maintain the invariant that $\delta_t \ge |G|^{-1/2}$. With this invariant and the choice of $L$, choosing $A_{t+1}$ to be the $A'$ in \cref{lemma:3d} for $A = A_t$ gives that
\begin{equation} \delta_{t+1} = \frac{|A_{t+1}|}{|G|^3} \ge e^{-O(\log(1/\delta)^C)} \ge e^{-O(\log(2L/\delta_t)^C)} \label{eq:iterate}
\end{equation}
for $\delta = \frac{|A_t| - |f^{-1}(*)|}{L|G|^3}$. Then, we have used that $\delta \ge \frac{\delta_t}{2L}$ because $\delta_t \ge |G|^{-1/2}$ and $|f^{-1}(*)| \le L|G|^2$ by induction using item (2). Iterating \eqref{eq:iterate} gives that $\log(1/\delta_L) \le L^{C^L}$, so $\delta_L \ge e^{-L^{C^L}}$. For $L = c\log\log\log|G|$ for sufficiently small $c$, we know that $\delta_L \ge |G|^{-1/2}$, thus establishing that the desired invariant holds throughout. This is a contradiction because we have no colors left.
\end{proof}
It is worth emphasizing again that we were only able to obtain ``reasonable'' bounds for the 3D corners problem in the coloring setting because of our quasipolynomial bounds for the density version of 2D corners. With the analogous version of the above argument, one can check that even inverse logarithmic bounds for density 2D corners would only yield tower type bounds for coloring 3D corners.

\appendix

\section{Almost Periodicity}\label{sec:almost-period}

We first state the key almost periodicity result which we require. This is proven as \cite[Theorem~5.4]{SS16}; we quote the statement from \cite[Theorem~8]{Mil24} (which while stated for $\mb{Z}/N\mb{Z}$ follows for general $G$ by changing each occurrence of $\mb{Z}/N\mb{Z}$ in the half page deduction from \cite[Theorem~5.4]{SS16} to $G$).

\begin{theorem}\label{thm:AP-input}
Let $\eps\in (0,1/2)$. Let $B_1, B_2, \dots$ be a $(d,\eta)$-small sequence of Bohr sets (see \cref{def:etasmall}).

Let $Y\subseteq B_1$ and $Z\subseteq B_2$ with $\beta = \E_{x\sim B_1}[\mbm{1}_{Y}(x)]$ and $\gamma = \E_{x\sim B_2}[\mbm{1}_{Z}(x)]$. Let $D\subseteq G$ be such that $|Y|\le |D|\le 2|B_1|$.

Then there exists a regular Bohr set $B'\subseteq B_2$ of dimension at most $d+d'$ where 
\[d' = O \left( \eps^{-4}\log(2\beta^{-1})^{3}\log(2\gamma^{-1}) \right)\]
and radius $r_2 \cdot (\eps \beta)/(24d^3d')$ and such that 
\[\Big|\E_{\substack{b\sim B'\\y\sim B_1\\z\sim B_2}}\mbm{1}_{D}(z-y + b) \mbm{1}_{Y}(y)\mbm{1}_{Z}(z) -\E_{\substack{y\sim B_1\\z\sim B_2}}\mbm{1}_{D}(z-y) \mbm{1}_{Y}(y)\mbm{1}_{Z}(z)\Big|\le \eps \beta \gamma.\]
\end{theorem}

We now give the proof of \cref{thm:ap+KM}.
\begin{proof}[{Proof of \cref{thm:ap+KM}}]
We first begin with essentially the dependent random choice argument of \cite{KM23}; we follow the proof as in \cite[Theorem~3.7]{FHHK24}. We let $B_3$ be a regular Bohr set with the same frequencies as $B_i$ and such that $r_3/r_2\in [\eta/2,\eta]$. Observe that 
\[\E_{\substack{x\sim B_1\\y\sim B_2\\z\sim B_3}}f_1(x-z)f_2(y + z)g(x+y)\ge (1+9\eps/10)\cdot \E[f_1] \cdot \E[f_2] \cdot \E[g].\]
Let $p = 1000\lceil k\log(1/\eps)/\eps\rceil$ and observe that 
\begin{align*}
\E_{\substack{x\sim B_1\\y\sim B_2\\z\sim B_3}}f_1(x-z)f_2(y + z)g(x+y)&\le (\E_{\substack{x\sim B_1\\z\sim B_3}}f_1(x-z))^{p-1}\cdot \Big(\E_{\substack{x\sim B_1}}\Big(\E_{y\sim B_2}f_2(y+z)g(x+y)\Big)^{p}\Big).
\end{align*}
Via the definition of $p$, this implies that 
\[(1+7\eps/8)^p \cdot (\E[f_2]\E[g])^{p}\le \E_{\substack{x\sim B_1\\y_1,\ldots,y_p\sim B_2\\z\sim B_3}}\prod_{j=1}^pf_2(y_j+z)g(x+y_j).\]
Define 
\[\mc{E}_1(y_1,\ldots,y_p) = \mbm{1}\Big[\E_{x\sim B_1}\prod_{j=1}^pg(x+y_j)\le (\E[g]\E[f_2])^p\text{ or }\E_{z\sim B_3}\prod_{j=1}^pf_2(y_j+z)\le (\E[g]\E[f_2])^p\Big]\]
and observe that 
\[(1+3\eps/4)^p \cdot (\E[f_2]\E[g])^{p}\le \E_{\substack{x\sim B_1\\y_1,\ldots,y_p\sim B_2\\z\sim B_3}}(1-\mc{E}_1(y_1,\ldots,y_p))\prod_{j=1}^pf_2(y_j+z)g(x+y_j).\]
We define 
\[\mc{E}_2(x,z) = \mbm{1}\Big[\E_{y\sim B_2}f_2(y+z)g(x+y)\le (1+5\eps/8) \cdot (\E[f_2]\cdot \E[g])\Big].\]
Observe that 
\begin{align*}
\E_{\substack{x\sim B_1\\y_1,\ldots,y_p\sim B_2\\z\sim B_3}}\mc{E}_2(x,z)&\prod_{j=1}^pf_2(y_j+z)g(x+y_j) \\ 
&\le 10^{-5}\cdot \eps \cdot \E_{\substack{x\sim B_1\\y_1,\ldots,y_p\sim B_2\\z\sim B_3}}(1-\mc{E}_1(y_1,\ldots,y_p))\prod_{j=1}^pf_2(y_j+z)g(x+y_j).\end{align*}
Then there exists $(y_1,\ldots,y_p)\in B_2^{\otimes p}$ such that $\mc{E}_1(y_1,\ldots,y_p) = 0$ and 
\begin{align*}
\E_{\substack{x\sim B_1\\z\sim B_3}}\mc{E}_2(x,z)\prod_{j=1}^pf_2(y_j+z)g(x+y_j)\le 10^{-5}\cdot \eps \cdot \E_{\substack{x\sim B_1\\z\sim B_3}}\prod_{j=1}^pf_2(y_j+z)g(x+y_j).\end{align*}
Let $H_1(x) =\prod_{j=1}^{p}g(x+y_j)$ and $H_2(z) =\prod_{j=1}^{p}f_2(y_j+z)$. By the definition of $\mc{E}_1$ we have that $\E[H_1]\ge (\E[g]\E[f_2])^p$ and $\E[H_2]\ge (\E[g]\E[f_2])^p$. The above conclusion is then equivalent to 
\[\E_{\substack{x\sim B_1\\z\sim B_3}}H_1(x)\mc{E}_2(x,z)H_2(z) \le 10^{-5}\cdot \eps \cdot \E_{\substack{x\sim B_1\\z\sim B_3}}H_1(x)H_2(z).\]

If we now forget the precise details of the construction, we may assume that $H_1:B_1\to \{0,1\}$ and $H_2:B_3\to \{0,1\}$ with $\E[H_1]\ge (\E[g]\E[f_2])^p/2$, $\E[H_2]\ge (\E[g]\E[f_2])^p/2$ and 
\[\E_{\substack{x\sim B_1\\z\sim B_3}}H_1(x)\mc{E}_2(x,z)H_2(z) \le 10^{-5}\cdot \eps \cdot \E_{\substack{x\sim B_1\\z\sim B_3}}H_1(x)H_2(z)\]
via \cref{lem:extract-cor}.

We now observe that 
\[\E_{y\sim B_2}f_2(y+z)g(x+y) = \E_{y\sim B_2}f_2(y)g(x+y-z)+O(\eta d).\]
Let $\mc{E}_3(t) = \mbm{1}[\E_{y\sim B_2}f_2(y)g(t + y)\le (1+9\eps/16) \cdot \E[f_2]\E[g]]$ and note that we have 
\[\E_{\substack{x\sim B_1\\z\sim B_3}}H_1(x)\mc{E}_3(x-z)H_2(z) \le 10^{-5}\cdot \eps \cdot \E_{\substack{x\sim B_1\\z\sim B_3}}H_1(x)H_2(z).\]
This is equivalent to 
\[\E_{\substack{x\sim B_1\\z\sim B_3}}H_1(x)(1-\mc{E}_3(x-z))H_2(z) \ge (1-10^{-5}\cdot \eps) \cdot \E_{\substack{x\sim B_1\\z\sim B_3}}H_1(x)H_2(z).\]
Observe that if $\E_{x\sim B_1}(1-\mc{E}_3(x))\le \E[H_1(x)]/2$ then 
\[\E_{\substack{x\sim B_1\\z\sim B_3}}H_1(x)(1-\mc{E}_3(x-z))H_2(z)\le \E_{\substack{x\sim B_1\\z\sim B_3}}(1-\mc{E}_3(x-z))H_2(z)\le 3/4 \cdot \E_{\substack{x\sim B_1\\z\sim B_3}}H_1(x)H_2(z)\]
which is a contradiction.

We now apply \cref{thm:AP-input}. We first suppose that $\E_{x\sim B_1}(1-\mc{E}_3(x))\ge \E[H_1(x)]$. Then $(1-\mc{E}_3)$ will be $\mbm{1}_D$, $H_1$ will be $\mbm{1}_Y$, and $H_2$ will be $\mbm{1}_Z$. There exists a Bohr set $B'$ of dimension bounded by $d + d'$ with 
\[d'\ll \eps^{-4} p^{4} \cdot \log(1/(\E[f_2]\E[g]))^{4} \ll \eps^{-8}(\log(1/\eps))^4 \cdot k^{8}\]
and radius $r'\gg r_3 \cdot \eps \cdot (\E[f_2]\E[g])^{-O(k\log(1/\eps)/\eps)}/d^4$ such that 
\[\E_{\substack{x\sim B_1\\z\sim B_3\\t\sim B'}}H_1(x)(1-\mc{E}_3(x-z+t)H_2(z) \ge (1-10^{-4}\cdot \eps) \cdot \E_{\substack{x\sim B_1\\z\sim B_3}}H_1(x)H_2(z).\]
Else we have that $\E_{x\sim B_1}(1-\mc{E}_3(x))\in [\E[H_1(x)]/2,\E[H_1(x)]]$ and observe that 
\[\E_{\substack{x\sim B_1\\z\sim B_3}}H_1(x+z)(1-\mc{E}_3(x))H_2(z) \ge (1-2\cdot 10^{-5}\cdot \eps) \cdot \E_{\substack{x\sim B_1\\z\sim B_3}}H_1(x)H_2(z).\]
Thus (taking $\mbm{1}_Z = H_2$, $\mbm{1}_D = H_1$ and $\mbm{1}_Y(y) = (1-\mc{E}_3)(-y)$) there exists $B'$ with the same properties as earlier such that 
\[\E_{\substack{x\sim B_1\\z\sim B_3\\t\sim B'}}H_1(x+z+t)(1-\mc{E}_3(x))H_2(z) \ge (1-5\cdot 10^{-5}\cdot \eps) \cdot \E_{\substack{x\sim B_1\\z\sim B_3}}H_1(x)H_2(z).\]
Via a change of variable we get 
\begin{align*}
    \E_{\substack{x\sim B_1\\z\sim B_3\\t\sim B'}}H_1(x)(1-\mc{E}_3(x-z-t))H_2(z) &= \E_{\substack{x\sim B_1\\z\sim B_3\\t\sim B'}}H_1(x)(1-\mc{E}_3(x-z+t))H_2(z) \\
    &\ge (1-10^{-4}\cdot \eps) \cdot \E_{\substack{x\sim B_1\\z\sim B_3}}H_1(x)H_2(z);
\end{align*}
this is the same conclusion as earlier. 

This immediately gives the desired conclusion modulo unwinding definitions. Observe that we have that 
\[\E_{\substack{x\sim B_1\\z\sim B_3\\t\sim B'}}H_1(x)(1-\mc{E}_3(x-z+t))H_2(z) \ge (1-10^{-4}\cdot \eps) \cdot \E_{\substack{x\sim B_1\\z\sim B_3}}H_1(x)H_2(z).\]
Via the definition of $\mc{E}_3$, we obtain that 
\[\E_{\substack{x\sim B_1\\y\sim B_2\\z\sim B_3}}H_1(x)H_2(z)f_2(y)\Big(\E_{t\sim B'}g(x-z+y + t)\Big)\ge (1+17\eps/32) \cdot \E[H_1] \cdot \E[H_2]\cdot \E[f_2] \cdot \E[g].\]
Therefore there exists $x\in B_1, y\in B_2, z\in B_3$ such that 
\[\Big(\E_{t\sim B'}g(x-z+y + t)\Big)\ge (1+17\eps/32) \cdot \E[g]\]
as desired.
\end{proof}

\bibliographystyle{amsplain0}
\bibliography{main.bib}

\providecommand{\bysame}{\leavevmode\hbox to3em{\hrulefill}\thinspace}
\providecommand{\MR}{\relax\ifhmode\unskip\space\fi MR }
\providecommand{\MRhref}[2]{%
  \href{http://www.ams.org/mathscinet-getitem?mr=#1}{#2}
}
\providecommand{\href}[2]{#2}
\begin{thebibliography}{10}

\bibitem{ASz74}
M.~Ajtai and E.~Szemer\'edi, \emph{Sets of lattice points that form no
  squares}, Studia Sci. Math. Hungar. \textbf{9} (1974), 9--11.

\bibitem{ALWZ21}
Ryan Alweiss, Shachar Lovett, Kewen Wu, and Jiapeng Zhang, \emph{Improved
  bounds for the sunflower lemma}, Ann. of Math. (2) \textbf{194} (2021),
  795--815.

\bibitem{BNS89}
L\'aszl\'o Babai, Noam Nisan, and M\'ari\'o Szegedy, \emph{Multiparty protocols
  and logspace-hard pseudorandom sequences}, Proceedings of the {T}wenty-first
  {A}nnual {ACM} {S}ymposium on {T}heory of {C}omputing, {STOC} 1989, Seattle,
  Washington, USA, May 15--17, 1989, Association for Computing Machinery (ACM),
  New York, 1989, p.~1–11.

\bibitem{BDPW10}
Paul Beame, Matei David, Toniann Pitassi, and Philipp Woelfel, \emph{Separating
  deterministic from randomized multiparty communication complexity}, Theory
  Comput. \textbf{6} (2010), 201--225.

\bibitem{BH09}
Paul Beame and Dang-Trinh Huynh-Ngoc, \emph{Multiparty communication complexity
  and threshold circuit size of {${\rm AC}^0$}}, 2009 50th {A}nnual {IEEE}
  {S}ymposium on {F}oundations of {C}omputer {S}cience---{FOCS} 2009, IEEE
  Computer Soc., Los Alamitos, CA, 2009, pp.~53--62.

\bibitem{Beh46}
F.~A. Behrend, \emph{On sets of integers which contain no three terms in
  arithmetical progression}, Proc. Nat. Acad. Sci. U.S.A. \textbf{32} (1946),
  331--332.

\bibitem{BMZ97}
Vitaly Bergelson, Randall McCutcheon, and Qing Zhang, \emph{A {R}oth theorem
  for amenable groups}, Amer. J. Math. \textbf{119} (1997), 1173--1211.

\bibitem{Bloom16}
T.~F. Bloom, \emph{A quantitative improvement for {R}oth's theorem on
  arithmetic progressions}, J. Lond. Math. Soc. (2) \textbf{93} (2016),
  643--663.

\bibitem{BS23b}
Thomas~F Bloom and Olof Sisask, \emph{An improvement to the {K}elley-{M}eka
  bounds on three-term arithmetic progressions}, arXiv:2309.02353.

\bibitem{BS21}
Thomas~F. Bloom and Olof Sisask, \emph{Breaking the logarithmic barrier in
  {R}oth's theorem on arithmetic progressions}, arXiv preprint arXiv:2007.03528
  (2021).

\bibitem{BS23}
Thomas~F. Bloom and Olof Sisask, \emph{The {K}elley-{M}eka bounds for sets free
  of three-term arithmetic progressions}, Essent. Number Theory \textbf{2}
  (2023), 15--44.

\bibitem{Bou99}
J.~Bourgain, \emph{On triples in arithmetic progression}, Geom. Funct. Anal.
  \textbf{9} (1999), 968--984.

\bibitem{Bou08}
Jean Bourgain, \emph{Roth's theorem on progressions revisited}, J. Anal. Math.
  \textbf{104} (2008), 155--192.

\bibitem{CFL83}
Ashok~K. Chandra, Merrick~L. Furst, and Richard~J. Lipton, \emph{Multi-party
  protocols}, Proceedings of the 15th Annual {ACM} Symposium on Theory of
  Computing, 25-27 April, 1983, Boston, Massachusetts, {USA} (David~S. Johnson,
  Ronald Fagin, Michael~L. Fredman, David Harel, Richard~M. Karp, Nancy~A.
  Lynch, Christos~H. Papadimitriou, Ronald~L. Rivest, Walter~L. Ruzzo, and
  Joel~I. Seiferas, eds.), {ACM}, 1983, pp.~94--99.

\bibitem{CFTZ22}
Matthias Christandl, Omar Fawzi, Hoang Ta, and Jeroen Zuiddam, \emph{Larger
  corner-free sets from combinatorial degenerations}, 13th {I}nnovations in
  {T}heoretical {C}omputer {S}cience {C}onference, LIPIcs. Leibniz Int. Proc.
  Inform., vol. 215, Schloss Dagstuhl. Leibniz-Zent. Inform., Wadern, 2022,
  pp.~Art. No. 48, 20.

\bibitem{CFZ14}
David Conlon, Jacob Fox, and Yufei Zhao, \emph{Extremal results in sparse
  pseudorandom graphs}, Adv. Math. \textbf{256} (2014), 206--290.

\bibitem{CFZ15}
David Conlon, Jacob Fox, and Yufei Zhao, \emph{A relative {S}zemer\'edi
  theorem}, Geom. Funct. Anal. \textbf{25} (2015), 733--762.

\bibitem{CLP17}
Ernie Croot, Vsevolod~F. Lev, and P\'eter~P\'al Pach, \emph{Progression-free
  sets in {$\Bbb Z^n_4$} are exponentially small}, Ann. of Math. (2)
  \textbf{185} (2017), 331--337.

\bibitem{EG17}
Jordan~S. Ellenberg and Dion Gijswijt, \emph{On large subsets of {$\Bbb F^n_q$}
  with no three-term arithmetic progression}, Ann. of Math. (2) \textbf{185}
  (2017), 339--343.

\bibitem{EHPS24}
Christian Elsholtz, Zach Hunter, Laura Proske, and Lisa Sauermann,
  \emph{Improving {B}ehrend's construction: {S}ets without arithmetic
  progressions in integers and over finite fields}.

\bibitem{FHHK24}
Yuval Filmus, Hamed Hatami, Kaave Hosseini, and Esty Kelman, \emph{Sparse graph
  counting and {K}elley-{M}eka bounds for binary systems}, 2024 IEEE 65th
  Annual Symposium on Foundations of Computer Science (FOCS), IEEE, 2024,
  pp.~1559--1578.

\bibitem{FK78}
H.~Furstenberg and Y.~Katznelson, \emph{An ergodic {S}zemer\'edi theorem for
  commuting transformations}, J. Analyse Math. \textbf{34} (1978), 275--291.

\bibitem{Fur77}
Harry Furstenberg, \emph{Ergodic behavior of diagonal measures and a theorem of
  {S}zemer\'edi on arithmetic progressions}, J. Analyse Math. \textbf{31}
  (1977), 204--256.

\bibitem{Gow98}
W.~T. Gowers, \emph{Fourier analysis and {S}zemer\'edi's theorem}, Proceedings
  of the {I}nternational {C}ongress of {M}athematicians, {V}ol. {I} ({B}erlin,
  1998), 1998, pp.~617--629.

\bibitem{Gow98a}
W.~T. Gowers, \emph{A new proof of {S}zemer\'edi's theorem for arithmetic
  progressions of length four}, Geom. Funct. Anal. \textbf{8} (1998), 529--551.

\bibitem{Gow01}
W.~T. Gowers, \emph{Arithmetic progressions in sparse sets}, Current
  developments in mathematics, 2000, Int. Press, Somerville, MA, 2001,
  pp.~149--196.

\bibitem{Gow01a}
W.~T. Gowers, \emph{A new proof of {S}zemer\'edi's theorem}, Geom. Funct. Anal.
  \textbf{11} (2001), 465--588.

\bibitem{Gow07}
W.~T. Gowers, \emph{Hypergraph regularity and the multidimensional
  {S}zemer\'edi theorem}, Ann. of Math. (2) \textbf{166} (2007), 897--946.

\bibitem{GS06}
Ron Graham and J\'ozsef Solymosi, \emph{Monochromatic equilateral right
  triangles on the integer grid}, Topics in discrete mathematics, Algorithms
  Combin., vol.~26, Springer, Berlin, 2006, pp.~129--132.

\bibitem{Gre05}
Ben Green, \emph{Finite field models in additive combinatorics}, Surveys in
  combinatorics 2005, London Math. Soc. Lecture Note Ser., vol. 327, Cambridge
  Univ. Press, Cambridge, 2005, pp.~1--27.

\bibitem{Gre21}
Ben Green, \emph{Lower bounds for corner-free sets}, New Zealand J. Math.
  \textbf{51} (2021), 1--2.

\bibitem{GT08}
Ben Green and Terence Tao, \emph{The primes contain arbitrarily long arithmetic
  progressions}, Ann. of Math. (2) \textbf{167} (2008), 481--547.

\bibitem{HPSSS24}
Lianna Hambardzumyan, Toniann Pitassi, Suhail Sherif, Morgan Shirley, and Adi
  Shraibman, \emph{An improved protocol for {\sc {e}xactly{n}} with more than 3
  players}, 15th {I}nnovations in {T}heoretical {C}omputer {S}cience
  {C}onference, LIPIcs. Leibniz Int. Proc. Inform., vol. 287, Schloss Dagstuhl.
  Leibniz-Zent. Inform., Wadern, 2024, pp.~Art. No. 58, 23.

\bibitem{Hat10}
Hamed Hatami, \emph{Graph norms and {S}idorenko's conjecture}, Israel J. Math.
  \textbf{175} (2010), 125--150.

\bibitem{HB87}
D.~R. Heath-Brown, \emph{Integer sets containing no arithmetic progressions},
  J. London Math. Soc. (2) \textbf{35} (1987), 385--394.

\bibitem{Hun22}
Zach Hunter, \emph{Corner-free sets via the torus}.

\bibitem{JL024}
Michael Jaber, Shachar Lovett, and Anthony Ostuni, \emph{Strong bounds for skew
  corner-free sets}.

\bibitem{KL23}
Peter Keevash and Noam Lifshitz, \emph{Sharp hypercontractivity for symmetric
  groups and its applications}, arXiv preprint arXiv:2307.15030 (2023),
  Available at \url{https://arxiv.org/pdf/2307.15030}.

\bibitem{kelleytalk}
Zander Kelley, \emph{Problems in extremal combinatorics and connections with
  multiparty communication complexity}, 2024,
  \url{https://www.youtube.com/watch?v=YkHabH4iPVo}.

\bibitem{KLM24}
Zander Kelley, Shachar Lovett, and Raghu Meka, \emph{Explicit separations
  between randomized and deterministic {N}umber-on-{F}orehead communication},
  S{TOC}'24---{P}roceedings of the 56th {A}nnual {ACM} {S}ymposium on {T}heory
  of {C}omputing, ACM, New York, [2024] \copyright 2024, pp.~1299--1310.

\bibitem{KL25}
Zander Kelley and Xin Lyu, \emph{More efficient sifting for grid norms, and
  applications to multiparty communication complexity}, arXiv:2505.01587.

\bibitem{KM23}
Zander Kelley and Raghu Meka, \emph{Strong bounds for 3-progressions}, 2023
  {IEEE} 64th {A}nnual {S}ymposium on {F}oundations of {C}omputer
  {S}cience---{FOCS} 2023, IEEE Computer Soc., Los Alamitos, CA, [2023]
  \copyright 2023, pp.~933--973.

\bibitem{KZ24}
Andrey Kupavskii and Dmitrii Zakharov, \emph{Spread approximations for
  forbidden intersections problems}, Adv. Math. \textbf{445} (2024), Paper No.
  109653, 29.

\bibitem{LM07}
Michael~T. Lacey and William McClain, \emph{On an argument of {S}hkredov on
  two-dimensional corners}, Online J. Anal. Comb. (2007), Art. 2, 21.

\bibitem{LS21}
Nati Linial and Adi Shraibman, \emph{Larger corner-free sets from better {NOF}
  exactly-{$N$} protocols}, Discrete Anal. (2021), Paper No. 19, 9.

\bibitem{mekatalk}
Raghu Meka, \emph{Strong bounds for 3-progressions}, 2023, Talk as part the
  Breakthroughs lecture series at the Simons Institute for the Theory of
  Computing. \url{https://www.youtube.com/watch?v=WN7rJPWy6z8}.

\bibitem{Mil24}
Luka Mili{\'c}evi{\'c}, \emph{Good bounds for sets lacking skew corners},
  arXiv:2404.07180.

\bibitem{NRS06}
Brendan Nagle, Vojt{\v{e}}ch R\"odl, and Mathias Schacht, \emph{The counting
  lemma for regular {$k$}-uniform hypergraphs}, Random Structures Algorithms
  \textbf{28} (2006), 113--179.

\bibitem{NW93}
Noam Nisan and Avi Wigderson, \emph{Rounds in communication complexity
  revisited}, SIAM J. Comput. \textbf{22} (1993), 211--219.

\bibitem{Pel24b}
Sarah Peluse, \emph{Finite field models in arithmetic combinatorics---twenty
  years on}, Surveys in combinatorics 2024, London Math. Soc. Lecture Note
  Ser., vol. 493, Cambridge Univ. Press, Cambridge, 2024, pp.~159--199.

\bibitem{Pel24}
Sarah Peluse, \emph{Subsets of {$\Bbb F^n_p \times \Bbb F^n_p$} without {$\rm
  L$}-shaped configurations}, Compos. Math. \textbf{160} (2024), 176--236.

\bibitem{Pol12}
D.~H.~J. Polymath, \emph{A new proof of the density {H}ales-{J}ewett theorem},
  Ann. of Math. (2) \textbf{175} (2012), 1283--1327.

\bibitem{Pyber97}
L.~Pyber, \emph{How abelian is a finite group?}, The mathematics of {P}aul
  {E}rd\H os, {I}, Algorithms Combin., vol.~13, Springer, Berlin, 1997,
  pp.~372--384.

\bibitem{Raz00}
Ran Raz, \emph{The {BNS}-{C}hung criterion for multi-party communication
  complexity}, Comput. Complexity \textbf{9} (2000), 113--122.

\bibitem{RS04}
Vojt{\v{e}}ch R\"odl and Jozef Skokan, \emph{Regularity lemma for {$k$}-uniform
  hypergraphs}, Random Structures Algorithms \textbf{25} (2004), 1--42.

\bibitem{Roth53}
K.~F. Roth, \emph{On certain sets of integers}, J. London Math. Soc.
  \textbf{28} (1953), 104--109.

\bibitem{San11}
Tom Sanders, \emph{On {R}oth's theorem on progressions}, Ann. of Math. (2)
  \textbf{174} (2011), 619--636.

\bibitem{San12}
Tom Sanders, \emph{On certain other sets of integers}, J. Anal. Math.
  \textbf{116} (2012), 53--82.

\bibitem{Sanders17}
Tom Sanders, \emph{Solving {$xz=y^2$} in certain subsets of finite groups}, Q.
  J. Math. \textbf{68} (2017), 243--273.

\bibitem{SS16}
Tomasz Schoen and Olof Sisask, \emph{Roth's theorem for four variables and
  additive structures in sums of sparse sets}, Forum Math. Sigma \textbf{4}
  (2016), Paper No. e5, 28.

\bibitem{Shk05}
I.~D. Shkredov, \emph{On a problem of {G}owers}, Dokl. Akad. Nauk \textbf{400}
  (2005), 169--172.

\bibitem{Shk06}
I.~D. Shkredov, \emph{On a generalization of {S}zemer\'edi's theorem}, Proc.
  London Math. Soc. (3) \textbf{93} (2006), 723--760.

\bibitem{Sol03}
J\'ozsef Solymosi, \emph{Note on a generalization of {R}oth's theorem},
  Discrete and computational geometry, Algorithms Combin., vol.~25, Springer,
  Berlin, 2003, pp.~825--827.

\bibitem{Sze70}
E.~Szemer\'{e}di, \emph{On sets of integers containing no four elements in
  arithmetic progression}, Number {T}heory ({C}olloq., {J}\'{a}nos {B}olyai
  {M}ath. {S}oc., {D}ebrecen, 1968), Colloq. Math. Soc. J\'{a}nos Bolyai,
  vol.~2, North-Holland, Amsterdam-London, 1970, pp.~197--204.

\bibitem{Sze75}
E.~Szemer\'{e}di, \emph{On sets of integers containing no {$k$} elements in
  arithmetic progression}, Acta Arith. \textbf{27} (1975), 199--245.

\bibitem{SZ90}
E.~Szemer\'edi, \emph{Integer sets containing no arithmetic progressions}, Acta
  Math. Hungar. \textbf{56} (1990), 155--158.

\bibitem{Tao06}
Terence Tao, \emph{A variant of the hypergraph removal lemma}, J. Combin.
  Theory Ser. A \textbf{113} (2006), 1257--1280.

\bibitem{TV10}
Terence Tao and Van~H. Vu, \emph{Additive combinatorics}, paperback ed.,
  Cambridge Studies in Advanced Mathematics, vol. 105, Cambridge University
  Press, Cambridge, 2010.

\bibitem{Zha23}
Yufei Zhao, \emph{Graph theory and additive combinatorics---exploring structure
  and randomness}, Cambridge University Press, Cambridge, 2023.

\end{thebibliography}
\end{document}